\documentclass[11pt]{amsart}

\usepackage{anysize} \marginsize{1.3in}{1.3in}{1in}{1in}
\usepackage{comment}
\usepackage{listings}
\usepackage{xcolor}
\usepackage{amsmath}
\usepackage{mathtools}
\usepackage{booktabs}
\usepackage[all]{xy}
\usepackage[utf8]{inputenc}
\usepackage{varioref}
\usepackage{upgreek}
\usepackage{amsfonts}
\usepackage{amssymb}
\usepackage{bbm}
\usepackage{esint}
\usepackage{graphicx}
\usepackage{subcaption}
\usepackage{tikz}
\usepackage{empheq}
\usepackage{enumitem}
\usepackage{tikz-cd}
\usepackage{todonotes}
\usetikzlibrary{matrix,arrows,decorations.pathmorphing}
\usepackage{mathrsfs}
\usepackage[hypertexnames=false,backref=page,pdftex,
 	pdfpagemode=UseNone,
 	breaklinks=true,
 	extension=pdf,
 	colorlinks=true,
 	linkcolor=blue,
 	citecolor=red,
 	urlcolor=blue,
 ]{hyperref}

\newcommand{\HH}{{\mathbb H}}

\newcommand{\tr}{\operatorname{tr}}

\newcommand{\Hom}{\operatorname{Hom}}

\newcommand{\diag}{\operatorname{diag}}

\newcommand{\sgn}{\operatorname{sgn}}

\newcommand{\str}{\operatorname{str}}

\newcommand{\Ind}{\operatorname{Ind}}
\newcommand{\id}{\operatorname{id}}

\renewcommand{\Im}{\operatorname{im}}
\newcommand{\Ker}{\operatorname{ker}}
\newcommand{\Ext}{\operatorname{Ext}}
\newcommand{\End}{\operatorname{End}}
\newcommand{\ad}{\operatorname{ad}}
\newcommand{\ch}{\operatorname{ch}_{\kk^{\CC}}}
\newcommand{\fch}{\operatorname{ch}}
\newcommand{\pr}{\operatorname{pr}}
\newcommand{\bil}{B}
\newcommand{\Dirac}{\operatorname{D}}
\newcommand{\Diracv}{\operatorname{D}\! v}
\newcommand{\Diracw}{\operatorname{D}\! w}
\newcommand{\DC}{\operatorname{H}_{\operatorname{{D}}}}
\newcommand{\DI}{\operatorname{I}}

\def\WW{{\mathcal{W}}}
\def\osp{{\mathfrak{osp}}}
\def\dd{{\mathfrak{d}}}
\def\Dp{{\Dirac^{\pp_{1}}}}
\def\Dq{{\Dirac^{\qq_{2}}}}
\def\ddp{{\mathrm{d}^{\pp_{1}}}}
\def\ddq{{\mathrm{d}^{\qq_{2}}}}
\def\delp{{\delta^{\pp_{1}}}}
\def\delq{{\delta^{\qq_{2}}}}

\def\CC{{\mathbb C}}
\def\RR{{\mathbb R}}
\def\su{{\mathfrak{su}}}

\def\Mat{{\text{Mat}}}
\def\HH{{\mathcal{H}}}

\def\hh{{\mathfrak h}}
\def\pp{{\mathfrak p}}
\def\uu{{\mathfrak u}}
\def\kk{{\mathfrak k}}
\def\u{{\mathfrak u}}
\def\ll{{\mathfrak l}}

\def\qq{{\mathfrak q}}

\def\bb{{\mathfrak b}}
\def\even{{\mathfrak{g}_{\bar{0}}}}
\def\odd{{\mathfrak{g}_{\bar{1}}}}
\def\pp{{\mathfrak p}}

\def\reg{{reg}}
\def\rform{\gg_{\bar{0}}^{\RR}}

\def\nn {{\mathfrak{n}}}
\def\pp{{\mathfrak{p}}}
\def\gg{{\mathfrak{g}}}

\def\ZZ{{\mathbb Z}}

\def\Cl{{\operatorname{Cl}}}
\def\gsmod{{\mathfrak{g}}\textbf{-smod}}

\def\gmod{{\mathfrak{g}_{\bar{0}}}\textbf{-mod}}
\def\evsmod{{\mathfrak{g}_{\bar{0}}}\textbf{-smod}}
\def\Weyl{{\mathscr{W}(\odd)}}
\def\slmn{{\mathfrak{sl}(m\vert n)}}
 
\def\glmn{{\mathfrak{gl}(m\vert n)}} 
\def\UE{{\mathfrak U}}
\def\Dis{{\mathscr{D}^{\text{hol}}}}
\def\sl{{\mathfrak{sl}}}

\newcommand{\gl}{{\mathfrak{gl}}}

\makeatletter
\newcommand*{\da@rightarrow}{\mathchar"0\hexnumber@\symAMSa 4B }
\newcommand*{\da@leftarrow}{\mathchar"0\hexnumber@\symAMSa 4C }
\newcommand*{\xdashrightarrow}[2][]{%
  \mathrel{%
    \mathpalette{\da@xarrow{#1}{#2}{}\da@rightarrow{\,}{}}{}%
  }%
}
\newcommand{\xdashleftarrow}[2][]{%
  \mathrel{%
    \mathpalette{\da@xarrow{#1}{#2}\da@leftarrow{}{}{\,}}{}%
  }%
}
\newcommand*{\da@xarrow}[7]{%
  \sbox0{$\ifx#7\scriptstyle\scriptscriptstyle\else\scriptstyle\fi#5#1#6\m@th$}%
  \sbox2{$\ifx#7\scriptstyle\scriptscriptstyle\else\scriptstyle\fi#5#2#6\m@th$}%
  \sbox4{$#7\dabar@\m@th$}%
  \dimen@=\wd0 %
  \ifdim\wd2 >\dimen@
    \dimen@=\wd2 %
  \fi
  \count@=2 %
  \def\da@bars{\dabar@\dabar@}%
  \@whiledim\count@\wd4<\dimen@\do{%
    \advance\count@\@ne
    \expandafter\def\expandafter\da@bars\expandafter{%
      \da@bars
      \dabar@ 
    }%
  }%
  \mathrel{#3}%
  \mathrel{%
    \mathop{\da@bars}\limits
    \ifx\\#1\\%
    \else
      _{\copy0}%
    \fi
    \ifx\\#2\\%
    \else
      ^{\copy2}%
    \fi
  }%
  \mathrel{#4}%
}
\makeatother

\makeatletter
\newsavebox\myboxA
\newsavebox\myboxB
\newlength\mylenA

\newcommand*\xtilde[2][0.8]{%
    \sbox{\myboxA}{$\m@th#2$}%
    \setbox\myboxB\null
    \ht\myboxB=\ht\myboxA%
    \dp\myboxB=\dp\myboxA%
    \wd\myboxB=#1\wd\myboxA
    \sbox\myboxB{$\m@th\widetilde{\copy\myboxB}$}
    \setlength\mylenA{\the\wd\myboxA}
    \addtolength\mylenA{-\the\wd\myboxB}%
    \ifdim\wd\myboxB<\wd\myboxA%
       \rlap{\hskip 0.5\mylenA\usebox\myboxB}{\usebox\myboxA}%
    \else
        \hskip -0.5\mylenA\rlap{\usebox\myboxA}{\hskip 0.5\mylenA\usebox\myboxB}%
    \fi}

\newbox\usefulbox

\def\getslant #1{\strip@pt\fontdimen1 #1}

\def\xxtilde #1{\mathchoice
 {{\setbox\usefulbox=\hbox{$\m@th\displaystyle #1$}%
    \dimen@ \getslant\the\textfont\symletters \ht\usefulbox
    \divide\dimen@ \tw@ 
    \kern\dimen@ 
    \xtilde{\kern-\dimen@ \box\usefulbox\kern\dimen@ }\kern-\dimen@ }}
 {{\setbox\usefulbox=\hbox{$\m@th\textstyle #1$}%
    \dimen@ \getslant\the\textfont\symletters \ht\usefulbox
    \divide\dimen@ \tw@ 
    \kern\dimen@ 
    \xtilde{\kern-\dimen@ \box\usefulbox\kern\dimen@ }\kern-\dimen@ }}
 {{\setbox\usefulbox=\hbox{$\m@th\scriptstyle #1$}%
    \dimen@ \getslant\the\scriptfont\symletters \ht\usefulbox
    \divide\dimen@ \tw@ 
    \kern\dimen@ 
    \xtilde{\kern-\dimen@ \box\usefulbox\kern\dimen@ }\kern-\dimen@ }}
 {{\setbox\usefulbox=\hbox{$\m@th\scriptscriptstyle #1$}%
    \dimen@ \getslant\the\scriptscriptfont\symletters \ht\usefulbox
    \divide\dimen@ \tw@ 
    \kern\dimen@ 
    \xtilde{\kern-\dimen@ \box\usefulbox\kern\dimen@ }\kern-\dimen@ }}%
 {}}

\newcommand*\xoverline[2][0.75]{%
    \sbox{\myboxA}{$\m@th#2$}%
    \setbox\myboxB\null
    \ht\myboxB=\ht\myboxA%
    \dp\myboxB=\dp\myboxA%
    \wd\myboxB=#1\wd\myboxA
    \sbox\myboxB{$\m@th\overline{\copy\myboxB}$}
    \setlength\mylenA{\the\wd\myboxA}
    \addtolength\mylenA{-\the\wd\myboxB}%
    \ifdim\wd\myboxB<\wd\myboxA%
       \rlap{\hskip 0.5\mylenA\usebox\myboxB}{\usebox\myboxA}%
    \else
        \hskip -0.5\mylenA\rlap{\usebox\myboxA}{\hskip 0.5\mylenA\usebox\myboxB}%
    \fi}

\def\xxoverline #1{\mathchoice
 {{\setbox\usefulbox=\hbox{$\m@th\displaystyle #1$}%
    \dimen@ \getslant\the\textfont\symletters \ht\usefulbox
    \divide\dimen@ \tw@ 
    \kern\dimen@ 
    \overline{\kern-\dimen@ \box\usefulbox\kern\dimen@ }\kern-\dimen@ }}
 {{\setbox\usefulbox=\hbox{$\m@th\textstyle #1$}%
    \dimen@ \getslant\the\textfont\symletters \ht\usefulbox
    \divide\dimen@ \tw@ 
    \kern\dimen@ 
    \xoverline{\kern-\dimen@ \box\usefulbox\kern\dimen@ }\kern-\dimen@ }}
 {{\setbox\usefulbox=\hbox{$\m@th\scriptstyle #1$}%
    \dimen@ \getslant\the\scriptfont\symletters \ht\usefulbox
    \divide\dimen@ \tw@ 
    \kern\dimen@ 
    \xoverline{\kern-\dimen@ \box\usefulbox\kern\dimen@ }\kern-\dimen@ }}
 {{\setbox\usefulbox=\hbox{$\m@th\scriptscriptstyle #1$}%
    \dimen@ \getslant\the\scriptscriptfont\symletters \ht\usefulbox
    \divide\dimen@ \tw@ 
    \kern\dimen@ 
    \xoverline{\kern-\dimen@ \box\usefulbox\kern\dimen@ }\kern-\dimen@ }}%
 {}}
\makeatother

\makeatletter
\newcommand{\mylabel}[2]{#2\def\@currentlabel{#2}\label{#1}}
\makeatother

\makeatletter
\newcommand{\Mac}{}
\DeclareRobustCommand{\Mac}{%
  M%
  \raisebox{\dimexpr\fontcharht\font`M-\height}{%
    \check@mathfonts\fontsize{\sf@size}{0}\selectfont
    c%
  }%
}
\makeatother

\newtheoremstyle{citing}
  {}
  {}
  {\itshape}
  {}
  {\bfseries}
  {\textbf{.}}
  {.5em}
  {\thmnote{#3}}

\theoremstyle{plain}
\newtheorem{theorem}{Theorem}

\newtheorem{lemma}[theorem]{Lemma}
\newtheorem{corollary}[theorem]{Corollary}

\newtheorem{proposition}[theorem]{Proposition}

\theoremstyle{remark}
\newtheorem{example}[theorem]{Example}

\theoremstyle{definition}

\newtheorem{definition}[theorem]{Definition}

\numberwithin{equation}{section}

\theoremstyle{remark}
\newtheorem{remark}[theorem]{Remark}

{\theoremstyle{citing}
}

{\theoremstyle{definition}
}

\title{Dirac Operators, Dirac Cohomology and Unitarity for $A(m\vert n)$}

\author{Steffen Schmidt}
\address{Center for Quantum Mathematics, 
University of Southern Denmark, DK-5230 Odense, Denmark}
\email{stschmidt@imada.sdu.dk}

\let\origmaketitle\maketitle
\def\maketitle{
  \begingroup
  \def\uppercasenonmath##1{} 
  \let\MakeUppercase\relax 
  \origmaketitle
  \endgroup
}

\begin{document}
\thispagestyle{empty}

\begin{abstract}
Dirac operators and Dirac cohomology for Lie superalgebras of Riemannian type, introduced by Huang and Pand\v{z}i\'{c}, provide an effective tool for the study of unitarizable supermodules. In this article, we study these objects for Lie superalgebras of type $A$ and relate them systematically to unitarity. In the first part, we establish the basic structure of the theory in this setting. We relate unitarity to the Dirac operator, derive the corresponding Dirac inequality, and show that Dirac cohomology determines unitarizable supermodules. We also determine explicitly the Dirac cohomology of unitarizable simple supermodules. In the second part, we turn to applications. We obtain a new characterization of unitarity, establish a relation with Kostant's cohomology, derive a formula for formal characters, and introduce a Dirac index.
\end{abstract}

\maketitle

\setlength{\parindent}{1em}
\setcounter{tocdepth}{1}

\tableofcontents


\section{Introduction}
\subsection{Vue d'ensemble} Dirac operators entered representation theory through the work of Parthasarathy and Vogan. Parthasarathy used them in the construction of discrete series representations of semisimple Lie groups \cite{atiyah1977geometric,parthasarathy1972dirac}, while Vogan introduced Dirac cohomology for $(\gg,K)$-modules and conjectured that it determines the infinitesimal character. This conjecture was proved by Huang and Pand\v{z}i\'{c} \cite{huang2002dirac}.

Since then, Dirac cohomology has become an important invariant in representation theory, with strong links to other cohomological theories. For highest weight modules it is closely related to $\mathfrak n$-cohomology, and for Vogan--Zuckerman modules $A_{\mathfrak q}(\lambda)$ it is closely related to $(\gg,K)$-cohomology \cite{Dirac_cohomology_Lie_algebra_cohomology,Classification_Dirac_unitarity,Dirac_cohomology_of_A_q}. At the same time, it is often computationally more accessible than $(\gg,K)$-cohomology while retaining substantial information about the representation.

In \cite{huang2007dirac}, Huang and Pand\v{z}i\'{c} extended the theory of Dirac operators and Dirac cohomology to Lie superalgebras of Riemannian type. Their construction requires a non-degenerate invariant supersymmetric bilinear form $B$ on $\gg$ and a decomposition of $\odd$ into complementary Lagrangian subspaces. Under these assumptions, one obtains a natural Dirac operator $\Dirac$ on tensor products of $\gg$-supermodules with the oscillator module $M(\odd)$ for the Weyl algebra of $\odd$, whose square is the sum of two quadratic Casimir operators and a constant. The associated Dirac cohomology of a $\gg$-supermodule $M$ is defined by
\begin{equation}
\DC(M)\coloneqq \ker\Dirac/(\ker\Dirac\cap\operatorname{Im}\Dirac).
\end{equation}
Huang and Pand\v{z}i\'{c} proved a superalgebra analogue of Vogan's conjecture: if $\DC(M)\neq 0$, then $\DC(M)$ determines the infinitesimal character of $M$ \cite{huang2005dirac}.

An important link with other cohomological invariants was established in \cite{cheng2004analogue}, where Dirac cohomology was related to Kostant's $\gg_{+1}$-cohomology for finite-dimensional $\gg$-supermodules, assuming that $\gg$ admits a $\mathbb Z$-grading compatible with the $\mathbb Z_{2}$-grading. For unitarizable $\gl(m\vert n)$-supermodules, Xiao later showed that Dirac cohomology is, up to a twist, isomorphic to $\mathrm H^{*}(\gg_{+1},M)$ \cite{xiao2015dirac}.

Explicit computations of Dirac cohomology in the super setting are still missing, and its relation to unitarity has not been investigated systematically. In this article, we address both problems for Lie superalgebras of type $A$, denoted by $\gg=A(m-1\vert n-1)$: we compute Dirac cohomology explicitly and relate it to unitarity.

\subsection{Results}
The results are divided into two parts. The first develops the structural relations between unitarity, the Dirac operator, and Dirac cohomology. The second gives explicit applications, including a new characterization of unitarity, relations with Kostant's cohomology, the construction of a Dirac index and explicit formulas for the formal character.

We begin with the first part. We realize the Dirac operator $\Dirac$ of Huang and Pand\v{z}i\'{c} as a distinguished element of the algebra of $\even$-invariants $\mathcal{W}(\gg,\even)$ in the quantum Weil algebra
$
\WW(\gg)\coloneqq \UE(\gg)\otimes \Cl(\gg).
$
This construction extends to arbitrary quadratic Lie superalgebras and provides a general framework for the study of $\Dirac$. In particular, viewed as an element of $\WW(\gg,\even)=\UE(\gg)\otimes\Cl(\odd)$, the Dirac operator $\Dirac$ acts naturally on $M\otimes M(\odd)$, where $M$ is a $\gg$-supermodule and $M(\odd)$ denotes the simple $\Cl(\odd)$-module, that is, the oscillator module.

In this article, we are primarily interested in unitarizable $\gg$-supermodules, that is, $\gg$-supermodules $M$ admitting a positive definite Hermitian form with respect to which the action of $\gg$ is contravariant for some conjugate-linear anti-involution $\omega$. Such anti-involutions are in bijection with real forms of $\gg$. For Lie superalgebras of type $A$, non-trivial simple unitarizable supermodules exist only for a restricted class of real forms, and in the cases considered here they are necessarily highest weight supermodules. 

To relate unitarity to the Dirac operator, one must work with a basis adapted not to a $\ZZ_{2}$-compatible $\ZZ$-grading, but to the real form defining the unitary structure. More precisely, for every real form admitting non-trivial unitarizable supermodules, we construct a basis of $\gg$ compatible both with the even non-degenerate supersymmetric invariant bilinear form $B$ and with the conjugate-linear anti-involution $\omega$. This basis is naturally adapted to the corresponding maximal compact subalgebra. In terms of it, the Hermitian structure on a supermodule and the Dirac operator become directly comparable, and one obtains the following characterization of contravariance (\emph{cf.}~Theorem~\ref{DiracUnit}).

\begin{theorem}
Let $M$ be a simple $\gg$-supermodule equipped with a positive definite Hermitian form $\langle \cdot, \cdot \rangle_M$, such that $M_{\bar{0}}$ and $M_{\bar{1}}$ are mutually orthogonal. Assume that $M$ is $\even$-semisimple. Then the following are equivalent:
\begin{enumerate}
 \item[(i)] $(M, \langle \cdot, \cdot \rangle_M)$ is a unitarizable $\gg$-supermodule.
 \item[(ii)] The Dirac operator $\Dirac$ is selfadjoint with respect to $\langle \cdot, \cdot \rangle_{M \otimes M(\odd)}$.
\end{enumerate}
\end{theorem}

As a direct consequence of selfadjointness, one obtains a Dirac inequality, which becomes a basic tool in the study of unitarizable supermodules. The following combines Proposition~\ref{prop::Dirac_inequality} and Proposition~\ref{prop::specified_Dirac_inequality} from the main text.

\begin{proposition}
Let $\HH$ be a unitarizable $\gg$-supermodule. Then
\[
\langle \Dirac^{2}v,v\rangle_{\HH\otimes M(\odd)}\geq 0
\qquad\text{for all }v\in \HH\otimes M(\odd).
\]
In particular, if $\HH$ is simple of highest weight $\Lambda$, and if $V$ is a $\even$-highest weight submodule of $\HH\otimes M(\odd)$ with highest weight $\mu$, then
\[
(\mu+2\rho,\mu)\geq (\Lambda+2\rho,\Lambda).
\]
\end{proposition}

Turning to Dirac cohomology, one observes that, for general supermodules, it does not admit an adjoint functor, although it satisfies a six-term exact sequence. 
A modified Dirac cohomology is therefore introduced; it admits a right adjoint and agrees with the usual Dirac cohomology on unitarizable supermodules.

For unitarizable supermodules, the Dirac cohomology simplifies to
$
\DC(M)=\Ker\Dirac.
$
Moreover, if $M$ is a simple unitarizable $\gg$-supermodule, then $\DC(M)$ is a unitarizable $\even$-module and decomposes as a direct sum of simple unitarizable $\even$-modules (neglecting parity). We show that these $\even$-modules are precisely those arising from the relative holomorphic discrete series of the universal covering group associated with $\even$. Combined with the superalgebra analogue of Vogan's conjecture, this leads to an explicit determination of the Dirac cohomology of all unitarizable $\gg$-supermodules. The following statements correspond to Theorem~\ref{thm::Dirac_cohomology_simple_supermodules} and Theorem~\ref{Unique} in the main text.

\begin{theorem}
Let $L(\Lambda)$ be a non-trivial unitarizable highest weight $\gg$-supermodule with highest weight $\Lambda$. Then
\[
\DC(L(\Lambda)) \cong L_{0}(\Lambda-\rho_{\bar{1}}).
\]
Consequently, two unitarizable $\gg$-supermodules $\HH_{1}$ and $\HH_{2}$ are equivalent if and only if
$\DC(\HH_{1}) \cong \DC(\HH_{2})$ as $\gg_{\bar{0}}$-modules.
\end{theorem}
 We extend this result to the Dirac cohomology of highest weight supermodules, which possess Jordan--Hölder filtrations with unitarizable quotients.

The second part of this article concerns several applications of Dirac operators and Dirac cohomology. The first result gives a criterion for unitarity in terms of the $\even$-structure of a highest weight $\gg$-supermodule. This is non-trivial in both directions: semisimplicity over $\even$ with unitarizable $\even$-constituents does not by itself imply unitarity, and induction from $\even$ to $\gg$ does not in general preserve it. The missing condition is provided by the Dirac inequality.

Indeed, if $M$ is a highest weight $\gg$-supermodule of highest weight $\Lambda$, then $M$ admits a $\even$-filtration with simple highest weight subquotients $L_{0}(\Lambda-\gamma)$, where $\gamma$ is a sum of distinct positive odd roots. In the unitarizable case, $M$ is completely reducible over $\even$, and the Dirac inequality is strict on every non-trivial $\even$-constituent. This suggests that unitarity should be characterized by the unitarizability of the highest weight $\even$-constituent together with the strict Dirac inequality on all remaining $\even$-composition factors. We show that this is indeed the case. The following is Theorem~\ref{UnitarityD2} in the main text.

\begin{theorem}
Let $M$ be a highest weight $\gg$-supermodule with highest weight $\Lambda$. Then $M$ is unitarizable if and only if the following two conditions are satisfied:
\begin{enumerate}
\item[a)] $\Lambda$ is the highest weight of a unitarizable highest weight $\even$-module.
\item[b)] If $L_{0}(\mu)$ is a simple $\even$-composition factor in a $\even$-filtration of $M$ with highest weight $\mu$, then
\[
(\mu+2\rho,\mu)>(\Lambda+2\rho,\Lambda).
\]
\end{enumerate}
\end{theorem}

Next, we relate Dirac cohomology to Kostant's cohomology. The selfadjointness condition for the Dirac operator fixes a positive system on $\odd$. As a consequence, Kostant's cohomology is naturally defined only as a module over the maximal compact subalgebra $\kk$ of the even part of the chosen real form. In this setting, we obtain the following comparison theorem which is Theorem~\ref{Injection} in the main text.

\begin{theorem}
For any unitarizable simple $\gg$-supermodule $\HH$, there is a $\kk^{\CC}$-module isomorphism
\[
\DC(\HH)\cong \mathrm{H}^{\ast}(\gg_{+1},\HH)\otimes \CC_{-\rho_{\bar{1}}}.
\]
\end{theorem}

We then consider the Euler characteristic of Dirac cohomology. 
For this purpose, we introduce the \emph{Dirac index} of a $\gg$-supermodule $M$, defined as the virtual $\gg_{\bar{0}}$-module
\begin{equation}
\DI(M)\coloneqq M\otimes M(\gg_{\bar{1}})_{\bar{0}} - M\otimes M(\gg_{\bar{1}})_{\bar{1}}.
\end{equation}
The Dirac index coincides with the Euler characteristic of Dirac cohomology (\emph{cf.}~Proposition~\ref{prop::Euler_char}), \emph{i.e.},
\begin{equation}
\DI(M)=\DC^{+}(M)-\DC^{-}(M).
\end{equation}
Moreover, it commutes with tensoring by finite-dimensional $\gg$-supermodules, in analogy with the classical case.

Finally, we obtain two formulas for the formal $\kk^{\CC}$-character. For unitarizable simple supermodules, the first is obtained from Kostant cohomology. For $\Dirac^{2}$-semisimple Harish-Chandra supermodules $\HH$, that is, those satisfying
\begin{equation}
\HH\otimes M(\odd)\cong \Ker\Dirac^{2}\oplus \Im\Dirac^{2},
\end{equation}
the second is obtained from the Dirac index. The main results are Theorems~\ref{formalCharacter} and~\ref{thm::formal_character_Dirac_index}.

\begin{theorem}
Let $F^\nu$ be the simple $\kk^{\CC}$-module of highest weight $\nu\in\hh^\ast$, and set $N(\mu)\coloneqq \bigwedge(\nn_{\bar1}^-)\otimes F^\mu$.
\begin{enumerate}
\item[a)] Let $\HH$ be a unitarizable simple $\gg$-supermodule. Then
\[
\ch(\HH)=\sum_\mu\sum_{k\geq 0}(-1)^k[\mathrm H^k(\gg_{+1},\HH):F^\mu]\,\ch\bigl(N(\mu)\bigr).
\]
\item[b)] Let $\HH$ be a $\Dirac^{2}$-semisimple Harish-Chandra supermodule. If $\DC^+(\HH)=\sum_\mu F^\mu$ and $\DC^-(\HH)=\sum_\nu F^\nu$, then
\[
\ch(\HH)=\sum_\mu \ch\bigl(N(\mu+\rho_{\bar1})\bigr)-\sum_\nu \ch\bigl(N(\nu+\rho_{\bar1})\bigr).
\]
\end{enumerate}
\end{theorem}

\subsection{Structure of the paper} Section~\ref{sec::preliminaries} introduces the structure theory of Lie superalgebras $A(m\vert n)$ and their real forms $\su(p,q\vert r,s)/\mathfrak{psu}(p,q\vert r,s)$.

Section~\ref{UnitarySL} introduces unitarizable $\gg$-supermodules and analyzes the consequences of unitarity. 
Harish-Chandra supermodules are defined, and it is shown that every unitarizable highest weight $\gg$-supermodule arises as a quotient of a Verma supermodule, leading to the Shapovalov form as the canonical contravariant Hermitian form. 
We further examine the $\gg_{\bar{0}}$-constituents of unitarizable highest weight $\gg$-supermodules and introduce a distinguished subclass, the relative holomorphic discrete series modules.

Section~\ref{ch::Dirac_operator_Dirac_cohomology} provides an overview of Dirac operators and Dirac cohomology in the language of the quantum Weil algebra and constructs a Dirac induction functor that is left adjoint to Dirac cohomology on unitarizable supermodules.

Section~\ref{DiracAndUnitarity} examines the interaction between Dirac operators, Dirac cohomology, and unitarity. 
In particular, the Dirac operator is shown to detect unitarity, and the Dirac cohomology of unitarizable simple $\gg$-supermodules is computed.

Finally, Section~\ref{sec::complementary_perspectives} discusses further applications of the Dirac operator and Dirac cohomology. 
These include a new characterization of unitarity, a relation with Kostant’s cohomology, the introduction of a Dirac index and formulas for formal $\kk^{\CC}$-characters.

\subsection{Notation}
We denote by $\ZZ_{+}$ the set of positive integers. Let $\ZZ_{2}\coloneqq \ZZ / 2\ZZ$ be the ring of integers modulo $2$. We denote the elements of $\ZZ_{2}$ by $\overline{0}$ (the residue class of even integers) and $\overline{1}$ (the residue class of odd integers). The ground field is $\CC$, unless otherwise stated.

If $V\coloneqq V_{\bar{0}} \oplus V_{\bar{1}}$ is a super vector space and $v \in V$ is a homogeneous element, then $p(v)$ denotes the parity of $v$, meaning $p(v) = \bar{0}$ if $v \in V_{\bar{0}}$ and $p(v) = \bar{1}$ if $v \in V_{\bar{1}}$. For any two super vector spaces $V$ and $W$, define the $\ZZ_{2}$-graded tensor product $V\otimes W$ by
\[
(V\otimes W)_{\bar{0}}\coloneqq (V_{\bar{0}}\otimes W_{\bar{0}})\oplus(V_{\bar{1}}\otimes W_{\bar{1}}),\quad
(V\otimes W)_{\bar{1}}\coloneqq (V_{\bar{0}}\otimes W_{\bar{1}})\oplus(V_{\bar{1}}\otimes W_{\bar{0}}).
\]
The assignment $(V,W)\mapsto V\otimes W$ is additive in each variable. Moreover, $\otimes$ is associative, and the map $V\otimes W\to W\otimes V$, $v\otimes w\mapsto(-1)^{p(v)p(w)}w\otimes v$, is an isomorphism. If $V$ and $W$ are superalgebras, the product in $V \otimes W$ is $(v_{1}\otimes w_{1})(v_{2}\otimes w_{2})=(-1)^{p(w_{1})p(v_{2})}(v_{1}v_{2}\otimes w_{1}w_{2})$ for any homogeneous $v_{1},v_{2}\in V$ and $w_{1},w_{2} \in W$.

For a given Lie superalgebra $\gg = \even \oplus \odd$, we denote its universal enveloping superalgebra by $\UE(\gg)$. The universal enveloping algebra of the Lie subalgebra $\even$ is denoted by $\UE(\even)$. The centers are denoted by $\mathfrak{Z}(\gg)$ and $\mathfrak{Z}(\even)$, respectively.

By a (left) $\gg$-supermodule, we mean a super vector space $M = M_{\bar{0}} \oplus M_{\bar{1}}$ equipped with a graded linear (left) action of $\gg$, such that $[X, Y]v = X(Yv) - (-1)^{p(X)p(Y)} Y(Xv)$ for all homogeneous $X, Y \in \gg$ and $v \in M$. Analogously, a $\UE(\gg)$-supermodule is defined. A morphism $f: M \to N$ of $\gg$-supermodules is a linear map such that $f(M_{\overline{i}}) \subset N_{\overline{i}}$ and $f(Xv) = Xf(v)$ for all $X \in \gg$ and $v \in M$. We write $\gsmod$ for the category of all (left) $\gg$-supermodules. This category is a $\CC$-linear abelian category, equipped with an endofunctor $\Pi$, the \emph{parity reversing functor}. The parity reversing functor is defined by $\Pi(M)_{\bar{0}} = M_{\bar{1}}$ and $\Pi(M)_{\bar{1}} = M_{\bar{0}}$. Moreover, $\Pi(M)$ is viewed as a $\gg$-supermodule with the induced action $X \cdot v\coloneqq (-1)^{p(X)} Xv$ for any $X \in \gg$ and $v \in M$. In particular, a $\gg$-supermodule $M$ is not necessarily isomorphic to $\Pi(M)$. 

We denote by $\gmod$ the category of (left) $\even$-modules. When $\even$ is considered as a purely even Lie superalgebra, the category $\gsmod$ is simply the direct sum of two copies of $\gmod$. We view any $\even$-module as a $\even$-supermodule concentrated in a single parity. Additionally, if we disregard parity, any $\gg$-supermodule $M$ can be viewed as a $\even$-module, also denoted by $M\vert_{\even}$. In what follows, the $\ZZ_2$-grading of $\even$-supermodules will often be left implicit and will be indicated only when it is relevant.

We will not make an explicit distinction between the categories $\gsmod$ and $\UE(\gg)$\textbf{-smod} or $\gmod$ and $\UE(\even)$\textbf{-mod}, which are the categories of all (left) $\UE(\gg)$-supermodules and left $\UE(\even)$-modules, respectively.

\medskip\noindent \textit{Acknowledgments.} I extend special thanks to Johannes Walcher and Rainer Weissauer for numerous conversations. I also thank Simone Noja and Raphael Senghaas for discussions and collaboration on related projects. 
I am grateful to IMPA - Instituto de Matemática Pura e Aplicada, Rio de Janeiro, for their support and hospitality.
This work is funded by the Deutsche Forschungsgemeinschaft (DFG, German Research Foundation) under
project number 517493862 (Homological Algebra of Supersymmetry: Locality, Unitary, Duality). 
This work is funded by the Deutsche Forschungsgemeinschaft (DFG, German Research Foundation) under
Germany’s Excellence Strategy EXC 2181/1 — 390900948 (the Heidelberg STRUCTURES Excellence Cluster).

\section{Preliminaries} \label{sec::preliminaries}
In this section, we introduce linear Lie superalgebras $A(m\vert n)$, present their structure theory, and describe the real forms $\su(p,q\vert r,s)$ and $\mathfrak{psu}(p,q\vert r,s)$.

\subsection{Lie superalgebras \texorpdfstring{$A(m\vert n)$}\ } \label{subsec::A} The \emph{general linear Lie superalgebra} $\glmn$ is the Lie superalgebra consisting of all block matrices
\begin{equation}
X = \left(\begin{array}{@{}c|c@{}}
 A & B \\
\hline
 C & D
\end{array}\right),
\end{equation}
where $A$ is an $m\times m$ complex matrix, $B$ an $m\times n$ complex matrix, 
$C$ an $n\times m$ complex matrix, and $D$ an $n\times n$ complex matrix. 
The \emph{special linear Lie superalgebra} $\slmn$ is the Lie subsuperalgebra of $\glmn$ consisting of all matrices $X$ with zero supertrace, that is,
\begin{equation}
    \str(X) = \tr(A)-\tr(D)=0.
\end{equation}
This is a codimension-1 ideal in $\glmn$. The even part of $\slmn$ is $\slmn_{\bar{0}} \cong \mathfrak{sl}(m) \oplus \mathfrak{sl}(n) \oplus \mathbb{C} E_{m\vert n}$, where $E_{m\vert n}$ denotes the diagonal matrix with $A = nE_{m}$ and $D = m E_{n}$, and where $E_{m}, E_{n}$ are identity matrices of the indicated sizes. Furthermore, $\slmn$ is simple if and only if $m \neq n$ and $m + n > 2$. When $m = n$ and $m+n > 2$, $\mathfrak{sl}(n\vert n)$ contains the nontrivial center $\mathbb{C} E_{n\vert n}$, and the quotient $\mathfrak{psl}(n\vert n)\coloneqq \mathfrak{sl}(n\vert n)/\mathbb{C} E_{n\vert n}$ is a simple Lie superalgebra. The Lie superalgebra $\mathfrak{sl}(1\vert 1)$ is nilpotent. For the remainder of this article, we assume $m + n > 2$.

In order to treat simultaneously the cases $m=n$ and $m\neq n$, we define the simple Lie superalgebras $A(m\vert n)$ by
\begin{equation}
A(m\vert n)\coloneqq 
\begin{cases}
 \mathfrak{sl}(m+1\vert n+1), & \text{if } m\neq n,\; m,n\ge 0,\\
 \mathfrak{psl}(n+1\vert n+1), & \text{if } m=n,\; n>0.
\end{cases}
\end{equation}
In the sequel we write $\gg\coloneqq A(m-1\vert n-1)$.

We denote by $\dd\coloneqq \{ H = \diag(h_1,\ldots,h_{m+n})\}$ the abelian Lie subalgebra of diagonal matrices 
in $\glmn$. We choose the subspace of diagonal matrices with vanishing supertrace, denoted by $\hh$, as a Cartan subalgebra of $\slmn$. The standard basis of the dual space $\dd^\ast$ of $\dd$ is
$(\epsilon_{1},\dotsc,\epsilon_{m},\delta_{1},\ldots,\delta_{n})$ where $\epsilon_{i}(H)=h_{i},\
\delta_{k}(H)=h_{k+m}$ for $i=1,\dotsc,m$ and $k=1,\dotsc,n$, for any $H\in\dd$. We use the symbols for the induced basis of $\mathfrak{h}^*$ of $\mathfrak{h}$. Namely, we identify weights $\lambda \in \hh^{\ast}$ for $\gg$ 
with tuples $(\lambda_{1},\dotsc,\lambda_{m}\vert \lambda'_{1},\dotsc,\lambda'_{n})$ via the expansion
$
\lambda=\lambda_{1}\epsilon_{1}+\dotsc+\lambda_{m}\epsilon_{m}+\lambda'_{1}\delta_{1}+\dotsc+\lambda'_{n}\delta_{n},
$
keeping in mind that shifts by $(1,\ldots,1\vert {-}1,\ldots,-1)$ do not change the weight. 

The Cartan subalgebra $\hh$ of $\gg$ is, in particular, a Cartan subalgebra for $\even$, and hence acts semisimply on $\even$. However, $\odd$ is a completely reducible $\even$-module under the adjoint action induced by the matrix supercommutator $[\cdot, \cdot]$, so that we obtain a $\ZZ_{2}$-graded root space decomposition of $\gg$:
\begin{equation}
 \gg = \hh \oplus \bigoplus_{\alpha \in \hh^{\ast}} \gg^{\alpha}, \qquad \gg^{\alpha}\coloneqq \{ X \in \gg \ : \ [H, X] = \alpha(H) X \ \text{for all} \ H \in \hh \}.
\end{equation}
We call $\alpha \in \hh^{\ast}$ a \emph{root} if $\alpha \neq 0$ and $\gg^{\alpha} \neq \{0\}$. Elements of $\gg^{\alpha}$ are called \emph{root vectors}, and $\gg^{\alpha}$ is called the \emph{root space} corresponding to the root $\alpha$. Let $\Delta$ denote the set of all roots. Each root space $\gg^{\alpha}$ has either superdimension $(1\vert 0)$ or $(0\vert 1)$. A root $\alpha \in \Delta$ is called \emph{even} if $\gg^{\alpha} \cap \even \neq \{0\}$, and \emph{odd} if $\gg^{\alpha} \cap \odd \neq \{0\}$. The associated sets of roots are denoted by $\Delta_{\bar{0}}$ and $\Delta_{\bar{1}}$, respectively. Altogether, the set of roots of $\gg$ is given by $\Delta = \Delta_{\bar{0}} \sqcup \Delta_{\bar{1}}$, where
\begin{equation}
\begin{split}
 \Delta_{\bar{0}} &= \{\pm(\epsilon_{i} - \epsilon_{j}), \pm(\delta_{k} - \delta_{l}) : 1 \leq i < j \leq m, \ 1 \leq k < l \leq n \}, \\
 \Delta_{\bar{1}} &= \{\pm(\epsilon_{r} - \delta_{s}) : 1 \leq r \leq m, \ 1 \leq s \leq n \}.
\end{split}
\end{equation}
The set of even roots $\Delta_{\bar{0}}$ is the disjoint union of the root systems of $\mathfrak{sl}(m)$ and $\mathfrak{sl}(n)$.

On $\Delta_{\bar{0}}$, we fix the standard positive system for the remainder of this article:
\begin{equation} \label{eq::standard_ordering}
 \Delta_{\bar{0}}^{+}\coloneqq \{\epsilon_{i} - \epsilon_{j}, \ \delta_{k} - \delta_{l} \ : \ 1 \leq i < j \leq m, \ 1 \leq k < l \leq n\},
\end{equation}
such that the root vectors for $\epsilon_{i}-\epsilon_{j}$, $i<j$ are strictly upper triangular
matrices of $\sl(m)$, and the root vectors for $\delta_{k}-\delta_{l}$, $k<l$, are strictly upper
triangular matrices of $\sl(n)$, both diagonally embedded in $\even$. A positive system for $\Delta_{\bar{1}}^{+}$ will be specified in Section~\ref{subsubsec::Real_form}. Then the set of positive roots is $\Delta^{+} = \Delta_{\bar{0}}^{+} \sqcup \Delta_{\bar{1}}^{+}$. With respect to a choice of $\Delta^{+}$, the Lie superalgebra $\gg$ has a triangular decomposition
\begin{equation}
\gg = \nn^{-} \oplus \hh \oplus \nn^{+}, \qquad \nn^{\pm}\coloneqq \sum_{\alpha \in \Delta^{+}} \gg^{\pm \alpha},
\end{equation}
with associated \emph{Borel subalgebra} $\bb = \hh \oplus \nn^{+}$. Restricting the triangular decomposition of $\gg$ to $\even$ yields a triangular decomposition of $\even$, namely $\even = \nn_{\bar{0}}^{-} \oplus \hh \oplus \nn_{\bar{0}}^{+}$, with $\nn_{\bar{0}}^{\pm}\coloneqq \sum_{\alpha \in \Delta_{\bar{0}}^{+}} \gg^{\pm \alpha}$.

For a fixed positive system $\Delta^{+}$, we define the \emph{fundamental system} $\Uppi \subset \Delta^{+}$ to be the set of all $\alpha \in \Delta^{+}$ which cannot be written as the sum of two roots in $\Delta^{+}$. Elements of $\Uppi$ are called \emph{simple} roots. For $\Uppi = \{\alpha_{1},\dots,\alpha_{r}\}$, any $\alpha \in \Delta$ can be uniquely represented as a linear combination
\begin{equation}
\alpha = \sum_{i = 1}^{r} k_{i}\alpha_{i},
\end{equation}
where either all $k_{i} \in \ZZ_{\geq 0}$ or all $k_{i} \in \ZZ_{\leq 0}$.

The general linear Lie superalgebra $\glmn$ is an example of a quadratic Lie superalgebra, that is, it carries a non-degenerate, even, supersymmetric, and invariant bilinear form $(\cdot, \cdot) : \glmn \times \glmn \longrightarrow \CC$ given by
\begin{equation} \label{eq::supertrace_form}
(X, Y)\coloneqq \str(XY)
\end{equation}
for any $X, Y \in \glmn$. Here, invariance means that $([X, Y], Z) = (X, [Y, Z])$ for all $X, Y, Z \in \glmn$, while supersymmetry means that $(\cdot, \cdot)$ is symmetric on $\glmn_{\bar{0}}$, skew-symmetric on $\glmn_{\bar{1}}$, and that $\glmn_{\bar{0}}$ and $\glmn_{\bar{1}}$ are orthogonal to each other. The restriction of $(\cdot, \cdot)$ to $\slmn$ is non-degenerate only when $m \neq n$. If $m = n$, the one-dimensional center of $\mathfrak{sl}(n\vert n)$ becomes the radical of $(\cdot, \cdot)$. Furthermore, the restriction to $\dd$ is non-degenerate while the restriction to $\hh$ is non-degenerate if $m \neq n$. The bilinear form induced
on the dual spaces is denoted by the same symbol. On the standard basis, we have for all
$1 \leq i,j \leq m$ and $1 \leq k,l \leq n$:
\begin{equation} \label{eq::roots}
 (\epsilon_{i}, \epsilon_{j}) = \delta_{ij}, \qquad (\delta_{k}, \delta_{l}) = - \delta_{kl}, \qquad (\epsilon_{i}, \delta_{k}) = 0.
\end{equation}
For any root $\alpha \in \Delta$, we associate a unique element $h_{\alpha} \in \hh$ through the condition $\alpha(H) = (H, h_{\alpha})$ for all $H \in \hh$. This element $h_{\alpha}$ is called the \emph{coroot} associated with $\alpha$. If $m = n$, we fix the coroots as elements of $\hh$ by requiring $\alpha(H) = (H, h_{\alpha})$ for all $H \in \dd$.
However, the bilinear form on $\hh^{\ast}$, given by the linear extension of $(\alpha,\beta)\coloneqq (h_{\alpha},h_{\beta})$ for $\alpha,\beta\in\Delta$, remains 
non-degenerate only when $m\neq n$. In addition, it follows from Equation~\ref{eq::roots} that the odd roots are all isotropic, that is, 
$(\alpha,\alpha)=0$ for all $\alpha\in\Delta_{\bar{1}}$.

The root system $\Delta$ admits a natural action by the Weyl group $W$ associated with the even part $\even$ of the Lie superalgebra. This group $W$ is isomorphic to the direct product of symmetric groups $S_m \times S_n$, acting on $m$ and $n$ elements, respectively. It is generated by reflections with respect to the even roots, given by
\begin{equation}
\label{evenreflection}
R_\alpha(\beta)\coloneqq \beta - 2\frac{(\alpha,\beta)}{(\alpha,\alpha)} \alpha, \qquad \alpha\in \Delta_{\bar{0}},\ \beta\in\Delta.
\end{equation}
This action extends linearly to the dual space $\hh^*$ and preserves the bilinear form $(\cdot,\cdot)$. The \emph{Weyl group of $\gg$} is defined to be the Weyl group of $\even$. The \emph{Weyl vector} $\rho$ is $\rho=\rho_{\bar{0}}-\rho_{\bar{1}}$ where
\begin{equation}
\rho_{\bar{0}}\coloneqq \frac{1}{2}\sum_{\alpha \in \Delta_{\bar{0}}^{+}}\alpha = \frac{1}{2}\left( \sum_{i=1}^{m} (m-2i+1)\epsilon_{i} + \sum_{j=1}^{n}(n-2j+1)\delta_{j}\right), \qquad \rho_{\bar{1}}\coloneqq 
\frac{1}{2}\sum_{\beta \in \Delta_{\bar{1}}^{+}}\beta.
\end{equation}
We define the \emph{dot action} of $W$ on $\hh^{\ast}$ by
\begin{equation}
w \cdot \lambda\coloneqq w(\lambda + \rho) - \rho
\end{equation}
for any $\lambda \in \hh^{\ast}$ and $w \in W$. We say that $\lambda, \mu \in \hh^{\ast}$ are $W$\emph{-linked} if there exists some $w \in W$ such that $\mu = w \cdot \lambda$. This gives an equivalence relation on $\hh^{\ast}$, and the Weyl orbit $\{w \cdot \lambda : w \in W\}$ of $\lambda$ under the dot action is called the $W$\emph{-linkage class} of $\lambda$. 

Finally, the Lie superalgebra $\gg$ of type $A$ is an example of an \emph{involutive Lie superalgebra}, that is, a Lie superalgebra endowed with a conjugate-linear anti-involution $\omega: \gg \to \gg$. A conjugate-linear map $\omega : \gg \to \gg$ is called a \emph{conjugate-linear anti-involution} if $\omega \vert_{\gg_{\bar{0},\bar{1}}} \neq \operatorname{id}_{\gg_{\bar{0},\bar{1}}}$, $\omega^{2} = \operatorname{id}_{\gg}$, and 
\begin{equation}
\omega([x,y]) = [\omega(y),\omega(x)], \qquad \forall x,y \in \gg.
\end{equation}
The conjugate-linear anti-involutions $\omega$ are in one-to-one correspondence with real forms of $\gg$ \cite{ParkerRealForms,serganova1983classification}, and we will see that a specific choice of $\omega$ is essential when studying unitarizable supermodules over $\gg$.

\subsection{Special unitary Lie superalgebras \texorpdfstring{$\mathfrak{(p)su}(p,q \vert r,s)$}\ } \label{sec::realforms}

The \emph{special unitary (indefinite) Lie superalgebras} $\mathfrak{(p)su}(p,q\vert r,s)$ are real forms of $\gg = A(m-1\vert n-1)$, where $p+q = m$ and $r+s = n$. We construct the real Lie superalgebras $\mathfrak{(p)su}(p,q\vert r,s)$ using Hermitian forms rather than super Hermitian forms, following \cite[Section 2]{jakobsen1994full}. For that, we consider $V = \CC^{m\vert n}$, \emph{i.e.}, the complex super vector space of superdimension $(m\vert n)$. For any fixed $p,q,r,s \in \ZZ_{+}$ with $p+q = m$ and $r+s = n$, we equip $V$ with the non-degenerate Hermitian form $\langle \cdot,\cdot\rangle: V\times V \to \CC$ given by 
\begin{equation}
 \langle v, w \rangle\coloneqq \overline{v}^{T} J_{(p,q\vert r,s)} w, \qquad J_{(p,q\vert r,s)}\coloneqq \left(\begin{array}{@{}c|c@{}}
 I_{p,q}
 & 0 \\
\hline
 0 &
 I_{r,s}
\end{array}\right).
\end{equation}
Here, $\overline{\ \cdot\ }$ means complex conjugation, and we consider any element of $\CC^{m\vert n}$ as a column vector. The matrix $I_{k,l}$ is the diagonal matrix whose first $k$ entries are $1$ and whose last $l$ entries are $-1$. This Hermitian form is \emph{consistent}, that is, $\langle V_{\bar{0}},V_{\bar{1}}\rangle = 0$.

We define the \emph{(indefinite) unitary Lie superalgebra} $\mathfrak{u}(p,q\vert r,s) = \mathfrak{u}(p,q\vert r,s)_{\bar{0}} \oplus \mathfrak{u}(p,q\vert r,s)_{\bar{1}}$ as the Lie superalgebra preserving the Hermitian form $\langle\cdot,\cdot\rangle$:
\begin{equation}
\mathfrak{u}(p,q\vert r,s)_{\bar{k}}\coloneqq \{ X \in \glmn_{\bar{k}}: \langle Xv,w\rangle+\langle v,Xw\rangle =0 \ \text{for all} \ v,w \in V\}.
\end{equation}
The \emph{(indefinite) special unitary Lie superalgebra} is defined as $\mathfrak{su}(p,q\vert r,s)\coloneqq \mathfrak{u}(p,q\vert r,s) \cap A(m-1\vert n-1)$ if $m \neq n$ and $\mathfrak{psu}(p,q\vert r,s)\coloneqq \mathfrak{u}(p,q\vert r,s) \cap A(n-1\vert n-1)$ if $m=n$. To treat both cases uniformly, we write
\begin{equation}
 \mathfrak{(p)su}(p,q\vert r,s)\coloneqq \begin{cases}
 \mathfrak{su}(p,q\vert r,s), \qquad &\text{if} \ m\neq n, \\
 \mathfrak{psu}(p,q\vert r,s) \qquad &\text{if} \ m = n.
 \end{cases}
\end{equation}

In our standard realization of $\gg$, $\mathfrak{(p)su}(p,q\vert r,s)$ can be described explicitly as
\begin{equation}
\mathfrak{(p)su}(p,q\vert r,s) = \left\{ X \in \gg : J_{(p,q\vert r,s)}^{-1} X^{\dagger} J_{(p,q\vert r,s)} = - X \right\}
\end{equation}
with $(\cdot)^{\dagger}$ denoting complex conjugate-transpose. The associated conjugate-linear anti-involution on $\gg$ defining the real form $\mathfrak{(p)su}(p,q\vert r,s)$ is
\begin{equation} \label{eq::definition_omega_matrix}
\omega(X) = J_{(p,q\vert r,s)}^{-1} X^{\dagger} J_{(p,q\vert r,s)}, \qquad X \in \gg.
\end{equation}

Of particular interest are the real forms $\mathfrak{(p)su}(p,q\vert n,0)$ and $\mathfrak{(p)su}(p,q\vert 0,n)$ for $p,q \neq 0$, as they are the only real forms of $\gg$ admitting non-trivial unitarizable supermodules (see Theorem~\ref{thm::HW_property_M}). 

\subsubsection{Lie superalgebras \texorpdfstring{$\mathfrak{(p)su}(p,q \vert n,0)$}\ and \texorpdfstring{$\mathfrak{(p)su}(p,q \vert 0,n)$}\ } \label{subsubsec::Real_form}

To simplify notation, we restrict attention to the real forms $\mathfrak{su}(p,q\vert n,0)$ and $\mathfrak{su}(p,q\vert 0,n)$ for $m \neq n$, which are real forms of $\sl(m\vert n)$ for $p+q=m$; the cases $\mathfrak{psu}(p,q\vert n,0)$ and $\mathfrak{psu}(p,q\vert 0,n)$ are obtained by factoring out the center.

Assume $m \neq n$. We fix as a real form either $\su(p,q \vert n,0)$ or $\su(p,q \vert 0,n)$. Then the even subalgebra of $\gg$ is $\even \cong (\mathfrak{su(p,q)\oplus \su(n)} \oplus \u(1))^{\CC}$ to emphasize the real form. This leads to a natural refinement of the set of even roots $\Delta_{\bar{0}}$. The Lie subalgebra of the real form has a maximal compact subalgebra, namely
\begin{equation}
\kk\coloneqq 
 \su(p)\oplus \su(q)\oplus \u(1) \oplus \su(n)\oplus \u(1)
\end{equation}
diagonally embedded in $\gg$. The subalgebra $\kk$ satisfies the \emph{equal rank condition}
\begin{equation}
\hh \subset \kk^{\CC} \subset \even \subset \gg,
\end{equation}
\emph{i.e.}, $\hh$ is a Cartan subalgebra for $\gg, \even$ and $\kk^{\CC}$. We say a root $\alpha \in \Delta_{\bar{0}}$ is \emph{compact} if the associated root vector lies in $\kk^{\CC}$; otherwise, the root is referred to as \emph{non-compact}. The compact roots make up precisely the root system of $(\hh,\kk^{\CC})$ and can be identified with the subset
\begin{equation}
\Delta_{c}\coloneqq \{\pm(\epsilon_{i}-\epsilon_{j}) : 1\leq i < j \leq p, \ p+1\leq i < j \leq m\} \cup \{\pm (\delta_{i}-\delta_{j}) : 1\leq i< j \leq n\}.
\end{equation}
The \emph{non-compact roots} will be denoted by $\Delta_{nc}\coloneqq \Delta_{\bar{0}} \setminus \Delta_{c}$, such that 
\begin{equation}
\Delta_{\bar{0}} = \Delta_{c} \sqcup \Delta_{nc}.
\end{equation}

The fixed positive system on $\Delta_{\bar{0}}$ induces a positive system on $\Delta_{c,nc}$, namely, $\Delta_{c,nc}^{+}\coloneqq \Delta^{+} \cap \Delta_{c,nc}$, and in analogy to Section~\ref{subsec::A}, we set 
\begin{equation}
\nn_{\bar{0},c}^{+}\coloneqq \bigoplus_{\alpha \in \Delta_{c}^{+}} \gg_{\bar{0}}^{\alpha}, \quad \text{and} \quad \nn_{\bar{0},nc}^{+}\coloneqq \bigoplus_{\alpha \in \Delta_{nc}^{+}} \gg_{\bar{0}}^{\alpha}.
\end{equation}
The associated Weyl vectors are denoted by $\rho_{c}$ and $\rho_{nc}$, respectively.

This refinement of the even structure allows for an explicit characterization of the possible positive odd root systems in $\Delta_{\bar{1}}$ compatible with the chosen real form. In general, we distinguish two cases: the case where either $p = 0$ or $q = 0$, and the case where $p, q \neq 0$. In the former case, the real form under consideration is the compact real form of $\gg$ satisfying $\even \cong \kk^{\CC}$.

We first consider the case in which either $p = 0$ or $q = 0$, that is, we consider the real Lie superalgebras $\su(m,0\vert n,0) \cong \su(0,m\vert 0,n)$ or $\su(m,0\vert 0,n) \cong \su(0,m\vert n,0)$ whose Lie subalgebra is isomorphic to $\kk^{\CC}$. The odd part $\odd$ is a completely reducible $\kk^{\CC}$-module under the adjoint action, and it decomposes as 
\begin{equation}
\odd = \left\{
\left(
\begin{array}{@{}c|c@{}}
 0 & B \\
\hline
 0 & 0
\end{array}
\right) \ : \ B\in \Mat(m\times n;\CC)
\right\} \oplus 
\left\{
\left(
\begin{array}{@{}c|c@{}}
 0 & 0 \\
\hline
 C & 0
\end{array}
\right) \ : \ C \in \Mat(n\times m;\CC)
\right\}.
\end{equation} Consequently, there are exactly two conjugate-linear anti-involutions compatible with the standard ordering, which produce the even real forms $\su(m) \oplus \su(n)$ on $\even$:
\begin{equation}
\omega_{\pm}
\left(
\begin{array}{@{}c|c@{}}
 A & B \\
\hline
 C & D
\end{array}
\right)
= 
\left(
\begin{array}{@{}c|c@{}}
 A^{\dagger} & \pm C^{\dagger} \\
\hline
 \pm B^{\dagger} & D^{\dagger}
\end{array}
\right).
\end{equation}
The real Lie superalgebra $\su(m,0\vert n,0) \cong \su(0,m\vert 0,n)$ corresponds to the conjugate-linear anti-involution $\omega_{+}$, while the real Lie superalgebra $\su(m,0\vert 0,n) \cong \su(0,m\vert n,0)$ corresponds to the conjugate-linear anti-involution $\omega_{-}$.
There are essentially two choices for the odd positive roots, namely
\begin{equation}
\begin{split}
 \Delta_{\bar{1}, \text{st}}^{+}\coloneqq \{ \epsilon_{i} - \delta_{j} \ : \ 1 \leq i \leq m, \ 1 \leq j \leq n\}, \qquad \Delta_{\bar{1}, \text{-st}}^{+}\coloneqq \{ -\epsilon_{i} + \delta_{j} \ : \ 1 \leq i \leq m, \ 1 \leq j \leq n\},
 \end{split}
\end{equation}
which differ by a sign. We fix $\Delta_{\bar{1},\text{st}}^{+}$ as our positive system, so that $\nn^{+}$ consists precisely of the upper block matrices in $\gg$. We denote this system simply by $\Delta^{+}$.
\\
\\
We now describe the conjugate-linear anti-involutions compatible with the standard ordering that produce the real forms $\su(p,q) \oplus \su(n)$ in the semisimple part of $\even$ for $p, q \neq 0$. In this case, the odd part decomposes under $\kk^{\CC}$ as
\[
\odd = \qq_{1} \oplus \qq_{2} \oplus \pp_{1} \oplus \pp_{2},
\]
where explicitly
\begin{equation}
\begin{split}
\pp_{1} &= \left\{
\left(\begin{array}{@{}cc|c@{}}
 0 & 0 & P_{1} \\
 0 & 0 & 0 \\
\hline
 0 & 0 & 0
\end{array}\right) : P_{1} \in \Mat(p \times n; \CC) \right\}, 
\qquad \pp_{2} = \left\{
\left(\begin{array}{@{}cc|c@{}}
 0 & 0 & 0 \\
 0 & 0 & P_{2} \\
\hline
 0 & 0 & 0
\end{array}\right) : P_{2} \in \Mat(q \times n; \CC) \right\}, \\
\qq_{1} &= \left\{
\left(\begin{array}{@{}cc|c@{}}
 0 & 0 & 0 \\
 0 & 0 & 0 \\
\hline
 Q_{1} & 0 & 0
\end{array}\right) : Q_{1} \in \Mat(n \times p; \CC) \right\}, 
\qquad \qq_{2} = \left\{
\left(\begin{array}{@{}cc|c@{}}
 0 & 0 & 0 \\
 0 & 0 & 0 \\
\hline
 0 & Q_{2} & 0
\end{array}\right) : Q_{2} \in \Mat(n \times q; \CC) \right\}.
\end{split}
\end{equation}
Given this decomposition of $\odd$, we express a general element $X \in \gg$ as
\begin{equation}
X=\left(\begin{array}{@{}c|c@{}}
 \begin{matrix}
 a & b \\ c & d
 \end{matrix}
 & \begin{matrix}
 P_{1} \\ P_{2}
 \end{matrix} \\
\hline
 \begin{matrix}
 Q_{1} & Q_{2}
 \end{matrix} &
 E
\end{array}\right),
\end{equation}
where $P_{i} \in \pp_{i}$, $Q_{j} \in \qq_{j}$ for $1 \leq i, j \leq 2$, $\begin{pmatrix}
 a & b\\ c & d 
\end{pmatrix} \in \su(p,q)^{\CC}$, $E \in \su(n)^{\CC}$ and $\tr \begin{pmatrix}
 a & b\\ c & d 
\end{pmatrix} - \tr E = 0$.
Then, there are exactly two conjugate-linear anti-involutions compatible with the standard ordering, whose restriction to the semisimple part of $\even$ yields the real form $\su(p,q) \oplus \su(n)$ \cite[Lemma 5.1]{jakobsen1994full}: 
\begin{equation} \label{eq::form_omega_plus_omega_minus} \omega_{(-,+)}(X) = \left(\begin{array}{@{}c|c@{}}
 \begin{matrix}
 a^{\dagger} & -c^{\dagger} \\ -b^{\dagger} & d^{\dagger}
 \end{matrix}
 & \begin{matrix}
 -Q_{1}^{\dagger} \\ Q_{2}^{\dagger}
 \end{matrix}\\
\hline
 \begin{matrix}
 -P_{1}^{\dagger} & P_{2}^{\dagger}
 \end{matrix} &
 \begin{matrix}
 E^{\dagger}
 \end{matrix}
\end{array}\right), \qquad
 \omega_{(+,-)}(X) = \left(\begin{array}{@{}c|c@{}}
 \begin{matrix}
 a^{\dagger} & -c^{\dagger} \\ -b^{\dagger} & d^{\dagger}
 \end{matrix}
 & \begin{matrix}
 Q_{1}^{\dagger} \\ -Q_{2}^{\dagger}
 \end{matrix}\\
\hline
 \begin{matrix}
 P_{1}^{\dagger} & -P_{2}^{\dagger}
 \end{matrix} &
 \begin{matrix}
 E^{\dagger}
 \end{matrix}
\end{array}\right).
\end{equation}
The real Lie superalgebra $\su(p,q\vert 0,n)$ corresponds to $\omega_{(-,+)}$ while $\su(p,q\vert n,0)$ corresponds to $\omega_{(+,-)}$. We conclude that essentially three positive systems are relevant for the odd part. These are uniquely determined by
\begin{equation}
\nn_{\bar{1}, \text{st}}^{+}\coloneqq \pp_{1} \oplus \pp_{2}, \quad \nn_{\bar{1}, \text{-st}}^{+}\coloneqq \qq_{1} \oplus \qq_{2}, \quad \nn_{\bar{1}, \text{nst}}^{+}\coloneqq \pp_{1} \oplus \qq_{2},
\end{equation}
referred to as \emph{standard}, \emph{minus standard}, and \emph{non-standard}, respectively. We will fix the non-standard choice and denote it by $\Delta^{+}$, as this choice is preferred in the context of the Dirac operator, as explained in Remark~\ref{rmk::Dirac_standard}. In this case, we have
\begin{equation}
 \rho_{\bar{1}}
=\frac{1}{2}\left(
n\sum_{i=1}^{p}\epsilon_{i}
- n\sum_{j=p+1}^{m}\epsilon_{j}
+ (q-p)\sum_{k=1}^{n}\delta_{k}\right).
\end{equation}

For convenience, we set $\nn_{\bar{1}}^{+}\coloneqq \bigoplus_{\alpha \in \Delta_{\bar{1}}^{+}}\gg^{\alpha}$ and $\nn_{\bar{1}}^{-}\coloneqq \bigoplus_{\alpha \in \Delta_{\bar{1}}^{+}}\gg^{-\alpha}$, leading to the decomposition of super vector spaces $\odd = \nn_{\bar{1}}^{-} \oplus \nn_{\bar{1}}^{+}$ and $\gg = \nn_{\bar{1}}^{-} \oplus \even \oplus \nn_{\bar{1}}^{+}$. However, $[\nn_{\bar{1}}^{\pm}, \nn_{\bar{1}}^{\pm}] \neq \{0\}$, indicating that this $\mathbb{Z}$-grading is not compatible with the $\mathbb{Z}_{2}$-grading since $\nn_{\bar{1}}^{\pm}$ are invariant under $\kk^{\CC}$ but not under $\even$. The following lemma summarizes the important commutation relations.

\begin{lemma} \label{lemm::commutation_relations}
 The following commutation relations hold in $\gg$ for $1\leq k,l\leq 2$:
$$
[\kk^{\CC}, \pp_{k}] \subset \pp_{k}, \qquad [\kk^{\CC}, \qq_{l}] \subset \qq_{l}, \qquad [\nn_{\bar{1}}^{+}, \nn_{\bar{1}}^{+}] \subset \nn_{\bar{0},nc}^{+}.
$$
\end{lemma}

For the remainder of this article, unless otherwise stated, we assume $m \leq n$ and $p, q \neq 0$. The results for the finite-dimensional case are then obtained by setting $p = 0$ or $q = 0$ and utilizing $\kk^{\CC} = \even$.

\section{Unitarizable supermodules} \label{UnitarySL}

In this section, we present unitarizable supermodules over $\gg$, examine their properties and algebraic realization, and explain how they relate to Harish-Chandra supermodules.

\subsection{Unitarizable supermodules} \label{subsec::Unitarizable} There are two equivalent approaches to defining unitarizable supermodules over $\gg$, based on the fact that Hermitian forms and super Hermitian forms on a complex super vector space can be transformed into one another. In this paper, we work with Hermitian forms, viewing them as complex-valued maps that are conjugate-linear in the first variable and linear in the second. For a more thorough discussion comparing Hermitian and super Hermitian forms, see \cite{jakobsen1994full, SchmidtGeneralizedSuperdimension}.

Let $\HH=\HH_{\bar{0}}\oplus \HH_{\bar{1}}$ be a complex super vector space equipped with a positive definite Hermitian form $\langle\cdot,\cdot\rangle$ such that $\langle \HH_{\bar{0}},\,\HH_{\bar{1}} \rangle = 0$.
We refer to such a structure as a \emph{pre-super Hilbert space}. Given this structure, we can now define unitarizable supermodules. These supermodules are defined with respect to real forms of the Lie superalgebra $\gg$, which are in one-to-one correspondence with conjugate-linear anti-involutions $\omega$ \cite{ParkerRealForms, serganova1983classification} (see Section~\ref{subsubsec::Real_form} for examples), such that the Hermitian form of the pre-super Hilbert space is \emph{contravariant} with respect to $\omega$.

\begin{definition}[{\cite[Definition 2.3]{jakobsen1994full}}]
Let $\HH$ be a $\gg$-supermodule, and let $\omega$ be a conjugate-linear anti-involution on $\gg$. The supermodule $\HH$ is called an $\omega$\emph{-unitarizable} $\gg$\emph{-supermodule} if $\HH$ is a super pre-Hilbert space such that for all $v,w \in \HH$ and all $X \in \gg$, the Hermitian form $\langle \cdot,\cdot \rangle$ is $\omega$-contravariant:
$$
\langle Xv,w \rangle = \langle v,\omega(X)w \rangle.
$$
\end{definition}

If we instead consider $\UE(\gg)$-supermodules, we extend $\omega$ to $\UE(\gg)$ in the obvious way, keeping the same notation. In this setting, a $\gg$-supermodule $\HH$ is unitarizable precisely when it is a \emph{Hermitian supermodule} over $(\UE(\gg), \omega)$, that is, when
$
\langle Xv, w \rangle = \langle v, \omega(X)w \rangle
$
holds for all $v, w \in \HH$ and $X \in \UE(\gg)$. Whenever $\omega$ is clear from the context, we will simply say “unitarizable”.

Direct consequences of this definition are collected in the following proposition.

\begin{proposition}\label{CompletelyReducible}
 Let $\HH$ be a $\omega$-unitarizable $\gg$-supermodule. Then the following assertions hold:
 \begin{enumerate}
 \item[a)] $\HH$ is completely reducible; that is, for any invariant subsuperspace, its orthogonal complement is also an invariant subsuperspace.
 \item[b)] As a $\gg_{\bar{0}}$-supermodule, $\HH$ is $\omega\vert_{\even}$-unitarizable. In particular, it is completely reducible as a $\even$-supermodule.
 \end{enumerate}
\end{proposition}

We will see that unitarizability is a very strong condition for supermodules since only a specific but central class of supermodules can satisfy it: the \emph{highest} or \emph{lowest weight supermodules}. 

\begin{definition}
A $\gg$-supermodule $M$ is called a \emph{highest weight} $\gg$\emph{-supermodule with highest weight} $\Lambda \in \hh^{\ast}$ if there exists a nonzero vector $v_{\Lambda} \in M$ such that:
\begin{enumerate}
 \item[a)] $Xv_{\Lambda}=0$ for all $X \in \nn^{+}$,
 \item[b)] $Hv_{\Lambda}=\Lambda(H)v_{\Lambda}$ for all $H \in \hh$, and 
 \item[c)] $\UE(\gg)v_{\Lambda}=M$.
\end{enumerate}
The vector $v_{\Lambda}$ is referred to as the highest weight vector of $M$. \emph{Lowest weight} $\gg$\emph{-supermodules} are defined analogously. 
\end{definition}

\begin{remark}
 We fixed in Section~\ref{sec::realforms} the non-standard odd positive system $\Delta^{+}_{\bar{1}}\coloneqq \Delta^{+}_{\bar{1},\text{nst}}$ for this article, which is not a substantial restriction. Indeed, set $\Delta_{\text{nst}}^{+}\coloneqq \Delta_{\bar{0}}^{+} \sqcup \Delta_{\bar{1},\text{nst}}^{+}$ and $\Delta_{\text{st}}^{+}\coloneqq \Delta_{\bar{0}}^{+} \sqcup \Delta_{\bar{1},\text{st}}^{+}$. If $L(\Lambda; \Delta_{\text{nst}}^{+})$ is a simple highest weight $\gg$-supermodule with respect to $\Delta_{\text{nst}}^{+}$, then there exists a $\Lambda' \in \hh^{\ast}$ such that either $L(\Lambda; \Delta_{\text{nst}}^{+}) \cong L(\Lambda'; \Delta_{\text{st}}^{+})$ or $L(\Lambda; \Delta_{\text{nst}}^{+}) \cong \Pi L(\Lambda'; \Delta_{\text{st}}^{+})$, with $\Lambda' + \rho_{\text{st}} = \Lambda + \rho_{\text{nst}}$, where $\rho_{\text{nst}}$ and $\rho_{\text{st}}$ are the Weyl vectors for $\Delta^{+}_{\text{nst}}$ and $\Delta^{+}_{\text{st}}$, respectively. This is because the associated fundamental systems can be transformed into each other via a sequence of odd reflections, as explained in \cite[Section 1.4]{cheng2012dualities}, using $\Delta_{\text{st}}^{+} \cap \Delta_{\bar{0}} = \Delta_{\bar{0}}^{+} = \Delta_{\text{nst}}^{+} \cap \Delta_{\bar{0}}$, and the multiple application of \cite[Lemma 1.40]{cheng2012dualities}.
\end{remark}

The central theorem concerning unitarizable supermodules is the following:

\begin{theorem}[\cite{furutsu1991classification,neeb2011lie}] \label{thm::HW_property_M} 
The Lie superalgebra $\gg$ of type $A$ admits non-trivial unitarizable $\gg$-supermodules only with respect to the real forms $\mathfrak{(p)su}(p,q\vert n,0)$ and $\mathfrak{(p)su}(p,q\vert 0,n)$. In particular, any such simple supermodule is either a lowest weight or a highest weight $\gg$-supermodule.
\end{theorem}

It suffices to consider highest weight $\gg$-supermodules, as the case of lowest weight $\gg$-supermodules follows by the same arguments. 

Recall the anti-involutions $\omega_{(-,+)}$ and $\omega_{(+,-)}$ of \eqref{eq::form_omega_plus_omega_minus}; the real forms $\mathfrak{(p)su}(p,q\vert 0,n)$ and $\mathfrak{(p)su}(p,q\vert n,0)$ correspond respectively to $\omega_{(-,+)}$ and $\omega_{(+,-)}$. Notably, the Lie superalgebra $\gg$ admits non-trivial unitarizable highest weight supermodules only with respect to $\omega_{(-,+)}$, whereas $\omega_{(+,-)}$ corresponds to unitarizable lowest weight $\gg$-supermodules (see \cite{jakobsen1994full, SchmidtGeneralizedSuperdimension}). Consequently, we restrict attention to the real form $\mathfrak{(p)su}(p,q\vert 0,n)$ and, for notational convenience, write $\mathfrak{(p)su}(p,q\vert n)$ instead. In what follows, we omit the prefix $\omega_{(-,+)}$ and refer simply to unitarizable highest weight $\gg$-supermodules.

Next, we collect further algebraic properties of unitarizable highest weight $\gg$-supermodules. We remind the reader of our assumptions that, unless otherwise stated, $m \leq n$ and $p, q \neq 0$.

\subsection{Harish-Chandra supermodules} \label{subsec::HC_Supermodule}
Unitarizable highest weight $\gg$-supermodules constitute a class of \emph{Harish-Chandra supermodules}; that is, they carry an intrinsic structure induced by the decomposition relative to a maximal compact subalgebra of a chosen real form.

Fix a real form $\mathfrak{(p)su}(p,q\vert n)$ of $\gg$, and let $\kk$ be the maximal compact subalgebra of $\mathfrak{(p)su}(p,q\vert n)_{\bar{0}}$ introduced in Section~\ref{sec::realforms}, such that the following relation holds:
\begin{equation}
\kk^{\CC} \subset \even \subset \gg.
\end{equation}
Any $\gg$-supermodule $M$ can be viewed as a $\even$-supermodule, and any $\even$-supermodule can be regarded as a $\kk^{\CC}$-supermodule. However, when viewing a $\gg$-supermodule as a $\kk^{\CC}$-supermodule, we ignore the parity and treat it simply as a $\kk^{\CC}$-module. 

If $M$ is unitarizable, the action of $\kk^{\CC}$ is contravariant by the definition of the Hermitian form, and $M$ can be decomposed under $\kk^{\CC}$, which is equivalent to decomposing $M\vert_{\even}$ under $\kk^{\CC}$. This naturally leads to the definition of a $(\gg,\kk^{\CC})$-supermodule or a $(\even,\kk^{\CC})$-module.

\begin{definition}
 A complex $(\gg, \kk^{\CC})$-supermodule ($(\even,\kk^{\CC})$-module) is a complex $\gg$-supermodule ($\even$-module) $M$ which, as a $\kk^{\CC}$-module, decomposes as a direct sum of simple $\kk^{\CC}$-modules, \emph{i.e.}, $M$ is $\kk^{\CC}$-semisimple.
\end{definition}

An important class of $(\even,\kk^{\CC})$-modules consists of those arising from unitary irreducible representations of the universal covering Lie group of the Lie group $S(U(p,q)\times U(n))$, associated with $\mathfrak{(p)su}(p,q\vert n)_{\bar{0}}$. These modules exhibit an additional property that leads to the definition of Harish-Chandra (super)modules.

\begin{definition}
 A complex $(\gg, \kk^{\CC})$-supermodule ($(\even,\kk^{\CC})$-module) is called a \emph{Harish-Chandra supermodule (Harish-Chandra module)} if it is finitely generated and is locally finite as a $\kk^{\CC}$-module. 
\end{definition}

For highest weight $\gg$-supermodules, the conditions of local finiteness over $\kk^{\CC}$ and finite generation are redundant.

\begin{proposition}[{\cite[Proposition 2.8]{Carmeli_Fioresi_Varadarajan_HW}}] \label{prop::HW_HC_Supermodules}
Let $M$ be a highest weight $\gg$-supermodule with highest weight $\Lambda$ and highest weight vector $v_{\Lambda}$. The following assertions are equivalent:
\begin{enumerate}
 \item[a)] $\dim(\UE(\kk^{\CC})v_{\Lambda}) < \infty$.
 \item[b)] $M$ is a $(\gg,\kk^{\CC})$-supermodule.
 \item[c)] $M$ is a Harish-Chandra supermodule.
\end{enumerate}
If these assertions hold, then $\UE(\kk^{\CC})v_{\Lambda}$ is a simple $\kk^{\CC}$-module.
\end{proposition}

We show that any unitarizable highest weight $\gg$-supermodule is a Harish-Chandra supermodule. Recall that any unitarizable highest weight $\gg$-supermodule is $\even$-semisimple, and each $\even$-constituent is a unitarizable highest weight $\even$-module (disregarding parity), for which we have the following well-known lemma.

\begin{lemma}[{\cite[Lemma IX.3.10]{NeebHW}}]
 Any unitarizable highest weight $\even$-module is a Harish-Chandra module.
\end{lemma}

To establish that unitarizable simple $\gg$-supermodules are Harish-Chandra supermodules, we need the following lemma, which is straightforward but useful.

\begin{lemma} \label{lemm::HC_g_HC_ev}
 A (complex) finitely generated $\gg$-supermodule $M$ is a $(\gg,\kk^{\CC})$-supermodule if and only if $M\vert_{\even}$ is a $(\even,\kk^{\CC})$-module.
\end{lemma}

\begin{proof}
 The statement follows immediately, as $\UE(\gg)$ is a finitely generated $\UE(\even)$-module.
\end{proof}

We now obtain the following proposition.

\begin{proposition} \label{prop::unitarizable_are_HC_supermodules}
 Let $\HH$ be a unitarizable highest weight $\gg$-supermodule. Then $\HH$ is a Harish-Chandra supermodule.
\end{proposition}

The category of Harish-Chandra (super)modules is the full subcategory whose objects are the Harish-Chandra (super)modules and whose morphisms are the even $\even$-module and $\gg$-supermodule homomorphisms, respectively. In this article, all (super)modules are understood to lie in this category.

\subsection{Infinitesimal characters} \label{subsubsec::infinitesimal_characters} 
Unitarizable simple $\gg$-supermodules are of highest weight type, and hence admit infinitesimal characters. These are algebra homomorphisms $\chi : \mathfrak{Z}(\gg) \to \CC$, which can be explicitly described using the Harish-Chandra homomorphism for Lie superalgebras. To construct this homomorphism, we use the following isomorphism of super vector spaces:
\begin{equation} \label{eq::PBW_decomposition}
\mathfrak{U}(\gg) \cong \mathfrak{U}(\hh) \oplus (\nn^{-}\mathfrak{U}(\gg) + \mathfrak{U}(\gg)\nn^{+}),
\end{equation}
which follows directly from the PBW Theorem for $\gg$. The decomposition is stable under $\omega$, extended to $\UE(\gg)$. The associated projection $\pr: \mathfrak{U}(\gg) \to \mathfrak{U}(\hh)$ is called the \emph{Harish-Chandra projection}, whose restriction to $\mathfrak{Z}(\gg)$ defines an algebra homomorphism:
\begin{equation}
\pr\big\vert_{\mathfrak{Z}(\gg)}: \mathfrak{Z}(\gg) \to \mathfrak{U}(\hh) \cong \operatorname{S}(\hh).
\end{equation}
The \emph{Harish-Chandra homomorphism} $\operatorname{HC}: \mathfrak{Z}(\gg) \to S(\hh)$ is given by the composition of $\pr\big\vert_{\mathfrak{Z}(\gg)}$ with the twist $\zeta: \operatorname{S}(\hh)\to \operatorname{S}(\hh)$, defined by 
$\lambda(\zeta(f))\coloneqq (\lambda-\rho)(f)$ for all $f \in \operatorname{S}(\hh)$ and $\lambda\in\hh^*$, where $\rho$ denotes the Weyl vector corresponding to the chosen positive system $\Delta^{+}$. For any $\Lambda \in \hh^{\ast}$, the associated map 
\begin{equation}
\chi_{\Lambda}(z)\coloneqq (\Lambda+\rho)(\text{HC}(z))
\end{equation}
defines an algebra homomorphism $\chi_{\Lambda} : \mathfrak{Z}(\gg) \to \CC$.
A $\gg$-supermodule $M$ is said to admit an \emph{infinitesimal character} if there exists $\Lambda \in \hh^{\ast}$ such that every $z \in \mathfrak{Z}(\gg)$ acts on $M$ by scalar multiplication with $\chi_{\Lambda}(z) \in \CC$. The map $\chi_{\Lambda}$ is referred to as the \emph{infinitesimal character} of $M$.

Highest weight $\gg$-supermodules admit an infinitesimal character by Proposition~\ref{prop::Properties_Verma_Supermodule}. More precisely, the following holds:

\begin{lemma}[{\cite[Lemma 8.5.3]{musson2012lie}}] \label{lemm::action_Casimir}
 Let $M$ be a highest weight $\gg$-supermodule with highest weight $\Lambda$. Then $M$ admits an infinitesimal character $\chi_{\Lambda}$. Furthermore, the quadratic Casimir element $\Omega_{\gg} \in \mathfrak{Z}(\gg)$ acts by the scalar
 \[
 (\Lambda + 2\rho, \Lambda).
 \]
\end{lemma}

The Harish-Chandra homomorphism is an injective ring homomorphism. To describe its image, we define 
\begin{equation}
\operatorname{S}(\hh)^{W}\coloneqq \{ f\in \operatorname{S}(\hh) \mid w (\lambda)(f) = \lambda(f) \ \text{for all} \ w \in W, \lambda \in \hh^*\},
\end{equation}
and for any $\lambda\in\hh^*$, we introduce
\begin{equation}
A_{\lambda}\coloneqq \{ \alpha \in \Delta_{\bar{1}}^{+} \mid (\lambda+\rho,\alpha) = 0\}.
\end{equation}
The image of $\operatorname{HC}$ is then given by \cite{GorelikHC, KacHC}:
\begin{equation}
\Im(\operatorname{HC}) = \bigl\{ f \in \operatorname{S}(\hh)^{W} \mid (\lambda+t\alpha)(f) = \lambda(f) \ 
\text{for all $t\in \CC$, $\lambda\in\hh^*$, and $\alpha \in A_{\lambda-\rho}$}\bigr\}.
\end{equation}
As an immediate consequence, one obtains $\chi_{\Lambda} = \chi_{\Lambda'}$ whenever 
\begin{equation}\label{eq::relation_infinitesimal_character}
\Lambda' = w\biggl(\Lambda+\rho+\sum_{i=1}^{k}t_{i}\alpha_{i}\biggr)-\rho,
\end{equation}
where $w \in W$, $t_{i}\in \CC$, and $\alpha_{1},\dotsc,\alpha_{k} \in A_\Lambda$ are linearly independent odd isotropic roots satisfying $(\Lambda+\rho,\alpha_{i}) = 0$. This leads to the concept of \emph{atypicality}. A weight $\Lambda \in \hh^{\ast}$ is called \emph{typical} if $A_\Lambda=\emptyset$, \emph{i.e.},
$(\Lambda+\rho,\alpha)\neq 0$ for
all $\alpha \in \Delta_{\bar{1}}^{+}$. Otherwise $\Lambda$ is called \emph{atypical}.

\subsection{Algebraic realization and Shapovalov form} \label{subsec::Realization_Shapovalov}

In this section, we realize unitarizable highest weight $\gg$-supermodules as simple quotients of Verma supermodules to specify their structure and to construct a natural contravariant Hermitian form on them.

\subsubsection{Verma supermodules} We consider the fixed involutive Lie superalgebra $(\gg,\omega = \omega_{(-,+)})$, and the fixed \emph{Borel subalgebra} 
\begin{equation}
\bb = \hh \oplus \nn^{+}, \qquad \nn^{+}\coloneqq \bigoplus_{\alpha \in \Delta^{+}} \gg^{\alpha},
\end{equation}
with $\Delta^{+}\coloneqq \Delta^{+}_{\text{nst}}$ (see Section~\ref{subsubsec::Real_form}). The Borel satisfies $\gg = \bb + \bb^{\omega}$ for
$
\bb^{\omega}\coloneqq \{\omega(X) : X \in \bb\}.
$
We consider $\UE(\gg)$ as a right $\UE(\bb)$-supermodule with respect to right multiplication. For any $\Lambda \in \hh^{\ast}$, let $\CC_{\Lambda}$ be the one-dimensional $\UE(\bb)$-supermodule with trivial action of $\nn^{+}$ and weight $\Lambda$. The \emph{Verma supermodule} associated to $\Lambda \in \hh^{\ast}$ is the induced $\gg$-supermodule
\begin{equation} \label{eq::Verma_supermodules}
M^{\bb}(\Lambda)\coloneqq \UE(\gg) \otimes_{\UE(\bb)} \CC_{\Lambda}.
\end{equation}
The superscript $\bb$ will be omitted when the Borel subalgebra is fixed.

The well-known properties of Verma supermodules are summarized in the following proposition. The proof of the proposition follows standard arguments similar to those used for reductive Lie algebras and will therefore be omitted.

\begin{proposition}[{\cite[Chapter 8]{musson2012lie}}] \label{prop::Properties_Verma_Supermodule} 
 \begin{enumerate}
 \item[a)] The Verma supermodule $M(\Lambda)$ is a highest weight $\gg$-supermodule with highest weight $\Lambda$, and highest weight vector $[1 \otimes 1]$. In particular, $\UE(\nn^{-})[1 \otimes 1] = M(\Lambda)$.
 \item[b)] For each $\bb$-eigenvector $v_{\Lambda}$ of weight $\Lambda$ in a $\gg$-supermodule $M$, there exists a uniquely determined surjective $\gg$-supermodule morphism $M(\Lambda) \to M$ sending $[1\otimes 1]$ to $v_{\Lambda}$.
 \item[c)] $M(\Lambda)$ is a \emph{weight supermodule}, that is, $M(\Lambda)=\bigoplus_{\lambda \in \hh^{\ast}}M(\Lambda)^{\lambda}$, where 
$M(\Lambda)^{\lambda}\coloneqq \{m\in M(\Lambda) : Hm=\lambda(H)m \text{ for all } H\in \hh\}$. The space $M(\Lambda)^{\lambda}$ is 
the \emph{weight space} of weight $\lambda$, and $\dim(M(\Lambda)^{\lambda})$ its \emph{multiplicity}. 
$\dim(M(\Lambda)^{\lambda}) < \infty$ for all $\lambda$, with $\dim(M(\Lambda)^{\Lambda})=1$.
 \item[d)] $\End_{\gg}(M(\Lambda)) \cong \CC$.
 \item[e)] $M(\Lambda)$ has a unique maximal subsupermodule and unique simple quotient, denoted by $L(\Lambda)$. In particular, $M(\Lambda)$ is indecomposable.
 \item[f)] $M(\Lambda)$ has a finite Jordan--Hölder series.
 \end{enumerate}
\end{proposition}

The properties of Verma supermodules pass, in particular, to highest weight supermodules. 
Moreover, all simple highest weight supermodules with the same highest weight and the same parity are isomorphic.

Combining Theorem~\ref{thm::HW_property_M} with Proposition~\ref{prop::Properties_Verma_Supermodule}, we conclude that any unitarizable highest weight $\gg$-supermodule is the unique simple quotient of a Verma supermodule. In general, Verma supermodules admit a natural $\omega$-contravariant Hermitian form, known as the Shapovalov form. It descends to a $\omega$-contravariant form on the unique simple quotient, also referred to as \emph{Shapovalov form}. A highest weight $\gg$-supermodule is unitarizable if and only if the associated Shapovalov form is positive definite.

\subsubsection{Shapovalov form} \label{subsubsec::Shapovalov_form}
Each Verma supermodule carries a natural contravariant Hermitian form if the highest weight $\Lambda \in \hh^{\ast}$ satisfies 
\begin{equation}\Lambda(\omega(H)) = \overline{\Lambda(H)}, \qquad \forall H \in \hh.
\end{equation} An element $\Lambda \in \hh^{\ast}$ that meets this condition is called \emph{symmetric}. Notably, all highest weights of unitarizable highest weight $\gg$-supermodules are symmetric. Indeed, let $\HH$ be a unitarizable highest weight $\gg$-supermodule with highest weight $\Lambda$ and highest weight vector $v_{\Lambda}$. Then, for any $H \in \hh$,
\begin{equation}
\langle H v_{\Lambda}, v_{\Lambda} \rangle = \langle v_{\Lambda}, \omega(H) v_{\Lambda} \rangle \quad \Leftrightarrow \quad \overline{\Lambda(H)} = \Lambda(\omega(H)).
\end{equation}
Recall that the decomposition \eqref{eq::PBW_decomposition} is stable under $\omega$. Let $\pr: \UE(\gg) \to \UE(\hh)$ be the Harish-Chandra projection from above. For any $\Lambda \in \hh^{\ast}$, we define a bilinear map on $\UE(\gg)$:
\begin{equation}
 (X,Y)_{\Lambda}\coloneqq \Lambda(\pr(\omega(X)Y)), \qquad X,Y \in \UE(\gg).
\end{equation}

\begin{lemma}
 Let $\Lambda \in \hh^{\ast}$ be symmetric. Then the following assertions hold:
 \begin{enumerate}
 \item[a)] $(\cdot,\cdot)_{\Lambda}$ is conjugate-linear in the first coordinate and linear in the second coordinate. 
 \item[b)] $(X,Y)_{\Lambda} = \overline{(Y,X)}_{\Lambda}$ for all $X,Y \in \UE(\gg)$.
 \item[c)] $(ZX,Y)_{\Lambda} = (X, \omega(Z)Y)_{\Lambda}$ for any $X, Y, Z \in \mathfrak{U}(\gg)$.
 \end{enumerate}
\end{lemma}

\begin{proof} a) and c) follow by definition and the fact that $\omega$ is conjugate-linear. For b), let $X,Y \in \UE(\gg)$, and note that $\omega(h)=\overline{h}$ for all $h \in \hh$ using \eqref{eq::form_omega_plus_omega_minus}. Using $\omega$-stability of the decomposition and $\overline{\Lambda(H)}=\Lambda(\omega(H))$ for all $H \in \hh$, one has \eqref{eq::PBW_decomposition}
 \[
 \begin{split}
(X,Y)_{\Lambda} = \Lambda(\pr(\omega(X)Y)) = \overline{\Lambda(\omega(\pr(\omega(X)Y)))} = \overline{\Lambda(\omega(Y)X)} = \overline{(Y,X)}_{\Lambda}. \qedhere
 \end{split}
 \]
\end{proof}

Since $M(\Lambda)=\UE(\gg)[1\otimes 1]$, the form $(\cdot,\cdot)_{\Lambda}$ induces a well-defined contravariant Hermitian form on $M(\Lambda)$. This form is the \emph{Shapovalov form}, denoted by $\langle\cdot,\cdot\rangle_{M(\Lambda)}$.

\begin{proposition} \label{prop::uniqueness_Shapovalov_form}
 If $\Lambda \in \hh^{\ast}$ is symmetric, then there exists a unique $\omega$-contravariant Hermitian form on $M(\Lambda)$ satisfying $\langle [1\otimes 1],[1\otimes 1]\rangle_{M(\Lambda)} = 1$. All other contravariant Hermitian forms are real multiples of this form.
\end{proposition}

\begin{proof}
 To complete the argument, we must show uniqueness. Thus it suffices to verify that any contravariant Hermitian form $\langle\cdot,\cdot\rangle$ on $M(\Lambda)$ with $\langle[1\otimes 1],[1\otimes 1]\rangle=0$ is necessarily trivial. This follows from the computation below:
\begin{equation*}
\begin{split}
 \langle M(\Lambda), M(\Lambda) \rangle &= \langle \UE(\gg)[1\otimes 1], \UE(\gg)[1\otimes 1] \rangle
 = \langle [1\otimes 1], \UE(\gg)[1\otimes 1] \rangle \\
 &= \langle [1\otimes 1], \UE(\bb^{\omega})[1\otimes 1] \rangle 
 = \langle \UE(\bb)[1\otimes 1], [1\otimes 1] \rangle \\
 &\subset \CC \langle [1\otimes 1], [1\otimes 1] \rangle = 0. \qedhere
 \end{split} 
\end{equation*} 
\end{proof}

By the construction of the Shapovalov form, weight spaces of different weights with respect to $\hh$ are orthogonal. We conclude the following corollary.

\begin{corollary}
 If $\Lambda$ is symmetric, the maximal proper subsupermodule of $M(\Lambda)$ coincides with the radical $R$ of the Shapovalov form $\langle \cdot, \cdot \rangle_{M(\Lambda)}.$ In particular, $L(\Lambda) \cong M(\Lambda)/R$.
\end{corollary}

Let $M$ be a highest weight $\gg$-supermodule with highest weight $\Lambda$ and highest weight vector $v_{\Lambda}$. Assume that $\Lambda$ is symmetric. By Proposition~\ref{prop::Properties_Verma_Supermodule}, there is a $\gg$-supermodule homomorphism $q: M(\Lambda) \to M$ that maps $[1\otimes 1]$ to $v_{\Lambda}$. Since $\Lambda$ is symmetric, the maximal subsupermodule of $M(\Lambda)$ is the same as the radical $R$ of the Shapovalov form. As a result, the kernel of $q$ must lie within $R$. This leads us to the following proposition. 

\begin{proposition}\label{prop::universal_property_Shapovalov}
 $M$ carries a nonzero unique $\omega$-contravariant consistent Hermitian form $\langle \cdot,\cdot \rangle$ such that $\langle v_{\Lambda},v_{\Lambda}\rangle = 1$. The form is non-degenerate if and only if $M$ is simple, in which case $M \cong L(\Lambda)\coloneqq M(\Lambda) / R$, and the space of $\bb$-eigenvectors of weight $\Lambda$ in $L(\Lambda)$ is one-dimensional.
\end{proposition}

\begin{proof}
 It remains to show that the form is non-degenerate if and only if $M$ is simple. It is clear that the Shapovalov form induces a non-degenerate form on $L(\Lambda)=M(\Lambda)/R$.

 If, conversely, $M$ carries a non-degenerate Hermitian form, then $\ker(q)$ cannot be properly contained in $R$. Consequently, $R= \ker(q)$ and $M \cong L(\Lambda)$.
\end{proof}

By abuse of notation, we also call the unique form on M the \emph{Shapovalov form}. In particular, $M$ is unitarizable iff the Shapovalov form is positive definite. We identify $M \cong L(\Lambda)$ whenever $M$ is simple. 

\subsection{\texorpdfstring{$\even$}{even}-Constituents} Let $\HH$ be a unitarizable highest weight $\gg$-supermodule of highest weight $\Lambda$, so that $\HH \cong L(\Lambda)$. By Proposition~\ref{CompletelyReducible}, $\HH$ is completely reducible as a $\gg_{\bar{0}}$-module, with parity suppressed. We denote the occurring unitarizable highest weight $\even$-modules by $L_{0}(\mu)$ for $\mu \in \hh^{\ast}$.

\begin{theorem}[{\cite[Theorem 10.4.5]{musson2012lie}, \cite[Theorem 2.5]{jakobsen1994full}}] \label{thm::even_constituents_HW} Let $\HH$ be a unitarizable highest weight $\gg$-supermodule with highest weight $\Lambda$. Then $\HH$ decomposes into a finite direct sum of unitarizable highest weight $\even$-modules, called $\even$-constituents; each of them has highest weight $\Lambda-\gamma$, where $\gamma$ is a sum of pairwise distinct odd positive roots.
\end{theorem}

We proceed to describe the $\even$-constituents of $L(\Lambda)$.

\begin{lemma} \label{lemm::Shapovalov_TP} Let $\gg_{\bar{0}}\coloneqq \gg_{1} \oplus \gg_{2}$ be a direct sum of reductive Lie algebras with root space decompositions. Let $L_{0}(\lambda;\gg_{1})$ and $L_{0}(\mu;\gg_{2})$ be two simple highest weight modules with highest weight vectors $v_{\lambda}$ and $v_{\mu}$, respectively. Assume that $\lambda$ and $\mu$ are symmetric. Define the $\even$-module $V\coloneqq L_{0}(\lambda; \gg_{1}) \boxtimes L_{0}(\mu;\gg_{2})$. Then the following assertions hold:
 \begin{enumerate}
 \item[a)] The tensor product of the Shapovalov forms on both factors defines a non-degenerate contravariant Hermitian form on $V$.
 \item[b)] $v_{\lambda} \otimes v_{\mu}$ is a primitive element in $V$, that is, a nonzero $\bb_{\bar{0}}$-eigenvector.
 \item[c)] $V \cong L_{0}(\lambda + \mu)$ if and only if $v_{\lambda} \otimes v_{\mu}$ is cyclic in $V$.
 \end{enumerate}
\end{lemma}
\begin{proof}
 We denote the Shapovalov forms on $L_{0}(\lambda;\gg_{1})$ and $L_{0}(\mu;\gg_{2})$ by $\langle \cdot, \cdot \rangle_{\lambda}$ and $\langle \cdot, \cdot \rangle_{\mu}$, respectively. Then $V$ has a contravariant Hermitian form 
 $$
 \langle v_{1} \otimes w_{1}, v_{2} \otimes w_{2} \rangle = \langle v_{1}, v_{2} \rangle_{\lambda} \langle w_{1}, w_{2} \rangle_{\mu}.
 $$
 To see that $\langle \cdot, \cdot \rangle$ is non-degenerate, consider $x\coloneqq \sum_{j=1}^{k} v_{j} \otimes w_{j}$ with $\langle x, y \rangle = 0$ for all $y \in V$. Without loss of generality, assume that $v_{1}, \ldots, v_{k}$ are linearly independent. Since $L_{0}(\lambda;\gg_{1})$ is simple, $\langle \cdot, \cdot \rangle_{\lambda}$ is non-degenerate, and we find for each $1 \leq j \leq k$ an element $y_{j} \in L_{0}(\lambda;\gg_{1})$ such that $\langle v_{i}, y_{j} \rangle_{\lambda} = \delta_{ij}$. Consequently, $\langle x, y_{j} \otimes w \rangle = \langle w_{j}, w \rangle_{\mu} = 0$, so $w_{j} = 0$ as $\langle \cdot, \cdot \rangle_{\mu}$ is non-degenerate. This concludes the proof for assertion a). 

 For assertions b) and c), note that the set of weights of $V$, denoted by $P_{V}$, is the sum of the sets of weights for $L_{0}(\lambda;\gg_{1})$ and $L_{0}(\mu; \gg_{2})$, \emph{i.e.}, $P_{V} \subset \lambda + \mu - \ZZ_{+}[\Delta_{\bar{0}}^{+}]$ with $\dim V^{\lambda + \mu} = 1$. Assertion b) is now clear, while assertion c) follows directly from a).
\end{proof}

\begin{proposition}
 Let $\gg_{\bar{0}}\coloneqq \gg_{1} \oplus \gg_{2}$ be the direct sum of Lie algebras with root space decompositions. Let $\hh\coloneqq \hh_{1} \oplus \hh_{2}$, where $\hh_{1}$ and $\hh_{2}$ denote the Cartan subalgebras of $\gg_{1}$ and $\gg_{2}$, respectively. Identify the root systems $\Delta_{1}$ and $\Delta_{2}$ with subsets of $\Delta_{\bar{0}}$, and let $\lambda = (\lambda_{1}, \lambda_{2}) \in \hh^{\ast} \cong \hh_{1}^{\ast} \times \hh_{2}^{\ast}$. Then 
 $$
 L_{0}(\lambda) \cong L_{0}(\lambda_{1}; \gg_{1}) \boxtimes L_{0}(\lambda_{2}; \gg_{2}).
 $$
\end{proposition}

\begin{proof}
 The positive systems for $\gg_{1}$ and $\gg_{2}$ are $\Delta_{1}^{+}\coloneqq \Delta_{1} \cap \Delta^{+}_{\bar{0}}$ and $\Delta_{2}^{+}\coloneqq \Delta_{2} \cap \Delta^{+}_{\bar{0}}$. The projections $p_{j}: \even \to \gg_{j}$ are Lie algebra morphisms, allowing us to identify $L_{0}(\lambda_{j}; \gg_{j})$ with $L_{0}(\lambda_{j})$. Define $v_{\lambda}\coloneqq v_{\lambda_{1}} \otimes v_{\lambda_{2}}$, where $v_{\lambda_{j}}$ are the highest weight vectors for $L_{0}(\lambda_{j})$. We compute
 \begin{equation*}
 \begin{aligned}
 \UE(\even) v_{\lambda} &= (\UE(\gg_{1}) v_{\lambda_{1}}) \otimes (\UE(\gg_{2}) v_{\lambda_{2}}) = L_{0}(\lambda_{1}; \gg_{1}) \boxtimes L_{0}(\lambda_{2}; \gg_{2}),
 \end{aligned}
 \end{equation*}
 \emph{i.e.}, $v_{\lambda}$ is cyclic, and the statement follows with Lemma~\ref{lemm::Shapovalov_TP}.
\end{proof}

\begin{corollary}
 Any unitarizable simple highest weight $\even$-module $L_{0}(\mu)$ is given by the outer tensor product of unitarizable highest weight modules over $\su(p,q)^{\CC}$, $\su(n)^{\CC}$, and $\u(1)^{\CC}$, respectively. If $m=n$, the $\mathfrak{u}(1)^{\CC}$-module is trivial.
\end{corollary}

Simple modules over $\su(n)$ and $\u(1)$ are finite-dimensional, hence highest weight and unitarizable; they are parametrized by dominant integral weights. 
The unitarizable highest weight modules of $\su(p,q)$ are described in \cite{enright1983classification, Jakobsen_Hermitian}. 
For Dirac cohomology, the relevant $\gg_{\bar{0}}$-modules are those attached to the relative holomorphic discrete series of the underlying Lie group.

\subsubsection{Relative holomorphic discrete series modules over \texorpdfstring{$\even$}{}} \label{subsubsec::relative_holomorphic} Given a unitarizable simple $\gg$-supermodule $\HH$, upon forgetting parity, we can consider its decomposition under $\even$ as a direct sum of unitarizable highest weight $\even$-modules. Among the unitarizable highest weight $\even$-constituents, there is a particularly important class known as \emph{relative holomorphic discrete series modules}. These modules are distinguished by the property that, for each $\even$-constituent, one can assign a formal dimension, which serves as an analog of the dimension for infinite-dimensional modules \cite{SchmidtGeneralizedSuperdimension}. Their explicit definition and construction are extensive and require additional terminology, so we only provide a parameterization of the highest weights. For an explicit discussion, we refer to \cite{Neeb}.

Let $\HH$ be a non-trivial unitarizable $\gg$-supermodule. Then $\HH$ is $\even$-semisimple and decomposes into a direct sum of unitarizable highest weight $\even$-supermodules $L_{0}(\mu)$.
We denote the associated real form by $\rform$ and the associated analytic reductive Lie group by $G^{\RR}_{\bar{0}}$, with complexification $G_{\bar{0}}$. Let $G\coloneqq \tilde{G}^{\RR}_{\bar{0}}$ be the universal covering Lie group, and let $T$ be the analytic subgroup of $G$ associated with $\hh^{\RR}$, the real form of the Cartan subalgebra.

We parametrize the unitarizable holomorphic supermodules in terms of the character group $X^{\ast}(T^{\CC})$ of the complexification $T^{\CC}$ of $T$ and Weyl chambers. The coset $X^{\ast}(T^{\CC})+\rho_{\bar{0}}$ of $X^{\ast}(T^{\CC})\otimes \RR$ is independent of the choice of positive roots. We call an element $\lambda$ in $X^{\ast}(T^{\CC})\otimes \RR$ \emph{regular} if no dual root is orthogonal to $\lambda$; otherwise, $\lambda$ is called \emph{singular}. Let $(X^{\ast}(T^{\CC})\otimes \RR)^{\reg}$ denote the set of regular elements in $X^{\ast}(T^{\CC})\otimes \RR$, and set 
$(X^{\ast}(T^{\CC})+\rho_{\bar{0}})^{\reg}\coloneqq (X^{\ast}(T^{\CC})+\rho_{\bar{0}})\cap (X^{\ast}(T^{\CC})\otimes \RR)^{\reg}$. A \emph{Weyl chamber} of $G$ is a connected component of $(X^{\ast}(T^{\CC})\otimes \RR)^{\reg}$. The Weyl chambers are in one-to-one correspondence with the systems of positive roots for $(\even,\hh)$.

Recall the definition of $\Delta_{nc}^{+}$ in Section~\ref{sec::realforms}. We define a Weyl chamber $C$ to be \emph{holomorphic} if for any $\lambda'$ in the interior of $C$ we have $(\lambda,\alpha)<0$ for all $\alpha \in \Delta_{nc}^{+}$. Then there are exactly $\vert W_{c}\vert$ holomorphic Weyl chambers, forming a single orbit for the action of $W_{c}$. A \emph{holomorphic Harish-Chandra parameter} is a pair $(\lambda,C)$, where $C$ is a holomorphic Weyl chamber and $\lambda \in (X^{\ast}(T^{\CC})+\rho_{\bar{0}})^{\reg}\cap C$.

There is a bijective correspondence between the isomorphism classes of relative holomorphic discrete series supermodules and the quotient $(X^{\ast}(T^{\CC}) + \rho_{\bar{0}})^{\reg} / W_{c}$, where the Weyl group $W_{c}$ acts naturally. Note that we have fixed the standard positive system and, consequently, the standard Weyl chamber $C_{\text{st}}$. Therefore, the weights in $(X^{\ast}(T^{\CC})+\rho_{\bar{0}})^{\reg} \cap C_{\text{st}}$ parameterize the relative holomorphic discrete series supermodules. For convenience, we rewrite the parameterization by a shift of $\rho_{\bar{0}}$
\begin{equation}
 \Dis\coloneqq \{ \Lambda \in (X^{\ast}(T^{\CC})\otimes \RR)^{\reg} : (\Lambda+\rho_{\bar{0}}, \alpha)<0 \ \text{for all} \ \alpha \in \Delta_{nc}^{+}\}.
\end{equation}
For the fixed standard positive system, $\Dis$ consists precisely of those $\Lambda\in (X^{\ast}(T^{\CC})\otimes\RR)^{\reg}$ satisfying
\begin{equation}
(\Lambda+\rho_{\bar0},\epsilon_{1}-\epsilon_{m})<0.
\end{equation}
In particular, $\Dis$ is a subset of the unitarizable highest weights of $\even$ consisting of those which are the unique highest weights of unitarizable $\even$-modules in their $W$-linkage class.

\begin{lemma} \label{lemm::D_Weyl_Orbit}
 Let $L_{0}(\mu)$ be a unitarizable highest weight $\even$-module with highest weight $\mu\in \Dis$. Then $\mu$ is the unique element in its $W$-linkage class that appears as the highest weight of a unitarizable simple $\even$-module with respect to the standard positive system.
\end{lemma}

Other consequences of the definition of $\Dis$ are collected in the following lemma.

\begin{lemma}\label{lemm::direct_consequences_Dis} Assume $\Lambda\in\Dis$. Then:
\begin{enumerate}
\item[a)] Every highest weight $\even$-module of highest weight $\Lambda$ is simple.
\item[b)] $L_{0}(\Lambda)$ is a Harish-Chandra module.
\item[c)] In the category of Harish-Chandra modules, one has
\[
\Ext^{1}(L_{0}(\Lambda),L_{0}(\Lambda'))=0
\qquad \forall\,\Lambda,\Lambda'\in\Dis.
\]
\end{enumerate}
\end{lemma}

\section{Dirac operators and Dirac cohomology: Revisited} \label{ch::Dirac_operator_Dirac_cohomology}

We introduce the \emph{Dirac operator} and \emph{Dirac cohomology} for simple Lie superalgebras $\gg\coloneqq A(m-1\vert n-1)$, as first described in \cite{huang2005dirac}. We also generalize these results by studying functorial properties of Dirac cohomology.

Recall the definition of the supertrace form $(\cdot, \cdot)$ on $\gg = \even \oplus \odd$ (see \eqref{eq::supertrace_form}). We modify it for later convenience by $B(\cdot, \cdot)\coloneqq \frac{1}{2}(\cdot, \cdot)$. Recall that $B(\cdot,\cdot)$ is non-degenerate, supersymmetric, and invariant.

\subsection{Dirac operator} \label{subsec::Dirac_operator} Quadratic Dirac operators and their relative counterparts are defined here in the setting of the supersymmetric analogue of the quantum Weil algebra introduced in \cite{Meinrenken}. Within this framework, the Dirac operator of \cite{huang2005dirac} appears as the relative Dirac operator associated with the pair $(\gg,\gg_{\bar{0}})$.

\subsubsection{Weil algebra, quantum Weil algebra and Weyl algebra} We begin by describing the quantum Weil algebra for a general quadratic Lie superalgebra $\ll$ and for embeddings of quadratic Lie subsuperalgebras. We then make these constructions explicit for the quadratic pair $(\gg,\even)$.

Let $V$ be a super vector space with a supersymmetric and non-degenerate bilinear form $\langle\cdot,\cdot\rangle$, and let $T(V)$ denote the tensor superalgebra over $V$, equipped with the $\ZZ_{2}$-grading induced by $V$. 
We endow $T(V)$ with its natural $\ZZ$-grading, and write $T(V)=\bigoplus_{n\geq 0}T^{n}(V)$ with $T^{0}(V)=\CC$. 
The following associative unital superalgebras arise naturally as quotients $T(V)/I$, where $I$ is a two-sided homogeneous ideal of $T(V)$. The $\ZZ_{2}$-grading is the one induced by the grading of $V$.

\begin{enumerate}
 \item \emph{Exterior superalgebra} $\bigwedge(V)$ with $I_{\wedge}$ generated by the elements
\begin{equation*}
 v \otimes w + (-1)^{p(v)p(w)}\, w \otimes v, \qquad v,w \in V.
\end{equation*}
Its multiplication will be written $\wedge$ and called \emph{exterior multiplication}. Moreover, as the ideal $I_{\wedge}$ is homogeneous, we have a $\ZZ$-grading 
\begin{equation*}
 \bigwedge(V)=\bigoplus_{n=0}^{\infty}\bigwedge\nolimits ^{n}(V), \qquad \bigwedge\nolimits ^{n}(V)\coloneqq T^{n}(V)/(I_{\wedge}\cap T^{n}(V)).
\end{equation*}

\item \emph{Symmetric superalgebra} $S(V)$ with $I_{S}$ generated by the elements
\[
v \otimes w - (-1)^{p(v)p(w)}w \otimes v, \qquad v,w \in V.
\]
Similarly to the exterior superalgebra, $S(V)$ has a $\ZZ$-grading $S(V) = \bigoplus_{n=0}^{\infty}S^{n}(V)$ induced by the $\ZZ$-grading of $T(V)$.
\item \emph{Clifford superalgebra} $\Cl(V)$ with $I_{\Cl}$ generated by the elements
\begin{equation*}
v \otimes w + (-1)^{p(v)p(w)} w \otimes v - 2\langle v,w \rangle 1_{T(V)}, \qquad v,w \in V.
\end{equation*}
\end{enumerate}

If $(\ll=V,B)$ is a quadratic Lie superalgebra, we define the \emph{Weil algebra} and \emph{quantum Weil algebra} generalizing the notion in \cite{Meinrenken} to the super setting.

\begin{definition}
 Let $\ll$ be a quadratic Lie superalgebra. We call $W(\ll)\coloneqq S(\ll) \otimes \bigwedge(\ll)$ the \emph{Weil algebra} and $\mathcal{W}(\ll)\coloneqq \UE(\ll) \otimes \Cl(\ll)$ the \emph{quantum Weil algebra}. $W(\ll)$ and $\mathcal{W}(\ll)$ are defined with respect to the $\ZZ_{2}$-graded tensor product over $\CC$.
\end{definition}

The quantum Weil algebra $\WW(\ll)$ is generated by the elements $x\otimes 1$ and $1\otimes x$, where $x\in\ll$. It carries a $\ZZ_2\times\ZZ_2$-grading given by
\begin{equation}\label{def::bidegree}
\deg(1\otimes x)=(\bar1,p(x)),\qquad \deg(x\otimes 1)=(\bar0,p(x)).
\end{equation}
For $\alpha=(\bar a,\bar b)$ and $\beta=(\bar c,\bar d)$ in $\ZZ_2\times\ZZ_2$, set
\begin{equation}
\chi(\alpha,\beta)\coloneqq (-1)^{\bar a\bar c+\bar b\bar d}.
\end{equation}
If $A,B\in\WW(\ll)$ are homogeneous of bidegrees $\alpha$ and $\beta$, define
\begin{equation}\label{eq::colour_commutator}
[A,B]_{\WW}\coloneqq AB-\chi(\alpha,\beta)BA,
\end{equation}
and extend $\CC$-bilinearly. The total degree is the sum of the two components of the bidegree in $\ZZ_2$. In particular, if $A\in\WW(\ll)$ has total degree $\bar1$, then $A^2=\frac12[A,A]_{\WW}$. Thus $\WW(\ll)$ becomes a colour Lie algebra (\emph{cf.}~\cite{Ree}), called the \emph{colour quantum Weil algebra}. Its relation to the Weil algebra is given by quantization and PBW supersymmetrization.

\begin{theorem} \label{thm::quantization}
 Let $\ll$ be a quadratic Lie superalgebra. Then there exists an isomorphism of $\ZZ$-graded super vector spaces $\mathcal{Q} : W(\ll) \to \mathcal{W}(\ll)$ given by $\mathcal{Q} = q \otimes \operatorname{sym}$, where 
 \begin{enumerate}
 \item[(i)] $q : \bigwedge(\ll) \to \Cl(\ll)$ is the quantization map: 
\[
q(x_{1} \wedge \ldots \wedge x_{k})\coloneqq \frac{1}{k!}\sum_{\sigma \in S_{k}} p(\sigma;x_{1}, \ldots, x_{k}) x_{\sigma(1)}\ldots x_{\sigma(k)},
\]
with
\[
p(\sigma; x_{1},\ldots, x_{k}) = \sgn(\sigma)\prod_{1 \leq i < j \leq k, \ \sigma^{-1}(i) > \sigma^{-1}(j)} (-1)^{p(x_{i})p(x_{j})}.
\]
\item[(ii)] $\operatorname{sym}$ is the PBW supersymmetrization 
\begin{equation*}
 \operatorname{sym}(x_{1}\ldots x_{k})\coloneqq \frac{1}{k!}\sum_{\sigma \in S_{k}} p'(\sigma;x_{1}, \ldots, x_{k}) x_{\sigma(1)}\ldots x_{\sigma(k)},
\end{equation*}
where $p'(\sigma; x_{1},\ldots, x_{k})
$ is given by
\begin{equation*}
p'(\sigma; x_{1},\ldots, x_{k}) = \prod_{1 \leq i < j \leq k, \ \sigma^{-1}(i) > \sigma^{-1}(j)} (-1)^{p(x_{i})p(x_{j})}.
\end{equation*}
 \end{enumerate}
\end{theorem}

\begin{proof}
 The quantization map $q : \bigwedge(\ll) \to \Cl(\ll)$ is an isomorphism of $\ZZ$-graded super vector spaces \cite{generalized_Clifford}, and the PBW supersymmetrization $\operatorname{sym} : S(\ll) \to \UE(\ll)$ is an isomorphism of $\ll$-supermodules by \cite[Theorem 6.4.4]{musson2012lie}. By definition of $W(\ll)$ and $\mathcal{W}(\ll)$, the claim follows. 
\end{proof}

We realize $\ll$ inside $\WW(\ll)$. For every $x\in\ll$, one has $\ad_x\in\osp(\ll)$ by the invariance and supersymmetry of $B$, where
\begin{equation}
\osp(\ll)\coloneqq \left\{T\in\End_{\CC}(\gg)\,:\, B(Tx,y)+(-1)^{p(T)p(x)}B(x,Ty)=0 \quad \forall\,x,y\in\ll\right\}
\end{equation}
is the orthosymplectic Lie superalgebra over $\ll$ equipped with canonical supercommutator. $\osp(\ll)$ identifies with $\bigwedge^2(\ll)$ via (\emph{cf.}~\cite[Prop.~2.13]{Meyer})
\begin{equation}
\lambda(T)=-\frac12\sum_a T(e^a)\wedge e_a,
\end{equation}
where $\{e_a\}$ is a basis of $\ll$ and $\{e^a\}$ its $B$-dual basis.
Its quantization
$
\gamma'(x)\coloneqq q\bigl(\lambda(\ad_{x})\bigr)
$
is naturally an element of $\Cl(\ll)$, and we obtain a map
\begin{equation} \label{eq::general_embedding}
\gamma^{\WW} : \ll\to \WW(\ll),
\qquad
\gamma^{\WW}(x)\coloneqq x\otimes 1-\tfrac{1}{2}(1\otimes \gamma'(x))
\end{equation}
so that $\gamma^{\WW}([x,y]_{\ll})=[\gamma^{\WW}(x),\gamma^{\WW}(y)]_{\WW}$ for all $x,y \in \ll$. We shall make this construction explicit below in a certain example.

Next, we incorporate subsuperalgebras into the construction. Let $\uu\subseteq\ll$ be a quadratic Lie subsuperalgebra, with $B_{\uu}\coloneqq B|_{\uu}$. Then $(\ll,\uu)$ is called a \emph{quadratic pair}. For a quadratic pair $(\ll,\uu)$, one has the orthogonal decomposition
\begin{equation}
\ll=\uu\oplus\pp,\qquad \pp\coloneqq \uu^{\perp},
\end{equation}
with respect to $B$. The restriction $B_{\pp}\coloneqq B|_{\pp}$ is non-degenerate. Hence one may form the Clifford superalgebra $\Cl(\pp)\coloneqq \Cl(\pp;B_{\pp})$. For example, one may take $\uu=\ll_{\bar 0}$ and $\pp=\ll_{\bar 1}$.

We define a map $j : \mathcal{W}(\uu) \to \mathcal{W}(\ll)$ by sending the generators $x \otimes 1$ and $1 \otimes x$ for $x \in \uu$ to the corresponding generators in $\mathcal{W}(\ll)$, that is, 
\begin{equation} \label{eq::definition_j}
 j(x\otimes 1)\coloneqq x \otimes 1, \qquad j(1\otimes x)\coloneqq 1 \otimes x, \qquad x \in \uu.
\end{equation}
This evidently defines an injective homomorphism preserving the colour commutator \eqref{eq::colour_commutator}. By construction of the map $j$, the following lemma is immediate.

\begin{lemma}
 The commutant of $j(\mathcal{W}(\uu))$ in $\mathcal{W}(\ll)$ is the space of $\uu$-invariants of $\mathcal{W}(\ll)$.
\end{lemma}

In what follows, we make these constructions explicit for $\gg=\ll$, $\even=\uu$, and $\odd=\pp$. In particular, we describe an explicit embedding of $\even$ into $\WW(\gg)$, which will play a central role in the rest of this article. We start by constructing an embedding $\mathcal{W}(\gg_{\bar{0}})\hookrightarrow \mathcal{W}(\gg)$. By the universal property of $\UE(\gg)$, there is a canonical embedding of $\UE(\even)$ into $\UE(\gg)$. Next, we embed $\Cl(\even)$ into $\Cl(\odd)$.

 The form $B$ restricts to a symplectic form on $\odd$ which allows us to identify $\Cl(\odd)$ with the \emph{Weyl algebra} over $\odd$. We fix two Lagrangian subspaces of $(\odd, B\vert_{\odd}(\cdot,\cdot))$, spanned by $\{x_{1}, \ldots, x_{mn}\}$ and $\{\partial_{1}, \ldots, \partial_{mn}\}$, chosen compatibly with the real form $\mathfrak{(p)su}(p,q \vert 0,n)$, that is, $B\vert_{\odd}(x_{i}, x_{j}) = B\vert_{\odd}(\partial_{i}, \partial_{j}) =0$, $B\vert_{\odd}(\partial_{i}, x_{j}) = \frac{1}{2}\delta_{ij}$ and $\omega(x_{i}) = - \partial_{i}$ for all $1 \leq i,j \leq mn$. Precisely, in accordance with the fixed non-standard positive system $\Delta^{+}$, we have:
\begin{equation}
\begin{split}
 \partial_{(i-1)n + (j-m)} &= 
 \begin{cases}
 E_{ij} \ \ \qquad &\text{for} \ 1 \leq i \leq p, \ m+1 \leq j \leq m+n, \\
 E_{ji} \qquad &\text{for} \ p+1 \leq i \leq m, \ m+1 \leq j \leq m+n,
 \end{cases} \\
 x_{(i-1)n + (j-m)} &= 
 \begin{cases}
 E_{ji} \qquad &\text{for} \ 1 \leq i \leq p, \ m+1 \leq j \leq m+n, \\
 -E_{ij} \qquad &\text{for} \ p+1 \leq i \leq m, \ m+1 \leq j \leq m+n.
 \end{cases}
\end{split}
\end{equation}
Here, the $x_{k}$'s span the space of negative odd roots $\nn_{\bar{1}}^{-} = \pp_{2} \oplus \qq_{1}$ and the $\partial_{k}$'s span the space of positive odd roots $\nn_{\bar{1}}^{+} = \pp_{1} \oplus \qq_{2}$ (see Section~\ref{subsubsec::Real_form}).

\begin{remark} \label{rmk::Dirac_standard}
It is important to note that, for the Lagrangian subspaces $\ll^{+}\coloneqq \pp_{1} \oplus \pp_{2}$ and $\ll^{-}\coloneqq \qq_{1} \oplus \qq_{2}$ associated with the standard positive system, there does not exist a basis $x_{1}, \ldots, x_{mn}$ of $\ll^{-}$ and $\partial_{1}, \ldots, \partial_{mn}$ of $\ll^{+}$ that simultaneously satisfies $B(\partial_{i}, x_{j}) = \frac{1}{2} \delta_{ij}$ and $\omega(x_{i}) = -\partial_{i}$ for all $1 \leq i,j \leq mn$.
\end{remark}

The \emph{Weyl algebra} $\Weyl$ is defined as the quotient $T(\odd)/I_{\mathcal{W}}$, where $I_{\mathcal{W}}$ is the two-sided ideal generated by all elements of the form $v \otimes w - w \otimes v - 2\bil(v, w)$ for $v, w \in \odd$. As $\odd$ is odd, we have by definition $\Cl(\odd) = \Weyl$. With this notation, the Weyl algebra $\Weyl$ over $\odd$ is generated by $\partial_{i}$ and $x_{j}$, \emph{i.e.}, it can be identified with the algebra of differential operators with polynomial coefficients in the variables $x_{1}, \dotsc, x_{mn}$, by identifying $\partial_{i}$ with the partial derivative $\partial/\partial x_{i}$ for all $i = 1, \dotsc, mn$. In particular, with the canonical Lie bracket $[\cdot, \cdot]_{\WW}$, the Weyl algebra $\Weyl$ is a Lie algebra with the following commutation relations:
\begin{equation}
 [x_{i}, x_{j}]_{\mathscr{W}} = 0, \qquad
 [\partial_{i}, \partial_{j}]_{\mathscr{W}} = 0, \qquad
 [\partial_{i}, x_{j}]_{\mathscr{W}} = \delta_{ij},
\end{equation}
for all $1 \leq i,j \leq mn$.

Since $\Weyl = \Cl(\odd)$, the Weyl algebra has a natural $\ZZ$-grading induced by the grading of $T(\odd)$. Similarly, the symmetric superalgebra $\operatorname{S}(\odd)$ has a natural $\ZZ$-grading $\operatorname{S}(\odd) = \bigoplus_{n\geq 0}S^{n}(\odd)$. As $\ZZ$-graded vector spaces, they are isomorphic, with isomorphism given by the symmetrization map
\begin{equation}
 \text{sym} : \operatorname{S}(\odd) \to \Weyl, \qquad \text{sym}(v_{1}\ldots v_{k})\coloneqq \tfrac{1}{k!}\sum_{\sigma \in S_{k}}v_{\sigma(1)}\ldots v_{\sigma(k)}.
\end{equation}

In $\operatorname{S}(\odd)$, the \emph{symmetric square} $S^{2}(\odd)$ is a Lie algebra spanned by $\{x_{i}x_{j}, \partial_{i}\partial_{j} : i \leq j\} \cup \{\partial_{i}x_{j}\}$. Under the isomorphism above, the image of $S^{2}(\odd)$ is the Lie algebra $\text{sym}(S^{2}(\odd))$ spanned by
 \begin{equation}
 \text{sym}(x_{i}x_{j})=x_{i}x_{j}, \qquad \text{sym}(\partial_{i}\partial_{j})=\partial_{i}\partial_{j}, \qquad \text{sym}(\partial_{i}x_{j})=\partial_{i}x_{j}-\frac{1}{2}\delta_{ij}
 \end{equation}
 in the basis $\{x_{i},\partial_{j}\}$. Now, if $\mathfrak{sp}(\odd)$ denotes the complex symplectic algebra over the vector space $\odd$ defined with respect to the symplectic form $B\vert_{\odd}(\cdot, \cdot)$, a direct calculation yields an isomorphism of Lie algebras 
 \begin{equation} \label{eq::isomorphism_of_Lie_algebras}
 (\mathfrak{sp}(\odd),[\cdot,\cdot]) \cong (\operatorname{sym}(S^{2}(\odd)),[\cdot,\cdot]_{W}).
 \end{equation}

 This identification yields a natural embedding of $\even$ into $(\Weyl,[\cdot,\cdot]_{W})$. 
Indeed, the adjoint action of $\even$ on $\odd$, induced by the super Lie bracket $[\cdot,\cdot]$ of $\gg$, defines a Lie algebra homomorphism
\begin{equation}
\nu:\even \longrightarrow \mathfrak{sp}(\odd),
\end{equation}
since $B$ is invariant. 
For $m\neq n$ the kernel of $\nu$ is trivial, whereas for $m=n$ it consists precisely of the scalar multiples of the identity matrix. This proves the following lemma.

\begin{lemma} \label{lemm::alpha_embedding}
There exists a Lie algebra homomorphism $\alpha : \even \to \Weyl$ which is injective for $m\neq n$, and whose kernel is $\CC E_{m\vert n}$ for $m=n$.
\end{lemma}

For later use, this Lie algebra homomorphism is given explicitly by (\emph{cf.}~\cite[Equation 10]{huang2005dirac}):
\begin{multline} \label{eq::alpha_explicit}
\alpha(X)= \sum_{i,j=1}^{mn}(B(X,[\partial_{i},\partial_{j}])x_{i}x_{j}+B(X,[x_{i},x_{j}])\partial_{i}\partial_{j}) 
\\ -\sum_{i,j=1}^{mn}2B(X,[x_{i},\partial_{j}])x_{j}\partial_{i}-\sum_{l=1}^{mn}B(X,[\partial_{l},x_{l}])
\end{multline}
for any $X\in \even$.

We define a diagonal embedding with respect to $\alpha$ (\emph{cf.}~\eqref{eq::general_embedding}):
\begin{equation}
\even \to \UE(\gg)\otimes \Weyl, \qquad X\mapsto X\otimes 1 + 1\otimes \alpha(X).
\end{equation}
We denote the images of $\even$, $\UE(\even)$, and $\mathfrak Z(\even)$ by $\gg_{\bar{0},\Updelta}, \UE(\gg_{\bar{0},\Updelta})$, and $\mathfrak{Z}(\gg_{\bar{0},\Updelta})$, respectively. Furthermore, the image of the quadratic Casimir $\Omega_{\even}$ of $\even$ is denoted by $\Omega_{\gg_{\bar{0}},\Updelta}$. The quadratic Casimir of $\gg$ is denoted by $\Omega_{\gg}$. The following lemma follows from a direct calculation.

\begin{lemma}[{\cite{huang2002dirac, kostant2001weyl}}] \label{lemm::constant_C}
Let $\{W_{k}\}$ be an orthonormal basis for $\even$ with respect to $\bil(\cdot,\cdot)$. Then the following assertions hold: 
\begin{enumerate}
\item[a)] $\Omega_{\even}=\sum_{k}W_{k}^{2}$.
 \item[b)] $\Omega_{\gg}=\Omega_{\even}+2\sum_{i}(x_{i}\partial_{i}-\partial_{i}x_{i})$.
 \item[c)] $\Omega_{\even,\Updelta}=\sum_{k}\left(W_{k}^{2}\otimes 1+2W_{k}\otimes \alpha(W_{k})+1\otimes \alpha(W_{k})^{2}\right)$.
 \item[d)] $C\coloneqq \sum_{k}\alpha(W_{k})^{2}$ is a constant. 
\end{enumerate}
\end{lemma}

\subsubsection{Relative quadratic Dirac operators}
Our basic example is the quadratic pair $(\gg,\even)$ with $\gg=\even\oplus\odd$ and $[\odd,\odd]\subseteq\even$. We therefore begin with a general quadratic pair $(\ll,\uu)$ such that
\begin{equation}
\ll=\uu\oplus\pp,\qquad [\pp,\pp]\subseteq\uu.
\end{equation}
The condition $[\pp,\pp]\subseteq\uu$ is the algebraic input needed for the relative quadratic Dirac operator. To encode the $\uu$-structure inside the Weil algebra, we write $W(\ll,\uu)$ and $\WW(\ll,\uu)$ for the subalgebras of $\uu$-invariants in $W(\ll)$ and $\WW(\ll)$, respectively, where the $\uu$-action on $\WW(\ll)$ is induced by \eqref{eq::general_embedding}. Concretely, the $\uu$-action is induced by the adjoint action on both factors, which is for $\WW(\ll)$ precisely $[\gamma^{\WW}(x),\cdot]_{\WW}$ for any $x \in \uu$.

Choose a homogeneous basis $\{e_a\}$ of $\ll$ with $B$-dual basis $\{e^a\}$. Then
\begin{equation}
\mathcal D_{\ll}\coloneqq \sum_a e^a\otimes e_a\in W(\ll)
\end{equation}
is the canonical quadratic Dirac element attached to $\ll$. Likewise, for $\uu$ one has $\mathcal D_{\uu} \in W(\uu)$. Passing to quantizations using $\mathcal{Q}$ (see Theorem~\ref{thm::quantization}) gives the operators $\Dirac_{\ll}$ and $\Dirac_{\uu}$. Their difference defines the $\uu$\emph{-relative quadratic Dirac operator}:
\begin{equation}
\Dirac_{\ll,\uu}\coloneqq j(\Dirac_{\uu})-\Dirac_{\ll}.
\end{equation}

\begin{lemma}[\cite{Dirac_quadratic}]\label{lemm::properties_relative_Dirac}
The operator $\Dirac_{\ll,\uu}$ satisfies the following properties:
\begin{enumerate}
\item[a)] It is independent of the choice of basis.
\item[b)] One has $\Dirac_{\ll,\uu}\in \UE(\ll)\otimes\Cl(\pp)$.
\item[c)] It is $\uu$-invariant for the induced $\uu$-action on $\UE(\ll)\otimes\Cl(\pp)$, that is,
\[
[j(\gamma^{\WW}(x)),\Dirac_{\ll,\uu}]_{\WW}=0
\qquad \forall\,x\in\uu.
\]
\item[d)] $\Dirac_{\ll,\uu}$ defines a differential on $\WW(\ll,\uu)$ by
\[
d^{\WW}\coloneqq [\Dirac_{\ll,\uu},\cdot]_{\WW}.
\]
\end{enumerate}
\end{lemma}

The fact that $d^{\mathcal{W}}$
squares to zero means that $\Dirac_{\ll,\uu}$ squares to a central element.

\subsubsection{The Dirac operator \texorpdfstring{$\Dirac$}{}}
In this article, we fix a Lie superalgebra $\gg$ of type $A$. Since $\gg$ is quadratic and $B|_{\even}$ is symmetric, non-degenerate, and invariant, the pair $(\gg,\even)$ is quadratic. Thus we may define the $\even$-relative quadratic Dirac operator
\begin{equation}
\Dirac\coloneqq j(\Dirac_{\even})-\Dirac_{\gg},
\end{equation}
which we simply call the \emph{Dirac operator}. This is the operator introduced in \cite{huang2005dirac}. In what follows, we identify $\WW(\even)$ with its image in $\WW(\gg)$ under $j$ and suppress $j$ from the notation. 

Since the definition of $\Dirac$ is basis-independent, we choose the basis $\{x_i,\partial_i\}$ introduced above, whose $B$-dual basis is $\{-2\partial_i,2x_i\}$. With this choice,
\begin{equation}
\Dirac=2\sum_{i=1}^{mn}\bigl(\partial_i\otimes x_i-x_i\otimes\partial_i\bigr)\in \UE(\gg)\otimes\Weyl.
\end{equation}
By Lemma~\ref{lemm::properties_relative_Dirac}, the operator $\Dirac$ is $\even$-invariant. For completeness, we record the square of $\Dirac$ (\emph{cf.}~\cite[Proposition 2]{huang2005dirac}).

\begin{theorem} \label{thm::square_dirac}
One has
\[
\Dirac^{2}=-\Omega_{\gg}\otimes 1+\Omega_{\even,\Updelta}+C(1\otimes 1), \qquad C=-\frac{1}{24}\bigl(\str_{\gg}(\ad_{\gg}(\Omega_{\gg}))-\str_{\even}(\ad_{\even}(\Omega_{\even}))\bigr)
\]
where $C$ is the constant of Lemma~\ref{lemm::constant_C}.
\end{theorem}

\begin{remark} \label{rmk::Explicit_form_D_square}
For explicit calculations, it is particularly useful to express the square of the Dirac operator in the following alternative form:
\begin{equation*}
\Dirac^{2} = 2 \sum_{i,j = 1}^{mn} ([\partial_{i},\partial_{j}] \otimes x_{i}x_{j} + [x_{i},x_{j}] \otimes \partial_{i}\partial_{j} - 2[\partial_{i},x_{j}]\otimes x_{i}\partial_{j}) - 4 \sum_{i=1}^{mn}x_{i}\partial_{i}\otimes 1.
\end{equation*}
\end{remark}

We finally decompose $\Dirac$ in appropriate $\kk^{\CC}$-invariant pieces. To this end, recall that $x_{1}, \ldots, x_{mn}$ span $\nn_{\bar{1}}^{-} = \pp_{2} \oplus \qq_{1}$ and $\partial_{1}, \ldots, \partial_{mn}$ span $\nn_{\bar{1}}^{+} = \pp_{1} \oplus \qq_{2}$. The spaces $\pp_{1,2}$ and $\qq_{1,2}$ (see Section~\ref{subsubsec::Real_form}) are invariant under $\kk^{\CC}$, but not under $\even$. Accordingly, we decompose the Dirac operator as $\Dirac\coloneqq \Dirac^{\pp_{1}} + \Dirac^{\qq_{2}}$, where
\begin{equation} \label{eq::decomposition_Dirac}
\begin{split}
\Dp &\coloneqq 2(\ddp - \delp), \qquad
\ddp\coloneqq \sum_{i=1}^{pn} \partial_{i} \otimes x_{i}, \qquad \ \ \ \
\delp\coloneqq \sum_{i=1}^{pn} x_{i} \otimes \partial_{i}, \\
\Dq &\coloneqq 2(\ddq - \delq), \qquad
\ddq\coloneqq \sum_{i=pn+1}^{mn} \partial_{i} \otimes x_{i}, \qquad
\delq\coloneqq \sum_{i=pn+1}^{mn} x_{i} \otimes \partial_{i}.
\end{split}
\end{equation}

\begin{lemma} \label{lemm::k-invariance}
 The operators $\ddp,\ddq$ and $\delp,\delq$ are $\kk^{\CC}$-invariant, meaning $[\mathrm{d}^{\pp_{1},\qq_{2}}, \kk^{\CC}] = 0$ and $[\delta^{\pp_{1},\qq_{2}}, \kk^{\CC}] = 0$.
\end{lemma}
\begin{proof}
We prove the statement only for $\ddp = \sum_{i=1}^{pn} \partial_{i} \otimes x_{i}$; the rest follows analogously.

Let $X \in \kk^{\CC}$. Since $\kk^{\CC}$ leaves $\pp_{1}$ invariant by Lemma~\ref{lemm::commutation_relations}, the commutators in the standard basis are given by
\begin{equation*}
[X, x_{i}] = -2 \sum_{j=1}^{pn} B(X, [x_{i}, \partial_{j}]) x_{j}, \qquad [X, \partial_{i}] = 2 \sum_{j=1}^{pn} B(X, [\partial_{i}, x_{j}]) \partial_{j}
\end{equation*}
for all $1\leq i \leq pn$. This leads to the following computation:
\begin{equation*}
\begin{split}
[\gamma^{\WW}(X),\ddp]_{\WW} &= \sum_{i=1}^{pn} \left([X,\partial_{i}] \otimes x_{i} + \partial_{i} \otimes [X,x_{i}]\right) \\ &= 2\sum_{i,j=1}^{pn}B(X,[\partial_{i},x_{j}])\partial_{j}\otimes x_{i}-2\sum_{i,j=1}^{pn}B(X,[x_{i},\partial_{j}])\partial_{i}\otimes x_{j}\\&=2\sum_{i,j=1}^{pn} \left(-B(X,[x_{i},\partial_{j}]) + B(X,[\partial_{j},x_{i}])\right) \partial_{i}\otimes x_{j} \\ &=0,
\end{split}
\end{equation*}
where we used in the last equality $B(X,[\partial_{i},x_{j}]) = B(X,[x_{j},\partial_{i}])$ for all $1 \leq i,j\leq pn$ since $x_{i}$ and $\partial_{j}$ are odd.
\end{proof}

\subsection{Dirac cohomology} \label{subsec::Dirac_Cohomology} 
The Dirac cohomology assigns to any $\gg$-supermodule $M$ a $\even$-supermodule $\DC(M)$, utilizing the $\even$-invariance of the Dirac operator $\Dirac$. To define it, we consider the natural componentwise action of $\Dirac \in \UE(\gg) \otimes \Weyl$ on $M \otimes M(\odd)$, where $M(\odd)$ denotes the oscillator module over $\Weyl$, which we briefly introduce.

\subsubsection{Oscillator module} \label{subsubsec::Weyl_Supermodule} The Weyl algebra $\Weyl$, identified with the algebra of differential operators with polynomial coefficients in the variables $x_i$, where $\partial_i$ corresponds to the partial derivative $\frac{\partial}{\partial x_i}$, acts naturally on $M(\odd)\coloneqq \CC[x_1,\dotsc, x_{mn}]$. We refer to $M(\odd)$ as the \emph{oscillator module}. Any element in $M(\odd)$ that is annihilated by all $\partial_{i}$ is
necessarily constant. We conclude that the maximal proper submodule of $M(\odd)$ is zero,
and $M(\odd)$ is a simple module over $\Weyl$.

\begin{proposition} \label{prop::Weyl_Simple}
 The oscillator module $M(\odd)$ is a simple $\Weyl$-module. 
\end{proposition}

We equip $M(\odd)$ with a $\ZZ_{2}$-grading by declaring $M(\odd)_{\bar{0}}$ to be the subspace generated by homogeneous polynomials of even degree, and $M(\odd)_{\bar{1}}$ to be the subspace generated by homogeneous polynomials of odd degree. 

Additionally, $M(\odd)$ carries a unique Hermitian form $\langle\cdot,\cdot\rangle_{M(\odd)}$, known as the \emph{Bargmann--Fock Hermitian form} or \emph{Fischer--Fock Hermitian form} \cite{BargmannHermitianForm, FischerHermitianForm, FockHermitianForm}, such that $\partial_{k}$ and $x_{k}$ are adjoint to each other and the following orthogonality relations hold:
\begin{equation}
\begin{split}
\langle \prod_{k=1}^{mn}x_{k}^{p_{k}},\prod_{k=1}^{mn}x_{k}^{q_{k}}\rangle_{M(\odd)}=\begin{cases}\prod_{k=1}^{mn}p_{k}! \qquad &\text{if} \ p_{k}=q_{k} \ \text{for all} \ k, \\
0 \qquad &\text{otherwise}.\end{cases}
\end{split}
\end{equation}

The conjugate-linear anti-involution $\omega$ induces a $^{\ast}$-structure on $\Cl(\odd)=\Weyl$, that is, a conjugate-linear involution $(\cdot)^{\ast} : \Weyl \to \Weyl$ such that $(xy)^{\ast}=y^{\ast}x^{\ast}$ for all $x,y \in \Weyl$. 

\begin{lemma} \label{lemm::star_structure}
 The conjugate-linear involution $\omega$ induces a $^{\ast}$-structure on $\Cl(\odd) = \Weyl$ via $x^{\ast}\coloneqq -\omega(x)$ such that $x_{i}^{\ast}=\partial_{i}$.
\end{lemma}
\begin{proof}
We first note that $\omega(\odd) \subset \odd$. 
We extend $\omega$ to $T(\odd)$ by defining it on simple tensors via
\[
\omega(v_{1} \otimes \cdots \otimes v_{k})
\coloneqq \omega(v_{k}) \otimes \cdots \otimes \omega(v_{1}).
\]
This map descends to $\Cl(\odd)$ provided that it preserves the ideal $I_{\Cl}$ generated by the elements
\[
v \otimes w + (-1)^{p(v)p(w)} w \otimes v - 2 B(v,w) \, 1_{T(\odd)}.
\]
Applying $\omega$ shows that this condition is equivalent to
\[
\overline{B(w,v)} = B(\omega(v),\omega(w)) \qquad \text{for all } v,w\in\odd.
\]

Recall that $\omega$ is defined by
\[
\omega(x)
 = J_{(p,q\vert 0,n)}^{-1} x^\dagger J_{(p,q\vert 0,n)},
 \qquad x\in\gg.
\]
Hence, for all $x,y\in\gg$ we obtain
\[
B(\omega(x),\omega(y))
 = \tfrac{1}{2} \str(\omega(x)\omega(y))
 = \tfrac{1}{2} \str\!\left(J_{(p,q\vert 0,n)}^{-1} x^\dagger y^\dagger J_{(p,q\vert 0,n)}\right)
 = \tfrac{1}{2} \overline{\str(yx)}
 = \overline{B(y,x)}.
\]
Now, set $v^{\ast}\coloneqq -\omega(v)$. 
Since $\omega$ preserves $\odd$ and, by construction, satisfies $\omega(vw)=\omega(w)\omega(v)$ for all $v,w\in\odd$, the map $(\cdot)^{\ast} : \Weyl \to \Weyl$ is a $^{\ast}$-structure on $\Weyl$. 
\end{proof}

We fix the $^{\ast}$-structure on $\Weyl$ induced by $\omega$. 
A $\Weyl$-module $M$ is called a $^{\ast}$-module if it is equipped with a positive definite Hermitian form $\langle\cdot,\cdot\rangle$ satisfying
\begin{equation} \label{eq::star_module}
 \langle x v, w \rangle = \langle v, x^{\ast} w \rangle
 \qquad x\in\Weyl,\; v,w\in M.
\end{equation}

This yields the following lemma.

\begin{lemma}\label{lemm::adjoint_Weyl_supermodule}
 The oscillator module $M(\odd)$, together with the form $\langle \cdot,\cdot\rangle_{M(\odd)})$, is a $^{\ast}$-module over $\Weyl$. In particular, the generators of $\Weyl$ satisfy the following relations for all $1 \leq i \leq mn$:
$$
\langle \partial_i v, w \rangle_{M(\odd)} = \langle v, x_i w \rangle_{M(\odd)}, \qquad \langle x_i v, w \rangle_{M(\odd)} = \langle v, \partial_i w \rangle_{M(\odd)}.
$$
\end{lemma}

We finally consider $M(\odd)$ as a $\even$-module under the Lie algebra homomorphism $\alpha : \even \to \Weyl$. 

\begin{lemma} \label{lemm::compatibility_omega_alpha}
 As operators on $M(\odd)$, one has for any $X \in \even$
 \[
 \alpha(X)^{\ast} = \alpha(\omega(X)).
 \]
\end{lemma}

\begin{proof}
We first recall $x_{i}^{\ast}=\partial_{i}$, and $\overline{B(x,y)} = B(\omega(y),\omega(x))$. The Lie algebra homomorphism $\alpha$ is explicitly given by 
\begin{multline*}
\alpha(X)= \sum_{i,j=1}^{mn}(B(X,[\partial_{i},\partial_{j}])x_{i}x_{j}+B(X,[x_{i},x_{j}])\partial_{i}\partial_{j}) 
\\ -\sum_{i,j=1}^{mn}2B(X,[x_{i},\partial_{j}])x_{j}\partial_{i}-\sum_{l=1}^{mn}B(X,[\partial_{l},x_{l}]).
\end{multline*}
 We demonstrate the statement for the first summand $B(X,[\partial_{i},\partial_{j}])x_{i}x_{j}$. The rest follows by a similar line of argument. One has
 \[
 \begin{aligned}
 (B(X,[\partial_{i},\partial_{j}])x_{i}x_{j})^{\ast} &= \overline{B(X,[\partial_{i},\partial_{j}])}x_{j}^{\ast}x_{i}^{\ast} \\ &= B(\omega([\partial_{i},\partial_{j}]),\omega(X))\partial_{j}\partial_{i}\\ &= B([\omega(\partial_{j}),\omega(\partial_{i})],\omega(X))\partial_{j}\partial_{i} \\ &= B([x_{j},x_{i}], \omega(X))\partial_{j}\partial_{i}.
 \end{aligned}
 \]
 This coincides with the second summand of $\alpha(\omega(X))$ as $B([x_{j},x_{i}],\omega(X))=B(\omega(X),[x_{j},x_{i}])$ since $B$ is supersymmetric and $p(X)=\bar{0}$.
\end{proof}

\begin{proposition} \label{prop:oscillator_module_even_semisimple}
 The oscillator module $M(\odd)$ is a $\omega\vert_{\even}$-unitarizable $\even$-module under the action induced by $\alpha: \even \to \Weyl$. In particular, $M(\odd)$ is $\even$-semisimple.
\end{proposition}

\begin{proof}
 The statement follows directly from Lemma~\ref{lemm::star_structure} and Lemma~\ref{lemm::compatibility_omega_alpha}.
\end{proof}

\begin{remark} \label{rmk::decompsoition_oscillator_module}
 $M(\odd)$ has a natural interpretation as the oscillator module for $\even$ \cite{kashiwara1978segal}, and its $\even$-constituents are often referred to as \emph{ladder modules}. In particular, it decomposes discretely into a direct sum of simple highest weight $\even$-modules.
\end{remark}

For the fixed real form $\mathfrak{(p)su}(p,q\vert n)$ of $\gg$, let $\kk$ be the maximal compact subalgebra of $\mathfrak{(p)su}(p,q\vert n)_{\bar{0}}$ introduced in Section~\ref{sec::realforms}, so that:
\begin{equation}
\kk^{\CC} \subset \even \subset \gg.
\end{equation}
We view $M(\odd)$ as a $(\even,\kk^{\CC})$-module (\emph{cf.}~Section~\ref{subsec::HC_Supermodule}), which allows us to compare the action of $\kk^{\CC}$ on $M(\odd)$ induced by $\alpha$ with the natural action of $\kk^{\CC}$ on $\CC[x_1, \dots, x_{mn}]$ arising from the adjoint action \footnote{This is the best available comparison, as $\nn_{\bar{1}}^{-}$ is invariant only under $\kk^{\CC}$.}. These two $\even$-modules are related by the following proposition.

\begin{proposition}\label{EvenRestriction}
The action of $\alpha(\kk^{\CC})$ on $M(\odd)$ and the adjoint action of $\kk^{\CC}$ on $\CC[x_{1},\dotsc,x_{mn}]$ differ by a twist via $\CC_{-\rho_{\bar{1}}}$. In particular, we have an isomorphism of $\kk^{\CC}$-modules:
$$
M(\odd)\cong \CC[x_{1},\dotsc,x_{mn}] \otimes \CC_{-\rho_{\bar{1}}}.
$$
\end{proposition}

\begin{proof}
First, we recall the commutation relations from Lemma~\ref{lemm::commutation_relations}:
\begin{equation*}
[\partial_{i},\partial_{j}] \in \nn_{\bar{0},nc}^{+}, \quad [x_{i},x_{j}] \in \nn_{\bar{0},nc}^{-}, \quad [x_{i},\partial_{i}] \in \hh, \quad [\kk^{\CC},\nn_{\bar{0},nc}^{\pm}] \subset \nn_{\bar{0},nc}^{\pm}
\end{equation*}
for all $1 \leq i, j \leq mn$. These relations imply that $B(X,[\partial_{k},\partial_{l}]) = 0$ and $B(X,[x_{i},x_{j}]) = 0$ for any $X \in \kk^{\CC}$ since $B$ is invariant and induced by the supertrace; in particular $B(x,y)=0$ whenever $x,y \in \nn_{\bar{0},nc}^{\pm}$. Consequently, $\alpha : \gg \big\vert_{\kk^{\CC}} \to \Weyl$ is given by
\begin{equation*}
\alpha(X) = -\sum_{i,j=1}^{mn} 2 \bil(X,[x_{i},\partial_{j}]) x_{j} \partial_{i} - \sum_{l=1}^{mn} \bil(X,[\partial_{l},x_{l}]), \qquad X \in \kk^{\CC}.
\end{equation*}
We claim that $\alpha(X) P = [X,P]$ for any $X \in \kk^{\CC}$ and any polynomial $P \in \CC[x_{1}, \ldots, x_{mn}] \otimes \CC_{-\rho_{\bar{1}}}$, where $\CC_{-\rho_{\bar{1}}}$ is identified with the one-dimensional $\kk^{\CC}$-module generated by the constant polynomial $1$. By linearity, it is sufficient to consider monomials of the form $P = \prod_{k=1}^{mn} x_{k}^{r_{k}} \otimes 1$, where $r_{k} \in \ZZ_{+}$. 

Since $\nn_{\bar{1}}^{-}$ is $\kk^{\CC}$-invariant, we have the following expression in the standard basis:
\begin{equation*}
[X,x_{i}] = -2 \sum_{j=1}^{mn} B([X,x_{i}],\partial_{j}) x_{j} = -2 \sum_{j=1}^{mn} B(X,[x_{i},\partial_{j}]) x_{j}.
\end{equation*}
This implies
\begin{equation*}
\alpha(X) = \sum_{i=1}^{mn} [X,x_{i}] \partial_{i} - \sum_{l=1}^{mn} B(X,[\partial_{l},x_{l}]).
\end{equation*}
The first term acts non-trivially only on $\prod_{l=1}^{mn} x_{l}^{r_{l}}$, while the second term acts as a scalar multiple of the identity. Hence, it suffices to consider the action on $1$.

To compute this, we evaluate the action of the two terms separately. First, we have
\begin{equation*}
\sum_{i=1}^{mn} [X,x_{i}] \partial_{i} P = \sum_{i=1}^{mn} [X,x_{i}] \frac{r_{i}}{x_{i}} \prod_{l=1}^{mn} x_{l}^{r_{l}} 
= \sum_{i=1}^{mn} x_{1}^{r_{1}} \cdots [X,x_{i}^{r_{i}}] \cdots x_{mn}^{r_{mn}} 
= [X, \prod_{l=1}^{mn} x_{l}^{r_{l}}] = [X, P].
\end{equation*}

Next, we compute the action of $-\sum_{l=1}^{mn} B(X, [x_{l}, \partial_{l}])$ on $1$. Since $[x_{l}, \partial_{l}] \in \hh$, the action is trivial unless $X = H \in \hh$:

\begin{equation*}
-\sum_{l=1}^{mn} \bil(H, [\partial_{l}, x_{l}]) \cdot 1 = -\sum_{l=1}^{mn} B(H, E_{ll} + E_{n+l, n+l}) \cdot 1 = -\frac{1}{2} \sum_{\alpha \in \Delta_{\bar{1}}^{+}} \alpha(H) \cdot 1 
= -\rho_{\bar{1}}(H) \cdot 1.
\end{equation*}
This concludes the proof.
\end{proof}

\subsubsection{Dirac cohomology} The Dirac operator $\Dirac$ is an element of $\UE(\gg)\otimes\Weyl$ and therefore acts naturally on $M\otimes M(\odd)$ for every $\gg$-supermodule $M$. Since $\Dirac$ is $\even$-invariant for the induced $\even$-action on $\UE(\gg)\otimes\Weyl$, the kernel $\ker\Dirac$ is a $\even$-supermodule. This leads to the following definition.

\begin{definition}[{\cite{huang2005dirac}}]
 Let $M$ be a $\gg$-supermodule. Consider the action of the Dirac operator $\Dirac \in \UE(\gg)\otimes \Weyl$ on $M \otimes M(\odd)$.
 The \emph{Dirac cohomology} of $M$ is the $\even$-supermodule
 $$
 \DC(M)\coloneqq \Ker \Dirac / \left(\Ker \Dirac \cap \Im \Dirac\right).
 $$
\end{definition}
\begin{remark}
One can analogously define Dirac cohomology for general quadratic pairs $(\ll,\uu)$. This was studied in \cite{Schmidt_Noja}.
\end{remark}
\begin{example}
Let $M$ be the trivial $\gg$-supermodule. Then there is an isomorphism of $\even$-modules
$$
\DC(M)\cong M(\odd).
$$
In particular, $\DC(M)$ is a unitarizable $\even$-module and it decomposes as a direct sum of unitarizable simple $\even$-modules that occur with multiplicity one \cite{kashiwara1978segal}.
\end{example}

In general, Dirac cohomology defines a functor 
\begin{equation}
\DC(\cdot): \ \gsmod \longrightarrow \evsmod, \qquad M \mapsto \DC(M),
\end{equation}
and it admits a natural decomposition, $\DC(M) = \DC^{+}(M) \oplus \DC^{-}(M)$, induced by the $\ZZ_{2}$-grading of $M(\odd)$ introduced above. We decompose the Dirac operator as $\Dirac = \Dirac^{+} + \Dirac^{-}$, where
\begin{equation}
\begin{split}
\Dirac^{+} &\coloneqq \Dirac\Big\vert_{M \otimes M(\odd)_{\bar{0}}} : M \otimes M(\odd)_{\bar{0}} \to M \otimes M(\odd)_{\bar{1}}, \\
\Dirac^{-} &\coloneqq \Dirac\Big\vert_{M \otimes M(\odd)_{\bar{1}}} : M \otimes M(\odd)_{\bar{1}} \to M \otimes M(\odd)_{\bar{0}},
\end{split}
\end{equation}
and define $\DC^{+}(M)\coloneqq \mathrm{H}_{\Dirac^{+}}(M)$ and $\DC^{-}(M)\coloneqq \mathrm{H}_{\Dirac^{-}}(M)$. 

\subsubsection{Dirac cohomology and infinitesimal characters}
Dirac cohomology becomes especially effective for $\gg$-supermodules with infinitesimal character. We therefore recall the relevant notion (\emph{cf.}~Section~\ref{subsubsec::infinitesimal_characters}). A $\gg$-supermodule $M$ has infinitesimal character if there exists $\lambda\in\hh^\ast$ such that $\mathfrak Z(\gg)$ acts on $M$ by
\begin{equation}
zv=\chi_\lambda(z)v,\qquad \chi_\lambda(z)\coloneqq (\lambda+\rho)(\operatorname{HC}_{\gg}(z))\qquad \forall\,z\in\mathfrak Z(\gg),\ v\in M.
\end{equation}
Here $\operatorname{HC}_{\gg}$ denotes the Harish-Chandra homomorphism for $\gg$.

Similarly, a $\even$-supermodule $V$ has infinitesimal character if there exists $\lambda\in\hh^\ast$ such that $\mathfrak Z(\even)$ acts on $V$ by
\begin{equation}
zv=\chi^{\bar0}_{\lambda}(z)v,\qquad \chi^{\bar0}_{\lambda}(z)\coloneqq (\lambda+\rho_{\bar0})(\operatorname{HC}_{\even}(z))\qquad \forall\,z\in\mathfrak Z(\even),\ v\in V.
\end{equation}
The infinitesimal character $\chi^{\bar0}_{\lambda}$ is determined up to the dot action of the Weyl group:
\begin{equation}\label{eq::dot_action_and_infinitesimal_character}
\chi^{\bar0}_{\lambda}=\chi^{\bar0}_{\mu}\iff \lambda=w\cdot\mu=w(\mu+\rho_{\bar0})-\rho_{\bar0}
\end{equation}
for some $w\in W$.

The link between the infinitesimal character of a $\gg$-supermodule and that of its Dirac cohomology is provided by the following result.

\begin{theorem}[{\cite[Theorem 6]{huang2005dirac}}]\label{thm::zeta}
For every $z\in\mathfrak Z(\gg)$, there exists an algebra homomorphism
$
\zeta\colon \mathfrak Z(\gg)\to \mathfrak Z(\even)\cong \mathfrak Z(\gg_{\bar0,\Updelta})
$
and a $\even$-invariant element $a\in \UE(\gg)\otimes\Weyl$ such that
\[
z\otimes 1=\zeta(z)+\Dirac a+a\Dirac
\]
in $\UE(\gg)\otimes\Weyl$. Moreover, $\zeta$ is characterized by the commutative diagram
\[
\begin{tikzcd}
\mathfrak Z(\gg)\arrow{r}{\zeta}\arrow[swap]{d}{\operatorname{HC}_{\gg}}&\mathfrak Z(\even)\arrow{d}{\operatorname{HC}_{\even}}\\
S(\hh)^W\arrow{r}{\id}&S(\hh)^W.
\end{tikzcd}
\]
\end{theorem}

As an immediate consequence, one obtains the superalgebra analog of Vogan's theorem.

\begin{theorem}\label{VoganConjecture}
Let $M$ be a $\gg$-supermodule with infinitesimal character $\chi_\Lambda$. If $\DC(M)$ contains a nonzero $\even$-supermodule with infinitesimal character $\chi^{\bar0}_{\lambda}$ for some $\lambda\in\hh^\ast$, then
\[
\chi_\Lambda(z)=\chi^{\bar0}_{\lambda}(\zeta(z))
\qquad \forall\,z\in\mathfrak Z(\gg).
\]
\end{theorem}

This applies in particular to highest weight and simple $\gg$-supermodules, since these admit infinitesimal characters. It is especially useful when the underlying $\gg$-supermodule is completely reducible over $\even$, as is the case for unitarizable supermodules.

\subsection{Dirac induction}\label{subsec::Dirac_Induction} For general $\gg$-supermodules, Dirac cohomology has poor functorial properties. In particular, $\DC(\cdot)$ does not in general admit an adjoint functor. However, for supermodules with infinitesimal character, it satisfies a six-term exact sequence.

\begin{lemma}\label{exactSequence}
Let $0 \to M' \xlongrightarrow{a} M \xlongrightarrow{b} M'' \to 0$ be a short exact sequence of $\gg$-supermodules admitting an infinitesimal character. Then there exists a six-term exact sequence:
\[
\begin{tikzcd}
\DC^{+}(M') \arrow{r} & \DC^{+}(M) \arrow{r} & \DC^{+}(M'') \arrow{d} \\
\DC^{-}(M'') \arrow{u} & \arrow{l} \DC^{-}(M) & \arrow{l} \DC^{-}(M')
\end{tikzcd}
\]
\end{lemma}

\begin{proof}
We modify \cite[Theorem 8.1]{Dirac_cohomology_Lie_algebra_cohomology} to our setting. We tensor the short exact sequence with $M(\odd)$ and denote the arrows again by $a$ and $b$, which get tensored with the identity on $M(\odd)$. As a result, we obtain a right exact sequence
$$
M'\otimes M(\odd)\to M\otimes M(\odd) \to M''\otimes M(\odd) \to 0.
$$
However, $m'' \neq 0$, yielding $c=0$. In addition, we have $b(\Dirac\! m)=\Dirac\! b(m)=\Dirac\! m''=0$, \emph{i.e.}, $\Dirac\! m$ lies in $\ker(b)=\Im(a)$, and we find some $m' \in M$ with $\Dirac\! m = a(m')$. The element $m'$ determines a cohomology class, as $a(\Dirac\! m')= \Dirac\! a(m')= \Dirac^{2}\! m =0$, and $a$ is injective. This class is by definition the image of the class of $m''$ under the connecting homomorphism. This defines both vertical arrows. Exactness is verified by a direct calculation, which we omit.
\end{proof}

This motivates an alternative definition of Dirac cohomology, which agrees with the classical one on unitarizable supermodules. The construction goes back to \cite{pandzic2010dirac}. Following \cite{pandzic2010dirac}, we define a functor, again called \emph{Dirac cohomology}, which admits a right adjoint, called \emph{Dirac} induction, and coincides with $\DC(\cdot)$ on the subcategory of unitarizable supermodules.

Any $\gg$-supermodule $M$ extends to a $(\UE(\gg)\otimes\Weyl)$-supermodule by $M\mapsto M\otimes M(\odd)$. Accordingly, let $(\UE(\gg)\otimes\Weyl)\textbf{-smod}^{\mathrm{osc}}$ be the full subcategory of $(\UE(\gg)\otimes\Weyl)\textbf{-smod}$ consisting of supermodules isomorphic to $M\otimes M(\odd)$ for some $\gg$-supermodule $M$. We call it the \emph{oscillator subcategory}. Conversely, any object $X$ of this subcategory defines a $\UE(\gg)$-supermodule by $X\mapsto \Hom_{\Weyl}(M(\odd),X)$, with the induced $\UE(\gg)$-action.
 
\begin{proposition} 
There is an equivalence of categories 
$$
F : \gg\textbf{-smod}\to \left( \UE(\gg)\otimes \Weyl\right)\textbf{-smod}^{\text{osc}}, \qquad F(M)\coloneqq M\otimes M(\odd)
$$
with inverse $G(X)\coloneqq \Hom_{\Weyl}(M(\odd),X)$. 
\end{proposition}
\begin{proof}
It is well known that $\UE(\gg)$\textbf{-smod} is equivalent to $\gsmod$. By abuse of notation, we still use $F$ for the functor from $\UE(\gg)\textbf{-smod}$ to $(\UE(\gg) \otimes \Weyl)\textbf{-smod}^{\text{osc}}$. 

We first show that $F$ and $G$ are naturally isomorphic on objects. For any $\UE(\gg)$-supermodule $M$, we have the following isomorphisms:
$$
\Hom_{\Weyl}(M(\odd), M \otimes M(\odd)) \cong M \otimes \Hom_{\Weyl}(M(\odd), M(\odd)) \cong M,
$$
where the second isomorphism uses the fact that $M(\odd)$ is a simple $\Weyl$-supermodule (see Proposition~\ref{prop::Weyl_Simple}) and Schur's lemma. Conversely, let $X \in (\UE(\gg)\otimes \Weyl)\textbf{-smod}^{\text{osc}}$. Then there exists some $\UE(\gg)$-supermodule $M$ such that $X \cong M\otimes M(\odd)$ and one has
\[
\Hom_{\Weyl}(M(\odd),X)
\cong
\Hom_{\Weyl}(M(\odd),M\otimes M(\odd))
\cong
M\otimes \End_{\Weyl}(M(\odd))
\cong M,
\]
where the last isomorphism uses Schur’s lemma. Therefore
$
F(G(X))\cong M\otimes M(\odd)\cong X.
$ Altogether, we have proved that the assignments $M \mapsto M\otimes M(\odd)$ and $X \mapsto \Hom_{\Weyl}(M(\odd),X)$ are inverses of each other. In particular, $F$ is essentially surjective. 

We show that $F$ is full and faithful. 
Consider the canonical tensor--Hom adjunction
\[
\Hom_{\UE(\gg) \otimes \Weyl}(M \otimes M(\odd),\, X)
 \cong
 \Hom_{\UE(\gg)}\!\bigl(M,\, \Hom_{\Weyl}(M(\odd),X)\bigr),
\]
which expresses precisely that $F$ is left adjoint to $G$. 
Hence
\[
\Hom_{\UE(\gg) \otimes \Weyl}(F(M),F(N))
 \cong
 \Hom_{\UE(\gg)}(M, G(F(N)))
 \cong
 \Hom_{\UE(\gg)}(M,N),
\]
showing that $F$ is full and faithful. 
Since $F$ is also essentially surjective, we conclude that $F$ is an equivalence of categories. \end{proof}

We consider $\UE(\gg) \otimes \Weyl$ as a $\UE(\even)$-module under the diagonal embedding, and denote the subspace of $\UE(\even)$-invariant elements of $\UE(\gg) \otimes \Weyl$ by $(\UE(\gg) \otimes \Weyl)^{\UE(\even)}$. Let $\mathscr{I}$ be the two-sided ideal in $(\UE(\gg) \otimes \Weyl)^{\UE(\even)}$ generated by the Dirac operator $\Dirac$. Since the Dirac operator $\Dirac$ commutes with $\even$, the ideal $\mathscr{I}$ is $\even$-stable by construction.

We define $\DC'(M)$ to be the subspace of $\mathscr{I}$-invariants in $M \otimes M(\odd)$ for a given $\gg$-supermodule $M$, that is,
\begin{equation}
\DC'(M)\coloneqq \{v \in M \otimes M(\odd) \ : \ x v = 0 \ \text{for all} \ x \in \mathscr{I}\}.
\end{equation}
In particular, $\DC'(M)$ is a $\even$-supermodule for any $\gg$-supermodule $M$. For unitarizable $\gg$-supermodules $M$, we will show that $\DC(M) \cong \DC'(M)$ in Corollary~\ref{DiracsCoincide}. For this reason, we also refer to $\DC'(\cdot)$ as Dirac cohomology.

\begin{proposition} \label{DiracInduction}
The Dirac cohomology $\DC'(\cdot)$ admits a left-adjoint functor $\Ind_{\Dirac}: \gg_{\bar{0},\Updelta}\textbf{-smod} \to (\UE(\gg)\otimes \Weyl)\textbf{-smod}^{\text{osc}}$, called \emph{Dirac induction}, given by 
$$
\Ind_{\Dirac}(V)\coloneqq \left(\UE(\gg)\otimes \Weyl\right)\otimes_{\UE(\gg_{\bar{0},\Updelta})(\CC 1\oplus \mathscr{I})}V
$$
for any $\even$-supermodule $V$ with trivial $\mathscr{I}$-action.
In particular, $\DC'(\cdot)$ is left-exact. 
\end{proposition}
\begin{proof}
We can understand $\DC'(\cdot)$ through the forgetful functor from $\UE(\gg)\otimes \Weyl$ to $\UE(\gg_{\bar{0},\Updelta})(\CC 1\oplus \mathscr{I})$ and then followed by taking $\mathscr{I}$-invariants. Indeed, $\UE(\gg_{\bar{0},\Updelta})(\CC 1\oplus \mathscr{I})$ is the $\even$-invariant subalgebra of $\UE(\gg)\otimes \Weyl$ generated by $\gg_{\bar{0},\Updelta}$ and $\mathscr{I}$. This has exactly as left-adjoint $\left(\UE(\gg)\otimes \Weyl\right)\otimes_{\UE(\gg_{\bar{0},\Updelta})(\CC 1\oplus \mathscr{I})}V$.
\end{proof}
\begin{corollary}
Let $M$ be a $\gg$-supermodule. Then a submodule $V \subset \DC(M)$ appears as a $\gg_{\bar{0},\Updelta}$-constituent if and only if $M$ is a quotient of $\Ind_{\Dirac}(V)$.
\end{corollary}

\section{Dirac operators, Dirac cohomology and unitarizable supermodules} \label{DiracAndUnitarity}
\noindent
In this section we explain how the Dirac operator and Dirac cohomology reflect unitarity. 
More precisely, we establish a Parthasarathy--Dirac inequality, show how the Dirac operator detects the unitarity of $M$ (in particular, the selfadjointness of $\Dirac$ is equivalent to the contravariance of a positive definite Hermitian form), derive a decomposition of $\HH \otimes M(\odd)$ with respect to $\Dirac^{2}$, and demonstrate that Dirac cohomology provides a complete characterization of unitarizable supermodules.

\subsection{Dirac operators and unitarity} Fix a unitarizable $\gg$-supermodule $\HH$, equipped with a positive definite Hermitian form $\langle \cdot, \cdot \rangle_{\HH}$. For $M(\odd)$, we use the Bargmann--Fock Hermitian form $\langle \cdot, \cdot \rangle_{M(\odd)}$, as defined in Section~\ref{subsubsec::Weyl_Supermodule}. We then endow the $\UE(\gg) \otimes \Weyl$-supermodule $\HH \otimes M(\odd)$ with the Hermitian form defined by
\begin{equation}
\langle v \otimes P, w \otimes Q \rangle_{\HH \otimes M(\odd)}\coloneqq \langle v, w \rangle_{\HH} \cdot \langle P, Q \rangle_{M(\odd)},
\end{equation}
for all $v \otimes P, w \otimes Q \in \HH \otimes M(\odd)$. If $\HH$ is simple, the Hermitian form $\langle \cdot, \cdot \rangle_{\HH \otimes M(\odd)}$ is positive definite and unique up to a real scalar, since $\HH$ is a highest weight supermodule and $M(\odd)$ is a simple $^{\ast}$-module over $\Weyl$ (\emph{cf.}~\eqref{eq::star_module}). 

We extend the conjugate-linear anti-involution $\omega$ naturally to $\UE(\gg)$. Moreover, $\omega$ defines a conjugate-linear anti-involution on the Weyl algebra $\Weyl$, and we extend it in the obvious way to a conjugate-linear anti-involution on the tensor product $\UE(\gg) \otimes \Weyl$, denoted by the same symbol.

The following lemma is a direct consequence of Lemma~\ref{lemm::adjoint_Weyl_supermodule}.

\begin{lemma} \label{PropertiesH}
For all $X \in \gg$ and for all generators $\partial_k, x_k$ of $\Weyl$, the following identities hold:
\begin{equation*}
\begin{aligned}
\langle (X \otimes x_k) v, w \rangle_{\HH \otimes M(\odd)} &= \langle v, (\omega(X) \otimes \partial_k) w \rangle_{\HH \otimes M(\odd)}, \\
\langle (X \otimes \partial_k) v, w \rangle_{\HH \otimes M(\odd)} &= \langle v, (\omega(X) \otimes x_k) w \rangle_{\HH \otimes M(\odd)},
\end{aligned}
\end{equation*}
for all $v, w \in \HH \otimes M(\odd)$.
\end{lemma}

We decompose the Dirac operator $\Dirac$ as $\Dirac = \Dp + \Dq$, where $\Dp, \Dq \in \UE(\gg) \otimes \Weyl$ are $\kk^{\CC}$-invariant elements defined in \eqref{eq::decomposition_Dirac}. Recall that
\begin{equation}
\Dp = 2(\ddp - \delp), \qquad \Dq = 2(\ddq - \delq).
\end{equation}
By Lemma~\ref{PropertiesH} and the identity $\omega(x_k) = -\partial_k$ in $\gg$ for all $k = 1, \dots, mn$, the operators $\ddp$ and $\delp$, as well as $\ddq$ and $\delq$, are skew-adjoint to each other up to sign with respect to the Hermitian form $\langle \cdot, \cdot \rangle_{\HH \otimes M(\odd)}$.

\begin{lemma} \label{Relationen}
The operators $\ddp$ and $\delp$, and similarly $\ddq$ and $\delq$, are skew-adjoint to each other with respect to the Hermitian form $\langle \cdot, \cdot \rangle_{\HH \otimes M(\odd)}$, \emph{i.e.},
\[
\langle \delta^{\pp_{1}} v, w \rangle_{\HH \otimes M(\odd)} 
= -\langle v, \mathrm{d}^{\pp_{1}} w \rangle_{\HH \otimes M(\odd)}, \qquad
\langle \delta^{\qq_{2}} v, w \rangle_{\HH \otimes M(\odd)} 
= -\langle v, \mathrm{d}^{\qq_{2}} w \rangle_{\HH \otimes M(\odd)}.
\]
In particular, the operators $\Dp$, $\Dq$, and $\Dirac$ are selfadjoint on $\HH \otimes M(\odd)$.
\end{lemma}

\begin{corollary} \label{iterativeKernel}
For all $k \in \ZZ_{+}$, the operators $\Dirac$, $\Dp$, and $\Dq$ satisfy:
\[
\Ker \Dirac = \Ker \Dirac^{k}, \qquad 
\Ker \Dp = \Ker \big( \Dirac^{\pp_{1}} \big)^{k}, \qquad 
\Ker \Dq = \Ker \big(\Dirac^{\qq_{2}}\big)^{k}.
\]
\end{corollary}

Another direct consequence of Lemma~\ref{Relationen} is an inequality for $\Dirac$, known as \emph{Parthasarathy's Dirac inequality}, or simply the \emph{Dirac inequality}.

\begin{proposition}[Parthasarathy's Dirac inequality] \label{prop::Dirac_inequality}
The square of the Dirac operator satisfies $\Dirac^2 \geq 0$ on $\HH \otimes M(\odd)$; that is,
\[
\langle \Dirac^2 v, v \rangle_{\HH \otimes M(\odd)} \geq 0 \quad \text{for all } v \in \HH \otimes M(\odd).
\]
\end{proposition}

\begin{proof}
Let $v \in \HH \otimes M(\odd)$ be arbitrary. By Lemma~\ref{Relationen}, the Dirac operator $\Dirac$ is selfadjoint with respect to the Hermitian form $\langle \cdot, \cdot \rangle_{\HH \otimes M(\odd)}$. Moreover, this form is positive definite by construction. Consequently,
\[
\langle \Diracv, \Diracv \rangle_{\HH \otimes M(\odd)} 
= \langle \Dirac^2 v, v \rangle_{\HH \otimes M(\odd)} \geq 0,
\]
which proves the claim.
\end{proof}

We will specify this inequality in Proposition~\ref{prop::specified_Dirac_inequality} for highest weight supermodules. 

Let $M$ be a simple $\gg$-supermodule equipped with a positive definite Hermitian form, and assume that $M$ is $\even$-semisimple. A natural question is: when does the Dirac operator $\Dirac$ act as a selfadjoint operator on $M \otimes M(\odd)$?

\begin{theorem} \label{DiracUnit}
Let $M$ be a simple $\gg$-supermodule equipped with a positive definite Hermitian form $\langle \cdot, \cdot \rangle_M$, such that $M_{\bar{0}}$ and $M_{\bar{1}}$ are mutually orthogonal. Assume that $M$ is $\even$-semisimple. Then the following are equivalent:
\begin{enumerate}
 \item[(i)] $(M, \langle \cdot, \cdot \rangle_M)$ is a unitarizable $\gg$-supermodule.
 \item[(ii)] The Dirac operator $\Dirac$ is selfadjoint with respect to $\langle \cdot, \cdot \rangle_{M \otimes M(\odd)}$.
\end{enumerate}
\end{theorem}

\begin{proof}
If $M$ is unitarizable with respect to $\langle \cdot, \cdot \rangle_M$, then $\Dirac$ is selfadjoint on $M \otimes M(\odd)$ by Lemma~\ref{Relationen}.

Conversely, assume $\Dirac$ is selfadjoint with respect to $\langle \cdot, \cdot \rangle_{M \otimes M(\odd)}$. It suffices to show that the operators $\partial_k \otimes 1$ and $x_k$ have adjoints $\omega(\partial_{k})$ and $\omega(x_{k})$, since they generate $\odd$ and $[\odd, \odd] = \even$. A direct calculation shows:
\[
\Dirac (1 \otimes \partial_k) - (1 \otimes \partial_k) \Dirac = -2(\partial_k \otimes 1), \qquad
(1 \otimes x_k) \Dirac - \Dirac (1 \otimes x_k) = 2(x_k \otimes 1).
\]
Taking adjoints $(\cdot)^{\dagger}$ with respect to $\langle \cdot,\cdot \rangle_{M}$ yields:
\[
\left( \Dirac (1 \otimes \partial_k) - (1 \otimes \partial_k) \Dirac \right)^{\dagger}
= (1 \otimes x_k) \Dirac - \Dirac (1 \otimes x_k),
\]
using Lemmas~\ref{lemm::adjoint_Weyl_supermodule} and~\ref{Relationen}. Consequently, for all $v, w \in M$, we compute:
\begin{equation*}
\begin{aligned}
\langle \partial_k v, w \rangle_M 
&= \langle (\partial_k \otimes 1)(v \otimes 1), w \otimes 1 \rangle_{M \otimes M(\odd)} \\
&= -\tfrac{1}{2} \langle (\Dirac (1 \otimes \partial_k) - (1 \otimes \partial_k) \Dirac)(v \otimes 1), w \otimes 1 \rangle_{M \otimes M(\odd)} \\
&= -\tfrac{1}{2} \langle v \otimes 1, ((1 \otimes x_k) \Dirac - \Dirac (1 \otimes x_k))(w \otimes 1) \rangle_{M \otimes M(\odd)} \\
&= -\langle v \otimes 1, (x_k \otimes 1)(w \otimes 1) \rangle_{M \otimes M(\odd)} \\
&= -\langle v, x_k w \rangle_M \\ &= \langle v, \omega(\partial_{k})w\rangle_{M}.
\end{aligned}
\end{equation*}
Analogously, one shows that $\langle x_{k} v, w \rangle_{M} = \langle v, \omega(x_{k}) w \rangle_{M}$. 
This completes the proof.
\end{proof}

In general, the (Kac) induction of a unitarizable $\even$-supermodule is not itself unitarizable. Moreover, it can happen that a simple $\gg$-supermodule $M$ restricts to a unitarizable $\even$-module, even though $M$ is not unitarizable as a $\gg$-supermodule. The Dirac operator provides a criterion that helps clarify this situation.

\begin{corollary} \label{DUnit}
Let $M$ be a simple $\gg$-supermodule. Assume that $M$ admits an infinitesimal character, is $\even$-semisimple and it is equipped with a $\omega\vert_{\even}$-contravariant positive definite Hermitian form $\langle \cdot, \cdot \rangle_M$ such that $M_{\bar{0}}$ and $M_{\bar{1}}$ are mutually orthogonal. Then $M$ is unitarizable as a $\gg$-supermodule if and only if the following conditions hold:
\begin{enumerate}
 \item[(i)] All eigenvalues of $\Dirac^2$ on $M \otimes M(\odd)$ are non-negative.
 \item[(ii)] For each eigenvalue $\lambda^2$ of $\Dirac^2$, and for all $v,w \in \Ker(\Dirac^{2}-\lambda^{2})$ we have
 \[
 \langle \Diracv, \Diracw \rangle_{M \otimes M(\odd)} = \lambda^2 \langle v, w \rangle_{M \otimes M(\odd)}.
 \]
\end{enumerate}
\end{corollary}

\begin{proof}
We adapt the argument from \cite[Corollary 2]{Dirac_unitarizability}. By assumption and Proposition~\ref{CompletelyReducible}, the $\even$-module $M \otimes M(\odd)$ decomposes completely into simples.

Moreover, by Theorem~\ref{thm::square_dirac} and Proposition~\ref{prop::Dirac_inequality}, the operator $\Dirac^2$ acts on each $\even$-isotypic component as a positive scalar $\lambda^2 \in \RR_{\geq 0}$. 

On each such component, $\Dirac$ has eigenvalues $\pm \lambda$, and these components are mutually orthogonal with respect to the Hermitian form. The operator $\Dirac$ is selfadjoint on $\Ker(\Dirac^2 - \lambda^2)$ if and only if the $+\lambda$ and $-\lambda$ eigenspaces are orthogonal. Indeed, for any $v \in \Ker(\Dirac^2 - \lambda^2)$, the span of $v$ and $\Diracv$ is $\Dirac$-invariant and contains the eigenvectors $\lambda v \pm \Diracv$ with eigenvalues $\pm \lambda$. These eigenspaces are orthogonal if and only if
\[
\langle \lambda v + \Diracv, \lambda w - \Diracw \rangle_{M \otimes M(\odd)} = 0
\]
for all $v, w \in \Ker(\Dirac^2 - \lambda^2)$. A direct computation shows that this is equivalent to
\[
\langle \Diracv, \Diracw \rangle_{M \otimes M(\odd)} = \lambda^2 \langle v, w \rangle_{M \otimes M(\odd)}. \qedhere
\]
\end{proof}

We decompose the $\UE(\gg) \otimes \Weyl$-supermodule $\HH \otimes M(\odd)$, where $\HH$ is a fixed unitarizable simple $\gg$-supermodule, with respect to the Dirac operator $\Dirac$. Since $\HH$ is unitarizable and simple, it is completely reducible as both a $\gg$- and a $\even$-supermodule (see Proposition~\ref{CompletelyReducible}). Moreover, the oscillator module $M(\odd)$ is also unitarizable and completely reducible as a $\even$-module, decomposing into a direct sum of unitarizable highest weight constituents (see Remark~\ref{rmk::decompsoition_oscillator_module}). Consequently, $\HH \otimes M(\odd)$ is a completely reducible unitarizable $\even$-supermodule. 

On any simple $\gg$-constituent of $\HH$, the quadratic Casimir acts as a scalar multiple of the identity. Similarly, on any simple $\even$-constituent of $\HH \otimes M(\odd)$, the even quadratic Casimir acts as a scalar by Dixmier’s Theorem \cite[Proposition 2.6.8]{dixmier1996enveloping}. By Theorem~\ref{thm::square_dirac}, it follows that $\Dirac^2$ is a semisimple operator on $\HH \otimes M(\odd)$, and we can decompose $\HH \otimes M(\odd)$ into eigenspaces $(\HH \otimes M(\odd))(c)$ corresponding to eigenvalue $c$, that is,
\begin{equation}
\HH \otimes M(\odd) = \bigoplus_{c} (\HH \otimes M(\odd))(c),
\end{equation}
where the direct sum is orthogonal with respect to the Hermitian form $\langle \cdot, \cdot \rangle_{\HH \otimes M(\odd)}$.

Each eigenspace is a $\even$-supermodule since $\Dirac$ is $\even$-invariant (Lemma~\ref{lemm::properties_relative_Dirac}). In particular, the eigenspace with eigenvalue $c = 0$ is precisely the kernel of $\Dirac$, as $\Ker \Dirac = \Ker \Dirac^k$ for all $k \in \ZZ_{+}$.

More generally, such a decomposition with respect to $\Dirac^2$ exists for any $\gg$-supermodule $\HH$ of finite length that is $\even$-semisimple.

\begin{lemma} \label{lemm::weight_space_decomp}
Let $M$ be a $\gg$-supermodule that admits an infinitesimal character and is $\even$-semisimple. Then, as a $\even$-supermodule, it decomposes into a direct sum of generalized $\Dirac^2$-eigenspaces:
\[
M \otimes M(\odd) = \bigoplus_{c \in \CC} (M \otimes M(\odd))(c),
\]
where
\[
(M \otimes M(\odd))(c)\coloneqq \left\{ v \in M \otimes M(\odd) \mid \exists n \in \ZZ_{+} \text{ such that } (c \cdot \operatorname{id} - \Dirac^2)^n v = 0 \right\}.
\]
\end{lemma}

The proof proceeds \emph{mutatis mutandis} as in \cite[Corollary 3.3]{Dirac_cohomology_category_O}.

\begin{lemma} \label{lemm::Hodge_decomposition}
Let $\HH$ be a unitarizable simple $\gg$-supermodule. Then there is an orthogonal decomposition:
\[
\HH \otimes M(\odd) = \Ker \Dirac^2 \oplus \Im \Dirac^2,
\]
where the direct sum is orthogonal with respect to the Hermitian form $\langle \cdot, \cdot \rangle_{\HH \otimes M(\odd)}$.
\end{lemma}

\begin{proof}
First, we recall the eigenspace decomposition with respect to $\Dirac^2$:
\[
\HH \otimes M(\odd) = (\HH \otimes M(\odd))(0) \oplus \bigoplus_{c \neq 0} (\HH \otimes M(\odd))(c),
\]
where the zero eigenspace satisfies $(\HH \otimes M(\odd))(0) = \Ker \Dirac = \Ker \Dirac^{2}$.

Since $\Dirac^{2}$ is a semisimple and selfadjoint operator on $(\HH \otimes M(\odd),\langle\cdot, \cdot \rangle_{\HH \otimes M(\odd)})$, it follows that $\Im \Dirac^{2} = (\Ker \Dirac^{2})^\perp$, \emph{i.e.}, the image and kernel are orthogonal complements. Hence,
\[
\Im \Dirac^{2} = (\Ker \Dirac^{2})^\perp = (\HH \otimes M(\odd))(0)^\perp = \bigoplus_{c \neq 0} (\HH \otimes M(\odd))(c).
\]
This yields the desired orthogonal decomposition.
\end{proof}

\subsection{Dirac cohomology and unitarity} 
The Dirac cohomology functor
\begin{equation}
\DC(\cdot) : \gsmod \to \gg_{\bar{0},\Updelta}\textbf{-smod} \cong \even\textbf{-smod}
\end{equation}
assigns to a $\gg$-supermodule $M$ its Dirac cohomology $\DC(M)\coloneqq \Ker \Dirac \big/ \left( \Ker \Dirac \cap \Im \Dirac \right)$. In this section, we investigate the Dirac cohomology of unitarizable $\gg$-supermodules and show that Dirac cohomology provides a complete characterization.

\subsubsection{Basic facts about Dirac cohomology of unitarizable supermodules}

For unitarizable supermodules, Dirac cohomology admits a particularly simple description: it coincides with the kernel of the Dirac operator.

\begin{proposition} \label{DiracGleichKernel}
Let $\HH$ be a unitarizable $\gg$-supermodule. Then
\[
\DC(\HH) = \Ker \Dirac.
\]
\end{proposition}

\begin{proof}
We prove that $\Ker \Dirac \cap \Im \Dirac = \{0\}$. Let $v \in \Ker \Dirac \cap \Im \Dirac $, so that$ \Dirac\! v = 0$ and there exists $w \in \HH \otimes M(\odd)$ with $v = \Dirac\! w$. Since $\HH \otimes M(\odd)$ is equipped with a positive definite Hermitian form, we compute:
\[
\langle v, v \rangle = \langle \Dirac\! w, v \rangle = \langle w, \Dirac\! v \rangle = 0.
\]
Thus $v = 0$, and the claim follows.
\end{proof}

As an immediate consequence, Dirac cohomology is additive on direct sums of unitarizable supermodules.

\begin{lemma}
Let $\HH_1, \HH_2$ be unitarizable $\gg$-supermodules. Then
\[
\DC(\HH_1 \oplus \HH_2) = \DC(\HH_1) \oplus \DC(\HH_2).
\]
\end{lemma}

We now express the Dirac operator as $\Dirac = \Dp + \Dq$, where
$
\Dp = 2(\ddp - \delp), \ \Dq = 2(\ddq - \delq),
$
and $\ddp, \ddq, \delp, \delq$ are $\kk^{\CC}$-invariant operators. As shown in Lemma~\ref{Relationen}, these operators are skew-adjoint to one another with respect to the Hermitian form $\langle \cdot, \cdot \rangle_{\HH \otimes M(\odd)}$. This decomposition provides a more explicit understanding of the kernel of $\Dirac$. We now record some structural relations among the constituent operators.

\begin{lemma}\label{Orthogonality} Let $\HH$ be a unitarizable $\gg$-supermodule. 
\begin{enumerate}
 \item[a)] The operators $\ddp, \delp, \ddq, \delq$ square to zero.
 \item[b)] The operators $\ddp$ and $\delq$, and the operators $\ddq$ and $\delp$ commute, \emph{i.e.}, $[\ddp,\delq] = 0$ and $[\ddq, \delp] = 0$.
 \item[c)] With respect to the form $\langle \cdot,\cdot \rangle_{M \otimes M(\odd)}$, the following holds:
 \begin{enumerate}
 \item[(i)] $\Im \ddp$ is orthogonal to $\Ker \delp$ and $\Im \delp$; $\Im \delp$ is orthogonal to $\Ker \ddp$. 
 \item[(ii)] $\Im \ddq$ is orthogonal to $\Ker \delq$ and $\Im \delq$; $\Im \delq$ is orthogonal to $\Ker \ddq$.
 \end{enumerate}
 \item[d)] $\Ker (\Dp)^{2} = \Ker \Dp = \Ker \ddp \cap \Ker \delp$, and $\Ker (\Dq)^{2} = \Ker \Dq = \Ker \ddq \cap \Ker \delq$.
\end{enumerate}
\end{lemma}

\begin{proof}
a) This is a direct consequence of the fact that $\pp_{1,2}$ and $\qq_{1,2}$ are abelian Lie supersubalgebras of $\gg$.

b) This is a direct consequence of $[\partial_{k},x_{l}] = 0$ unless $k = l$.

c) We only prove that $\Im \ddp$ and $\Ker \delp$ are orthogonal; the rest can be proved similarly using a) and b). First, let $v \in \Im \ddp$ and $w \in \Ker \delp$. Then there exists a non-trivial $v' \in \HH \otimes M(\odd)$ such that $\ddp v' = v$, and consequently by Lemma~\ref{Relationen}
$$
\langle v,w \rangle_{\HH\otimes M(\odd)} = \langle \ddp v',w \rangle_{\HH\otimes M(\odd)} = -\langle v', \delp w \rangle_{\HH\otimes M(\odd)} = 0,
$$
\emph{i.e.}, $\langle \Im \ddp,\Ker \delp \rangle_{\HH\otimes M(\odd)} = 0$. 

d) The operators $\Dp$ and $\Dq$ are selfadjoint by Lemma~\ref{Relationen}, and therefore $\Ker \Dp = \Ker (\Dp)^{2}$ and $\Ker \Dq = \Ker (\Dq)^{2}$. We prove that $\Ker \Dp = \Ker \ddp \cap \Ker \delp$. Let $v \in \Ker \Dp$, then $\Dp v = 2(\ddp - \delp)v = 0$, \emph{i.e.}, $\ddp v = \delp v$. By b), $\Im \ddp$ and $\Im \delp$ are orthogonal to each other, hence $v \in \Ker \delp \cap \Ker \ddq$. The other inclusion is trivial. Analogously, the equality $\Ker \Dq = \Ker \ddq \cap \Ker \delq$ follows.
\end{proof}

In summary, we can describe the Dirac cohomology of unitarizable supermodules in terms of $\Ker \Dp$ and $\Ker \Dq$.

\begin{lemma} The Dirac cohomology of a unitarizable $\gg$-supermodule $\HH$ is
$$
\DC(\HH) = \Ker \Dp \cap \Ker \Dq.
$$
\end{lemma}

\begin{proof}We decompose the Dirac operator as $\Dirac = \Dp + \Dq$, such that 
$$
\Ker \Dirac = \{ v \in \HH \otimes M(\odd) : \Dp v = -\Dq v\}.
$$
Here, note that $\left(\Ker \Dp \cap \Ker \Dq\right) \subset \{ v \in \HH \otimes M(\odd) : \Dp v = -\Dq v\}$. Assume $v\coloneqq m \otimes P \in \HH \otimes M(\odd)$ satisfies $w\coloneqq \Dp v = -\Dq v$. Then
\begin{equation*}
\begin{split}
0 &\leq \langle w, w \rangle_{\HH \otimes M(\odd)} \\ 
&= -\langle \Dp v, \Dq v \rangle_{\HH \otimes M(\odd)} \\ 
&= -\sum_{k=1}^{pn} \sum_{l=pn+1}^{mn} \langle (\partial_{k}-x_{k})m, (\partial_{l}-x_{l})m \rangle_{\HH} \langle (x_{k}-\partial_{k})P, (x_{l}-\partial_{l})P \rangle_{M(\odd)}.
\end{split}
\end{equation*}
However, for all $1 \leq k \leq pn$ and $p+1 \leq l \leq mn$, we have
$$
\langle (x_{k}-\partial_{k})P \rangle_{\HH},(x_{l}-\partial_{l})P \rangle_{M(\odd)} = 0
$$
by the construction of the Bargmann--Fock form and $k \neq l$. We conclude $w = 0$, as the Hermitian form is positive definite, \emph{i.e.}, $\Dp v = 0$ and $\Dq v = 0$ or $v \in (\Ker \Dp \cap \Ker \Dq)$. The assertion now follows from Proposition~\ref{DiracGleichKernel}.
\end{proof}

In Section~\ref{subsec::Dirac_Induction}, we introduced a left exact functor $\DC'(\cdot)$, which we also referred to as Dirac cohomology. It coincides with the Dirac cohomology $\DC(\cdot)$ on unitarizable $\gg$-supermodules.

\begin{corollary} \label{DiracsCoincide}
 Let $\HH$ be a unitarizable $\gg$-supermodule. Then the Dirac cohomologies $\DC(\HH)$ and $\DC'(\HH)$ coincide. In particular, $\DC(\cdot)$ is left exact in the category of unitarizable $\gg$-supermodules. 
\end{corollary}
\begin{proof} The Dirac cohomology of a unitarizable $\gg$-supermodule $\HH$ is $\DC(\HH)=\Ker \Dirac=\Ker \Dirac^{2}$ by Corollary~\ref{iterativeKernel} and Proposition~\ref{DiracGleichKernel}. Moreover, $\Dirac^{2}$ commutes with any element of $(\UE(\gg)\otimes \Weyl)^{\UE(\even)}$ by Theorem~\ref{thm::square_dirac}, and it is an element of $\mathscr{I}$, where
 $\mathscr{I}$ is the two-sided ideal in $(\UE(\gg)\otimes \Weyl)^{\UE(\even)}$ generated by the Dirac operator $\Dirac$. The Dirac cohomology $\DC'(\HH)$ is defined as the $\mathscr{I}$-invariants in $\HH\otimes M(\odd)$. The relation $\DC'(\HH)\subset \Ker \Dirac^{2}$ is immediate. Conversely, for any $v\in \DC(\HH)$, we have $DXv=0$ for any $X\in (\UE(\gg)\otimes \Weyl)^{\UE(\even)}$, as $\Dirac^{2}\!Xv=X\!\Dirac^{2}\!v=0$. This shows $\DC(\HH) = \DC'(\HH)$. The second assertion follows by Proposition~\ref{DiracInduction}.
\end{proof}

The Dirac cohomology $\DC(\HH)$ inherits the structure of a unitarizable $\even$-supermodule. Since the oscillator module $M(\odd)$ is itself a unitarizable $\even$-supermodule (Proposition~\ref{prop:oscillator_module_even_semisimple}), it follows that $\HH \otimes M(\odd)$ is also a unitarizable $\even$-supermodule. Here the action is induced by the diagonal embedding of $\even$ into $\mathcal{W}(\gg)$. Since the Dirac operator $\Dirac$ is $\even$-invariant, it follows that $\Ker \Dirac$ likewise carries the structure of a unitarizable $\even$-supermodule. 

\begin{proposition} \label{HDseimsimple}
 The Dirac cohomology $\DC(\HH)$ of a unitarizable $\gg$-supermodule $\HH$ is a unitarizable $\even$-supermodule. In particular, if $\HH$ is simple, it decomposes completely into unitarizable highest weight $\even$-supermodules. 
\end{proposition}

\subsubsection{Computation of Dirac cohomology} Having established the basic properties of the Dirac cohomology for unitarizable supermodules, we now proceed to compute it explicitly. We begin by noting that general highest weight $\gg$-supermodules possess non-trivial Dirac cohomology. 

\begin{proposition} \label{prop::HW_DC}
 Let $M$ be a highest weight $\gg$-supermodule with highest weight $\Lambda$. Then $\DC(M)$ contains a highest weight $\even$-supermodule with highest weight $\Lambda-\rho_{\bar{1}}$ that occurs with multiplicity one. In particular, $\DC(M) \neq \{0\}$.
\end{proposition}

\begin{proof}
Let $v_{\Lambda}$ be the highest weight vector of $M$, meaning that $\nn^{+}v_{\Lambda} = 0$, and specifically $\nn_{\bar{1}}^{+}v_{\Lambda} = 0$. Since $1$ is constant and annihilated by $\partial_{k}$, the vector $v_{\Lambda} \otimes 1$ lies in the kernel of 
$
\Dirac = 2 \sum_{k=1}^{mn} (\partial_{k} \otimes x_{k} - x_{k} \otimes \partial_{k}).
$
Further, we assert that $v_{\Lambda} \otimes 1$ generates a highest weight $\UE(\even)$-supermodule.

Any element $X \in \even$ acts on $v_{\Lambda}\otimes 1$ via the diagonal embedding:
 \begin{equation*}
 Xv_{\Lambda} \otimes 1 + v_{\Lambda} \otimes \alpha(X) 1,
 \end{equation*}
 and we recall: 
 \begin{multline*}
\alpha(X)= \sum_{k,j=1}^{mn}(B(X,[\partial_{k},\partial_{j}])x_{k}x_{j}+B(X,[x_{k},x_{j}])\partial_{k}\partial_{j}) 
\\ -\sum_{k,j=1}^{mn}2B(X,[x_{k},\partial_{j}])x_{j}\partial_{k}-\sum_{l=1}^{mn}B(X,[\partial_{l},x_{l}]).
\end{multline*} 

By Lemma~\ref{lemm::commutation_relations}, we have the commutation relations $[\partial_{k},\partial_{j}] \in \nn_{\bar{0}}^{+}$, $[x_{k},x_{j}] \in \nn_{\bar{0}}^{-}$ and $[x_{k},\partial_{j}] \in \hh$. Let $X \in \nn^{+}_{\bar{0}}$. Then the definition of $B(\cdot,\cdot)$ forces
\begin{equation*}
 Xv_{\Lambda} \otimes 1 + v_{\Lambda} \otimes \alpha(X) 1 = v_{\Lambda} \otimes \sum_{k,j=1}^{mn} B(X,[x_{k},x_{j}])\partial_{k}\partial_{j}1=0,
\end{equation*}
where we use that $\nn_{\bar{0}}^{+}v_{\Lambda} = 0$.

Any $H \in \hh$ acts on $v_{\Lambda} \otimes 1$ by 
\begin{equation*}
\begin{split}
 Hv_{\Lambda}\otimes 1 + v_{\Lambda} \otimes \alpha(H)v_{\Lambda} &= \Lambda(H)v_{\Lambda}\otimes 1 - v_{\Lambda}\otimes \sum_{l=1}^{mn}B(H,[\partial_{l},x_{l}])1 \\ &=\Lambda(H)v_{\Lambda} \otimes 1 - v_{\Lambda} \otimes \sum_{\alpha \in \Delta_{\bar{1}}^{+}} \frac{1}{2}\alpha(H)1 \\ &= (\Lambda-\rho_{\bar{1}})(H)(v_{\Lambda}\otimes 1).
 \end{split}
\end{equation*}

Hence, $v_{\Lambda}\otimes 1 \in \DC(M)$ generates a (simple) highest weight $\even$-supermodule, which is in particular unitarizable by Proposition~\ref{HDseimsimple}. This module appears with multiplicity one, as the weight spaces of $v_{\Lambda}$ and $1$ are one-dimensional. 
\end{proof}

The following corollary is an immediate consequence of Proposition~\ref{CompletelyReducible}, Theorem~\ref{thm::HW_property_M} and Proposition~\ref{prop::HW_DC}.

\begin{corollary}
 Let $\HH$ be a unitarizable $\gg$-supermodule. Then $\DC(\HH)$ is non-trivial.
\end{corollary}

 To compute $\DC(\HH)$ explicitly, we use Theorem~\ref{thm::square_dirac}. For that, we relate the constant $C$ with the Weyl vector $\rho = \rho_{\bar{0}} - \rho_{\bar{1}}$. 

\begin{lemma} \label{lemm::Value_C}
 The constant is $C = -(\rho_{\bar{1}}-2\rho_{\bar{0}},\rho_{\bar{1}})$.
\end{lemma}

\begin{proof}
Let $\HH$ be a non-trivial unitarizable highest weight $\gg$-supermodule with highest weight $\Lambda$. Such a supermodule exists \cite{jakobsen1994full}. By Proposition~\ref{prop::HW_DC}, there exists a $\even$-supermodule in $\DC(\HH)$ with highest weight $\Lambda-\rho_{\bar{1}}$. As $\DC(\HH)=\Ker \Dirac = \Ker \Dirac^{2}$, we have by Theorem~\ref{VoganConjecture}:
$$
0 = -(\Lambda+2\rho, \Lambda) + (\Lambda-\rho_{\bar{1}}+2\rho_{\bar{0}},\Lambda-\rho_{\bar{1}}) + C,
$$
and a direct calculation yields 
$$(\Lambda-\rho_{\bar{1}}+2\rho_{\bar{0}},\Lambda-\rho_{\bar{1}}) = (\Lambda+2\rho,\Lambda) +(\rho_{\bar{1}}-2\rho_{\bar{0}},\rho_{\bar{1}}).$$
This finishes the proof.
\end{proof}

As a direct consequence, using the expression $\Dirac^{2}=-\Omega_{\gg}\otimes 1+\Omega_{\even,\Delta}-C$, we obtain an explicit form of the Dirac inequality for unitarizable simple supermodules. 

\begin{proposition}\label{prop::specified_Dirac_inequality}
    Let $\HH$ be a unitarizable simple $\gg$-supermodule of highest weight $\Lambda$. If $L_{0}(\mu)$ is a $\even$-constituent of $\HH\otimes M(\odd)$ of highest weight $\mu$, then 
    \[
    (\mu+2\rho,\mu)\geq (\Lambda+2\rho,\Lambda).
    \]
    In particular, $L_{0}(\mu)$ belongs to $\DC(\HH)$ if and only if
    \[
    (\Lambda+2\rho,\Lambda)=(\mu+2\rho,\mu).
    \]
\end{proposition}

\begin{proof}
 The square $\Dirac^{2} = - \Omega_{\gg}\otimes 1 + \Omega_{\even,\Delta} +C$ acts on $L_{0}(\mu)$ by the scalar
$$
-(\Lambda+2\rho,\Lambda) + (\mu-\rho_{\bar{1}}+2\rho_{\bar{0}},\mu-\rho_{\bar{1}}) - (\rho_{\bar{1}}-2\rho_{\bar{0}},\rho_{\bar{1}}) = -(\Lambda+2\rho,\Lambda) + (\mu + 2\rho,\mu).
$$
The first part of the proof follows by Proposition~\ref{prop::Dirac_inequality} and the second part by $\DC(\HH) = \ker \Dirac = \ker \Dirac^{2}$.
\end{proof}

We now determine the Dirac cohomology of unitarizable $\gg$-supermodules. Since any unitarizable $\gg$-supermodule is completely reducible and $\DC(\cdot)$ is additive, it suffices to consider unitarizable simple $\gg$-supermodules. Our computation is based on the following consequence of Vogan’s theorem.

\begin{lemma} \label{lemm::relation_HW_constituents}
 Let $\HH$ be a unitarizable highest weight $\gg$-supermodule with highest weight $\Lambda$. If the Dirac cohomology $\DC(\HH)$ contains a nonzero $\even$-supermodule with even infinitesimal character $\chi_{\lambda}^{\bar{0}}$ for some $\lambda \in \hh^{\ast}$, then 
 \[
 \Lambda - \rho_{\bar{1}} = w(\lambda + \rho_{\bar{0}})-\rho_{\bar{0}}
 \]
 for some $w \in W$.
\end{lemma}

\begin{proof}
By Proposition~\ref{prop::HW_DC}, the supermodule $\DC(\HH)$ contains a $\even$-constituent of highest weight $\Lambda-\rho_{\bar{1}}$. 
If $\lambda$ is the highest weight of another $\even$-constituent, then by Vogan's theorem (Theorem~\ref{VoganConjecture}) we have
\[
\chi^{\bar{0}}_{\Lambda-\rho_{\bar{1}}}(\zeta(z))
 = \chi^{\bar{0}}_{\lambda}(\zeta(z)),
 \qquad z \in \mathfrak{Z}(\gg).
\]
If $z \in \mathfrak{Z}(\even)$, then $\zeta(z)=z$, and therefore the above equality holds for all $z\in\mathfrak{Z}(\even)$. 
In particular,
\[
\chi^{\bar{0}}_{\Lambda-\rho_{\bar{1}}}
 = \chi^{\bar{0}}_{\lambda},
\]
which is the case if and only if there exists $w\in W$ such that (\emph{cf.}~Section~\ref{subsubsec::infinitesimal_characters})
\[
\Lambda - \rho_{\bar{1}}
 = w \cdot \lambda
 = w(\lambda + \rho_{\bar{0}}) - \rho_{\bar{0}}.
\qedhere
\]
\end{proof}

\begin{theorem}
\label{thm::Dirac_cohomology_simple_supermodules}
The Dirac cohomology of a non-trivial unitarizable simple $\gg$-supermodule $\HH$ with highest weight $\Lambda$ is 
$$
\DC(\HH) \cong L_{0}(\Lambda-\rho_{\bar{1}}).
$$
\end{theorem}
\begin{proof}
 We decompose $\DC(\HH)$ into its $\even$-constituents by Proposition~\ref{HDseimsimple}. Note that it decomposes discretely since $\HH$ is simple and by Remark~\ref{rmk::decompsoition_oscillator_module}. Then the module $L_{0}(\Lambda-\rho_{\bar{1}})$ is a simple constituent by Proposition~\ref{prop::HW_DC}, and unitarizable. Moreover, it is a relative holomorphic $\even$-supermodule
 by Harish-Chandra's condition:
 $$
 (\Lambda-\rho_{\bar{1}}+ \rho_{\bar{0}},\epsilon_{1}-\epsilon_{m}) = \lambda_{1}-\lambda_{m} + m-1 -n < 0
 $$
 as $n\geq m$ and $\lambda_{1}-\lambda_{m}\leq 0$, \emph{i.e.}, $\Lambda - \rho_{\bar{1}} \in \Dis$, where $\Dis$ denotes the set of all Harish-Chandra parameters of relative holomorphic discrete series $\even$-modules described in Section~\ref{subsubsec::relative_holomorphic}. In particular, as $\Delta_{\bar{0}}^{+}$ is fixed, $\Lambda-\rho_{\bar{1}}$ is the unique highest weight in its $W$-linkage class of a unitarizable highest weight $\even$-supermodule~\ref{lemm::D_Weyl_Orbit}.

 Any simple $\even$-constituent $L_{0}(\mu)$ is a highest weight $\even$-supermodule by Proposition~\ref{HDseimsimple}, and the highest weight $\mu$ is of the form
 $$
 w\cdot (\Lambda-\rho_{\bar{1}}) = \mu,
 $$
 by Lemma~\ref{lemm::relation_HW_constituents}. Thus, by uniqueness, $w$ must be the identity, and $\Lambda-\rho_{\bar{1}} = \mu$. In addition, the multiplicity is one by Proposition~\ref{prop::HW_DC}. This concludes the proof. 
\end{proof}

\begin{remark}
The proof of the theorem applies verbatim to any $\gg_{\bar{0}}$–semisimple simple highest weight $\gg$-supermodule $M$ whose highest weight $\Lambda$ is the highest weight of a unitarizable highest weight $\gg_{\bar{0}}$–module. In addition, the preceding reasoning extends, without essential modification, to the case where $\Lambda$ is regular and dominant integral. 
\end{remark}

\begin{corollary} \label{cor::even_constituents_Dirac_relative_holomorphic}
 Let $\HH$ be a unitarizable $\gg$-supermodule. Then any $\even$-constituent in $\HH \otimes M(\odd)$ belongs to the relative holomorphic discrete series. 
\end{corollary}

\begin{proof}
 Any $\even$-constituent in $\HH \otimes M(\odd)$ has highest weight 
 \[
 \mu = \Lambda - \rho_{\bar{1}}-\sum_{i} n_{i} \alpha_{i} - \sum_{j}m_{j} \beta_{j}, \qquad n_{i},m_{j} \in \ZZ_{+}, \quad \alpha_{i} \in \Delta_{\bar{0}}^{+}, \ \beta_{j} \in \Delta_{\bar{1}}^{+}.
 \]
 A direct calculation yields for all $i,j$:
 \[
 (\alpha_{i}, \epsilon_{1}-\epsilon_{m}) \geq 0, \qquad (\beta_{j}, \epsilon_{1}-\epsilon_{m})\geq 0
 \]
 which implies $(\mu+\rho_{\bar{0}},\epsilon_{1}-\epsilon_{m})<0$ since $(\Lambda-\rho_{\bar{1}}+\rho_{\bar{0}}, \epsilon_{1}-\epsilon_{m})<0$. 
\end{proof}

The $\even$-supermodule generated by $1 \in M(\odd)$ is entirely concentrated in even parity due to the $\mathbb{Z}_2$-grading of $M(\odd)$. This leads to the following corollary.

\begin{corollary}\label{DiracOfSimple}
 Let $\HH$ be a non-trivial unitarizable simple $\gg$-supermodule. Then 
 \begin{equation*}
 \DC(\HH)\cong \DC^{+}(\HH).
 \end{equation*}
\end{corollary}

Furthermore, unitarizable $\gg$-supermodules are uniquely determined by their Dirac cohomology. 

\begin{theorem} \label{Unique}
Let $\HH_{1}$ and $\HH_{2}$ be two unitarizable $\gg$-supermodules. Then $\HH_{1}\cong \HH_{2}$ as $\gg$-supermodules if and only if $\DC(\HH_{1})\cong \DC(\HH_{2})$ as $\even$-supermodules.
\end{theorem}
\begin{proof}
 Unitarizable $\gg$-supermodules are completely reducible, and the Dirac cohomology is additive. It therefore suffices to consider unitarizable highest weight $\gg$-supermodules, say $\HH_{1}=L(\Lambda_1)$ and $\HH_{2}=L(\Lambda_2)$ for some $\Lambda_1,\Lambda_2\in \hh^{\ast}$. 

Then, by Theorem~\ref{thm::Dirac_cohomology_simple_supermodules}, the associated Dirac cohomology is $\DC(\HH_{1}) = L_{0}(\Lambda_{1}-\rho_{\bar{1}})$ and $\DC(\HH_{2}) = L_{0}(\Lambda_{2}-\rho_{\bar{1}})$. The statement now follows, since two highest weight supermodules with respect to $\Delta_{\bar{0}}^{+}$ are isomorphic if and only if they have the same highest weight.
\end{proof}

Finally, we apply the preceding results to compute the Dirac cohomology of highest weight $\gg$-supermodules $M$ whose highest weight $\Lambda$ is the highest weight of a unitarizable highest weight $\even$-module. For an explicit discussion of these supermodules, we refer to \cite{SchmidtWalcher}.

Fix such a supermodule $M$. Then it has a composition series (Proposition~\ref{prop::Properties_Verma_Supermodule}):
$$
\{0\} = M_{0}\subset M_{1}\subset \dotsc \subset M_{n}=M
$$
such that each $M_{i+1}/M_{i}$ is a simple highest weight $\gg$-supermodule, say $L(\Lambda_{i})$. In particular, using \cite[Theorem 2.5, Corollary 2.7]{jakobsen1994full}, any $\Lambda_{i}$ is of the form $\Lambda_{i}=\Lambda-\gamma_{i}$, where $\gamma_{i} \in \Gamma$ is a sum of pairwise distinct odd positive roots. 

\begin{proposition} \label{DecompositionSimple} The Dirac cohomology of $M$ is given by 
$$
\DC(M) =\bigoplus_{i}\DC(L(\Lambda_{i})) \cong \bigoplus_{i} L_{0}(\Lambda_{i}-\rho_{\bar{1}}).
$$
\end{proposition}
\begin{proof}
Recall that we work in the category of Harish-Chandra (super)modules. In particular, it is well-known that for $\even$-modules $L_{0}(\Lambda)$ and $L_{0}(\Lambda')$ with $\Lambda,\Lambda'\in\Dis$, one has
\begin{equation} \label{eq::Ext_eins_HC}
\Ext^{1}(L_{0}(\Lambda),L_{0}(\Lambda'))=0
\end{equation}
in the category of Harish-Chandra modules.

We first consider the case $k=2$. The short exact sequence
\[
0\to M_{1}=L(\Lambda_{0})\to M_{2}\to L(\Lambda_{1})\to 0
\]
induces, by Lemma~\ref{exactSequence}, Corollary~\ref{DiracOfSimple}, and Theorem~\ref{thm::Dirac_cohomology_simple_supermodules}, an exact sequence
\[
0\to L_{0}(\Lambda_{0}-\rho_{\bar1})\to \DC^{+}(M_{2})\to L_{0}(\Lambda_{1}-\rho_{\bar1})\to 0,
\]
where parity is suppressed. By Corollary~\ref{cor::even_constituents_Dirac_relative_holomorphic}, both $\Lambda_{0}-\rho_{\bar1}$ and $\Lambda_{1}-\rho_{\bar1}$ lie in $\Dis$. Hence $\Ext^{1}\bigl(L_{0}(\Lambda_{1}-\rho_{\bar1}),L_{0}(\Lambda_{0}-\rho_{\bar1})\bigr)=0$ and the sequence splits, so that
\[
\DC^{+}(M_{2})\cong \DC(L(\Lambda_{0}))\oplus \DC(L(\Lambda_{1})).
\]
Moreover, the six-term exact sequence of Lemma~\ref{exactSequence} gives
$
\DC^{-}(M_{2})=0,
$
since $\DC^{-}(L(\Lambda_{0}))=\DC^{-}(L(\Lambda_{1}))=0$. Thus
\[
\DC(M_{2})\cong \DC(L(\Lambda_{0}))\oplus \DC(L(\Lambda_{1})).
\]
The general case follows by induction, using the same argument and the facts that $\Lambda_i-\rho_{\bar1}\in\Dis$ for all $i$, and \eqref{eq::Ext_eins_HC}.
\end{proof}

\section{Complementary perspectives} \label{sec::complementary_perspectives}

In this section, we consider several complementary viewpoints on Dirac cohomology. 
First, we provide a new characterization of unitarity via the Dirac inequality, which in turn yields a novel classification the full set of unitarizable simple $\gg$-supermodules. 
Next, we introduce an analogue of Kostant’s cohomology and show that, as $\kk^{\CC}$-modules, it agrees with Dirac cohomology. 
We then define the Dirac index and establish that it equals the Euler characteristic of Dirac cohomology. 
Finally, by means of Kostant’s cohomology and the Dirac index, we obtain $\kk^{\CC}$-character formulas.

\subsection{\texorpdfstring{$\even$}\ -Decomposition and unitarity}\label{subsec::Dirac_and_even_decomposition} 

For highest weight $\gg$-supermodules, an important problem is the decomposition under $\even$. This passes from the super setting to the reductive Lie algebra $\even$, where the representation theory is classical. In general, if $M$ is a highest weight $\gg$-supermodule of highest weight $\Lambda$, then $M$ admits a $\even$-filtration whose composition factors are simple highest weight $\even$-supermodules of the form $L_0(\Lambda-\gamma)$, where $\gamma$ is a sum of distinct positive odd roots. In what follows, the $\ZZ_{2}$-grading of the corresponding modules is left implicit, and we simply refer to them as $\even$-modules.

For unitarizable supermodules, this picture becomes more rigid. By Proposition~\ref{CompletelyReducible}, such a module is completely reducible over $\even$. At the same time, the Dirac inequality is strict on every $\even$-constituent of $M\otimes M(\odd)$. In particular,
\begin{equation}
(\Lambda-\gamma+2\rho,\Lambda-\gamma)>(\Lambda+2\rho,\Lambda)>0
\qquad \forall\,\gamma\neq 0
\end{equation}
for every highest weight $\Lambda-\gamma$ occurring in the $\even$-decomposition. This suggests that the strict Dirac inequality should detect unitarity.

We now show that this is indeed the case. In Section~\ref{subsec::Realization_Shapovalov}, we proved that every simple highest weight $\gg$-supermodule $M$ carries a unique non-degenerate contravariant Hermitian form, induced from the Shapovalov form on the corresponding Verma supermodule. Thus $M$ is unitarizable if and only if this form is positive definite. The main point is that this positivity is equivalent to the strict Dirac inequality induced from every $\even$-composition factor of $M$ different from $L_{0}(\Lambda)$, together with the unitarizability of $L_{0}(\Lambda)$ as a $\even$-module.

In what follows, $M$ denotes a simple highest weight $\gg$-supermodule with highest weight $\Lambda$, where $\Lambda$ is the highest weight of a unitarizable highest weight $\even$-module.

We transfer the question of positive definiteness of the Shapovalov form $\langle\cdot,\cdot\rangle_M$ on $M$ to that of the form $\langle\cdot,\cdot\rangle_{M\otimes M(\odd)}$ on $\langle M\otimes 1\rangle_{\even}$, where $\langle M\otimes 1\rangle_{\even}$ denotes the $\even$-submodule of $M\otimes M(\odd)$ generated by $M\otimes 1$ under the diagonal action induced by \eqref{eq::general_embedding}. On this module, positive definiteness is equivalent to strictness of the Dirac inequality.

\begin{lemma}\label{lemm::shift_unitarity}
A simple highest weight $\gg$-supermodule $M$ is unitarizable if and only if $\langle M\otimes 1\rangle_{\even}$ is a unitarizable $\even$-module with respect to the restriction of $\langle\cdot,\cdot\rangle_{M\otimes M(\odd)}$.
\end{lemma}

\begin{proof}
Assume that $M$ is unitarizable. Then $M\otimes M(\odd)$ is a unitarizable $\gg$-module. Hence, under the diagonal action, it is also a unitarizable $\even$-module. Therefore its $\even$-submodule $\langle M\otimes 1\rangle_{\even}$ is unitarizable.

Conversely, assume that $\langle M\otimes 1\rangle_{\even}$ is unitarizable with respect to the restriction of $\langle\cdot,\cdot\rangle_{M\otimes M(\odd)}$. It suffices to show that the Shapovalov form $\langle\cdot,\cdot\rangle_M$ on $M$ is positive definite. For any $0\neq v\in M$ one has $v\otimes 1\in\langle M\otimes 1\rangle_{\even}$. Then
\[
0<\langle v\otimes 1,v\otimes 1\rangle_{M\otimes M(\odd)}=\langle v,v\rangle_M\langle 1,1\rangle_{M(\odd)}=\langle v,v\rangle_M.
\]
Thus $\langle\cdot,\cdot\rangle_M$ is positive definite. Hence $M$ is unitarizable.
\end{proof}

Let $\Gamma$ be the set of all sums of distinct positive odd roots. By~\cite[Theorem~10.4.5]{musson2012lie}, there exists a subset $S(\Lambda)\subseteq\Gamma$ such that, regarded as a $\even$-module, $M$ admits a filtration whose simple highest weight subquotients have highest weights $\Lambda-\gamma$ with $\gamma\in S(\Lambda)$ and they appear with multiplicity $m(\gamma)$ in this filtration. This yields the following decomposition.

\begin{lemma}\label{lemm::decomposition_even_M_otimes_1}
Let $M$ be a simple highest weight $\gg$-supermodule of highest weight $\Lambda$. Assume that $\Lambda$ is the highest weight of a unitarizable highest weight $\even$-module. 

\begin{enumerate}
    \item[a)] There exists a subset $S(\Lambda)\subseteq\Gamma$ such that
\[
\langle M\otimes 1\rangle_{\even}=\bigoplus_{\gamma\in S(\Lambda)}m(\gamma)L_{0}(\Lambda-\rho_{\bar{1}}-\gamma),
\]
where $m(\gamma)\in \ZZ_{\ge 0}$ denotes the multiplicity of the constituent $L_{0}(\Lambda-\rho_{\bar{1}}-\gamma)$.
\item[b)] $L_{0}(\Lambda-\rho_{\bar{1}})$ appears in $\langle M\otimes 1\rangle_{\even}$ as a $\even$-constituent with multiplicity one.
\item[c)] $\langle M\otimes 1\rangle_{\even}$ is a unitarizable $\even$-module.
\item[d)] Considered as a $\even$-submodule in $M\otimes M(\odd)$, $\Dirac^{2}$ acts on each $L_{0}(\Lambda-\rho_{\bar{1}}-\gamma)$ as the scalar multiple of $-(\Lambda+2\rho, \Lambda)+(\Lambda-\gamma+2\rho, \Lambda-\gamma)$.
\end{enumerate}
\end{lemma}

\begin{proof}
By~\cite[Theorem~10.4.5]{musson2012lie}, there exists a subset $S(\Lambda)\subseteq\Gamma$ and a $\even$-filtration
\[
0=M_{0}\subset M_{1}\subset\cdots\subset M_{r}=M
\]
such that each quotient $M_{i+1}/M_{i}$ is isomorphic to $L_{0}(\Lambda-\gamma)$ for some $\gamma\in S(\Lambda)$. Set $N_{i}\coloneqq \langle M_{i}\otimes 1\rangle_{\even}$. Then $(N_{i})$ is a $\even$-filtration of $\langle M\otimes 1\rangle_{\even}$, and each quotient $N_{i+1}/N_{i}$ is generated by the image of $(M_{i+1}/M_{i})\otimes 1$. In particular, the composition factors of $\langle M\otimes 1\rangle_{\even}$ are highest weight $\even$-modules of highest weight $\Lambda-\rho_{\bar1}-\gamma$, with $\gamma\in S(\Lambda)$, occurring with multiplicity $m(\gamma)$.

As in the proof of Corollary~\ref{cor::even_constituents_Dirac_relative_holomorphic}, one has $\Lambda-\rho_{\bar1}-\gamma\in\Dis$ for all $\gamma\in S(\Lambda)$. Hence, by Lemma~\ref{lemm::direct_consequences_Dis}, each such highest weight module is simple, and $\Ext^{1}$ between any two composition factors vanishes. Therefore the filtration splits, and $\langle M\otimes 1\rangle_{\even}$ is the indicated direct sum. This proves~a). Statements~b),~c), and~d) follow immediately.
\end{proof}

\begin{remark}
For $\langle M\otimes 1\rangle_{\even}\subset M\otimes M(\odd)$, one has
\[
\ker\Dirac^{2}\big|_{\langle M\otimes 1\rangle_{\even}}=L_{0}(\Lambda-\rho_{\bar{1}}),\qquad
\Im\Dirac^{2}\big|_{\langle M\otimes 1\rangle_{\even}}=\bigoplus_{\gamma\in S(\Lambda)\setminus\{0\}}L_{0}(\Lambda-\rho_{\bar{1}}-\gamma).
\]
\end{remark}

By Lemma~\ref{lemm::shift_unitarity}, it suffices to show that $\langle M\otimes 1\rangle_{\even}$ is a unitarizable $\even$-module with respect to the restriction of $\langle\cdot,\cdot\rangle_{M\otimes M(\odd)}$. By Lemma~\ref{lemm::decomposition_even_M_otimes_1}, one has
\begin{equation}
\langle M\otimes 1\rangle_{\even}=\bigoplus_{\gamma\in S(\Lambda)}m(\gamma)L_{0}(\Lambda-\rho_{\bar1}-\gamma).
\end{equation}
Equipped with the direct sum of the positive-definite Hermitian forms $\langle\cdot,\cdot\rangle_{L_{0}(\nu)}$ on its simple summands, this becomes a unitarizable $\even$-module. On each simple summand $L_{0}(\nu)$, the restriction of $\langle\cdot,\cdot\rangle_{M\otimes M(\odd)}$ is a nonzero $\even$-invariant Hermitian form. Hence
\begin{equation}\label{eq::idea}
\langle\cdot,\cdot\rangle_{M\otimes M(\odd)}\big|_{L_{0}(\nu)}=c_{\nu}\langle\cdot,\cdot\rangle_{L_{0}(\nu)}
\end{equation}
for some $c_{\nu}\in\RR\setminus\{0\}$. It remains to show that $c_{\nu}>0$ for every $\nu=\Lambda-\rho_{\bar1}-\gamma$ with $\gamma\in S(\Lambda)$. Since $\langle\cdot,\cdot\rangle_{M\otimes M(\odd)}\big|_{L_{0}(\nu)}$ is therefore either positive definite or negative definite, it suffices to evaluate it on a highest weight vector $v_{\nu}$ of $L_{0}(\nu)$. If $\langle v_{\nu},v_{\nu}\rangle_{M\otimes M(\odd)}>0$ for all $\nu$, then $c_{\nu}>0$ for all $\nu$, and the claim follows. For these highest weight vectors, we need the following lemma. 

\begin{lemma} \label{lemm::induction_HW_Dirac_inequality}
    Let $v_{\mu} \otimes 1$ be a highest weight vector in the $\even$-decomposition of $\langle M\otimes M(\odd)\rangle_{\even}$ in Lemma~\ref{lemm::decomposition_even_M_otimes_1}. Then 
    $$
    ((\mu+2\rho,\mu)-(\Lambda+2\rho,\Lambda))\langle v_{\mu}\otimes 1, v_{\mu}\otimes 1\rangle_{M\otimes M(\odd)} = 4\sum_{k=1}^{mn}\langle \partial_{k}v_{\mu}\otimes 1,\partial_{k}v_{\mu}\otimes 1\rangle_{M\otimes M(\odd)}.
    $$
\end{lemma}

\begin{proof} By Lemma~\ref{lemm::decomposition_even_M_otimes_1}, $D^{2}$ acts on each $v_{\mu}\otimes 1$ as the scalar multiple of $-(\Lambda+2\rho, \Lambda)+(\mu+2\rho,\mu)$. Consequently, 
\begin{equation*}
\begin{aligned}
 ((\mu+2\rho,\mu)-(\Lambda+2\rho,\Lambda))\langle v_{\mu}\otimes 1,v_{\mu}\otimes 1\rangle_{M\otimes M(\odd)} &= \langle \Dirac^{2} (v_{\mu} \otimes 1), v_{\mu} \otimes 1 \rangle_{M \otimes M(\odd)} \\
 &= -4 \sum_{k=1}^{mn} \langle x_{k} \partial_{k} v_{\mu} \otimes 1, v_{\mu} \otimes 1 \rangle_{M \otimes M(\odd)} \\
 &= 4 \sum_{k=1}^{mn} \langle \partial_{k} v_{\mu} \otimes 1, \partial_{k} v_{\mu} \otimes 1 \rangle_{M \otimes M(\odd)} \\ &= 4\sum_{k=1}^{mn}\langle \partial_{k}(v_{\mu}\otimes 1),\partial_{k}(v_{\mu}\otimes 1)\rangle_{M\otimes M(\odd)} \end{aligned},
\end{equation*}
where we use in the second equality the explicit form of $\Dirac^{2}$ given in Remark~\ref{rmk::Explicit_form_D_square}, 
and 
$$
\langle 2\sum_{k,l}[\partial_{k},\partial_{l}]v \otimes x_{k}x_{l}, v\otimes 1\rangle_{M\otimes M(\odd)} = 2\sum_{k,l=1}^{mn}\langle[\partial_{k},\partial_{l}]v,v\rangle_{M}\langle x_{k}x_{l},1\rangle_{M(\odd)} = 0,
$$
since $\langle x_{k}x_{l},1 \rangle_{M\odd)} = \langle x_{l}, \partial_{k}1\rangle_{M(\odd)} = 0$ for all $1 \leq k,l \leq mn$. The third equality uses the contravariance of $\langle \cdot, \cdot \rangle_{M \otimes M(\odd)}$ and the fact that $\omega(x_{k}) = -\partial_{k}$ for all $1 \leq k \leq mn$. The last equality uses
\begin{equation*}
\partial_{k}(v_{\Lambda-\gamma}\otimes 1)=\partial_{k}v_{\Lambda-\gamma}\otimes 1\in\langle M\otimes 1\rangle_{\even}. \qedhere
\end{equation*}
\end{proof}

\begin{theorem}\label{UnitarityD2} Let $M$ be a highest weight $\gg$-supermodule with highest weight $\Lambda$. 
Then $M$ is unitarizable if and only if the following two conditions are satisfied:
    \begin{enumerate}
        \item[a)] $\Lambda$ is the highest weight of a unitarizable highest weight $\even$-module.
        \item[b)] If $L_{0}(\mu)$ is a simple $\even$-composition factor in a $\even$-filtration of $M$ with highest weight $\mu$, then
\begin{equation*}
    \begin{aligned}
        (\mu + 2\rho, \mu) > (\Lambda+2\rho, \Lambda).
    \end{aligned}
\end{equation*}
    \end{enumerate}
\end{theorem}

\begin{proof} If $M$ is unitarizable, then a) and b) follow directly from Proposition~\ref{CompletelyReducible} and the Dirac inequality in Proposition~\ref{prop::specified_Dirac_inequality}.

Conversely, following the idea of \eqref{eq::idea}, we show that $c_{\nu}>0$ for every $\nu=\Lambda-\rho_{\bar1}-\gamma$ with $\gamma\in S(\Lambda)$ by the strict Dirac inequality and an induction argument. 

We first endow $\langle M\otimes 1\rangle_{\even}$ with a $\ZZ$-grading compatible with the action of the $\partial_{k}$ on the highest weight vectors $v_{\Lambda-\gamma}\otimes 1$ for $\gamma \in S(\Lambda)$. By definition, each $\gamma\in\Gamma$ is a sum of distinct positive odd roots. By our choice of positive system, the number of summands is uniquely determined by $\gamma$. We call it the \emph{length} of $\gamma$ and denote it by $\operatorname{len}(\gamma)$. This yields a well-defined grading of the $\gg_{\bar0}$-constituents $L_{0}(\mu)$ by $\operatorname{len}(\gamma)$. Explicitly,
\[
\langle M\otimes 1\rangle_{\even}=\bigoplus_{k}\langle M\otimes 1\rangle_{\even}[k],\qquad
\langle M\otimes 1\rangle_{\even}[k]\coloneqq\bigoplus_{\substack{\gamma\in S(\Lambda)\\ \operatorname{len}(\gamma)=k}}m(\gamma)L_{0}(\Lambda-\rho_{\bar{1}}-\gamma).
\]
By construction, $\langle M\otimes 1\rangle_{\even}[0]=L_{0}(\Lambda-\rho_{\bar{1}})$, and the restriction of the Shapovalov form to this summand is positive definite. 

We next show that this grading is compatible with the action of the operators $\partial_k$ on the vectors $v_{\Lambda-\gamma}\otimes 1$. Since $\partial_k(v_{\Lambda-\gamma}\otimes 1)\in \langle M\otimes 1\rangle_{\even},
$
and $\langle M\otimes 1\rangle_{\even}$ is the direct sum of the modules $L_0(\Lambda-\rho_{\bar1}-\gamma')$, every nonzero vector in $\langle M\otimes 1\rangle_{\even}$ must have a weight occurring in one of these summands. Hence such a vector is zero as soon as its weight does not occur in any of them. Recall that every weight of a $\even$-constituent $L_0(\nu)$ is of the form
\begin{equation}\label{eq::form_weights}
\nu-\sum_j s_j\beta_j,\qquad s_j\in\ZZ_+,\ \beta_j\in\Delta_{\bar0}^+.
\end{equation}

Let $v_{\Lambda-\gamma}\otimes 1$ be a highest weight vector in the constituent indexed by $\gamma=\gamma_{1}+\cdots+\gamma_{l}$, where the $\gamma_{i}$ are distinct positive odd roots. We fix such an arbitrary decomposition of $\gamma$; the same argument applies verbatim to any other one. Then $v_{\Lambda-\gamma}\otimes 1 \in \langle M\otimes 1\rangle_{\even}[l]$. Let $\partial_{k}$ be an odd weight vector of positive odd root $\alpha$. We consider $\partial_{k}(v_{\Lambda-\gamma}\otimes 1)$.

The idea is to compare $\alpha_k$ with each summand $\gamma_i$ of $\gamma$. Note that $\partial_{k}(v_{\Lambda-\gamma}\otimes 1)$ is either trivial or non-trivial and has weight of the form
$
\Lambda-\rho_{\bar1}-\gamma+\alpha_{k}.
$ Write
\[
A\coloneqq\{\epsilon_{i}-\delta_{j}:1\leq i\leq p,\ 1\leq j\leq n\},\qquad B\coloneqq\{-\epsilon_{i}+\delta_{j}:p+1\leq i\leq m,\ 1\leq j\leq n\}.
\]
For $\alpha\in A$ and $\beta\in B$, neither $\alpha-\beta$ nor $\beta-\alpha$ is a sum of roots. If $\alpha,\alpha'\in A$ or $\beta,\beta'\in B$, then $\alpha-\alpha'$ and $\beta-\beta'$ are either $0$, not roots, or even roots of either sign. Hence, for each $i$,
\[
\alpha_{k}-\gamma_{i}=0,\qquad \alpha_{k}-\gamma_{i}\notin\Delta,\qquad\text{or}\qquad \alpha_{k}-\gamma_{i}\in\Delta_{\bar0}^{\pm}.
\]
If $\alpha_{k}-\gamma_{i}=0$ or $\alpha_{k}-\gamma_{i}\in-\Delta_{\bar0}^{+}$, then the resulting weight can be written as $\Lambda-\rho_{\bar1}-(\gamma-\alpha_{k})$ in the first case and as $\Lambda-\rho_{\bar1}-(\gamma-\gamma_{i})-(\gamma_{i}-\alpha_{k})$ in the second. Moreover, $\operatorname{len}(\gamma-\alpha_{k})=l-1$ and $\operatorname{len}(\gamma-\gamma_{i})=l-1$.
Consequently, if $\partial_{k}(v_{\Lambda-\gamma}\otimes 1)$ is non-trivial its weight can occur only in a constituent $L_{0}(\Lambda-\rho_{\bar1}-\gamma')$ with $\operatorname{len}(\gamma')\le l-1$. Indeed, its weight then belongs to the constituent
$L_0\bigl(\Lambda-\rho_{\bar1}-(\gamma-\gamma_i)\bigr)$ or $L_{0}(\Lambda-\rho_{\bar{1}}-(\gamma-\alpha_{k}))$,
which lie in layer $l-1$. 	If no such $i$ exists, then this weight occurs in no even constituent, so the vector is zero. Therefore
\[
\partial_{k}(v_{\Lambda-\gamma}\otimes 1)\in\bigoplus_{r\le l-1}\langle M\otimes 1\rangle_{\even}[r].
\]

We are now ready to complete the proof. By assumption, the restriction of the Shapovalov form to $L_{0}(\Lambda-\rho_{\bar{1}})$ is positive definite.  
On each $\langle M \otimes 1\rangle_{\even}[l]$, we define the restriction of the Shapovalov form to be $\langle\cdot,\cdot\rangle_{\langle M \otimes 1\rangle_{\even}[l]}$. Note that $\langle M \otimes 1\rangle_{\even}[0]\cong L_{0}(\Lambda-\rho_{\bar{1}})$. It suffices to prove that $\langle v,v\rangle_{M\otimes M(\odd)}>0$ for any $\even$-highest weight module $v$ appearing in layer $l$, that is, $v=v_{\Lambda-\gamma}\otimes 1$ for some $\gamma \in S(\Lambda)$ with $\operatorname{len}(\gamma)=l$. We argue by induction on $l$.

The case $l=0$ holds by assumption.  
Suppose that the restrictions of the Shapovalov form to $\langle M \otimes 1\rangle_{\even}[0],\langle M \otimes 1\rangle_{\even}[1],\ldots,\langle M \otimes 1\rangle_{\even}[l]$ are positive definite.  
Let $v\in \langle M \otimes 1\rangle_{\even}[l+1]\setminus\{0\}$ be a highest weight vector of a $\even$-constituent in layer $l+1$. Assume $v$ is of the form $v_{\Lambda-\gamma}\otimes 1$. For every $k=1,\ldots,mn$ we have $\partial_{k}v=0$ or $\partial_{k}v\in \langle M \otimes 1\rangle_{\even}[l]$, hence
\[
\langle \partial_{k}v,\partial_{k}v\rangle_{\langle M \otimes 1\rangle_{\even}[l]}\geq 0,
\]
with equality if and only if $\partial_{k}v=0$. At least one $\partial_{k}v$ must be nonzero; otherwise $\mathfrak{n}^{+}_{\bar{1}} v = 0$ and $v$ is $\bb$-primitive which contradicts the simplicity of $M$. Consequently, using Lemma~\ref{lemm::induction_HW_Dirac_inequality}, one has
\[
\bigl((\mu+2\rho,\mu) - (\Lambda+2\rho,\Lambda)\bigr)
\langle v_{\Lambda-\gamma}\otimes 1,v_{\Lambda-\gamma}\otimes 1\rangle_{\langle M \otimes 1\rangle_{\even}[l+1]}
=
4\sum_{k=1}^{mn}
\langle \partial_{k}v_{\Lambda-\gamma}\otimes 1, \partial_{k}v_{\Lambda-\gamma}\otimes 1\rangle_{\langle M \otimes 1\rangle_{\even}[l]}
>0.
\]
Since the Dirac inequalities hold strictly, the coefficient on the left is positive, and therefore
$
\langle v_{\Lambda-\gamma}\otimes 1,v_{\Lambda-\gamma}\otimes 1\rangle_{\langle M \otimes 1\rangle_{\even}[l+1]} > 0.$
This completes the induction and the proof.
\end{proof}

\subsection{Relation to Kostant's cohomology} 
\label{subsec::g_+1 cohomology} We define an analog of Kostant's cohomology similarly to \cite{cheng2004analogue}, which captures the $\kk^{\CC}$-module structure of the Dirac cohomology. When we consider a unitarizable $\gg$-supermodule as a $\kk^{\CC}$-module, we neglect parity. We adapt the notation of Section~\ref{subsec::Dirac_operator}.

\subsubsection{Construction of \texorpdfstring{$\mathrm{H}^{\ast}(\gg_{+1},M)$}\ }
The Lie superalgebra $\gg$ of type $A$ admits a $\ZZ_{2}$-compatible $\ZZ$-grading. We consider $\odd$ as a $\even$-module under the supercommutator. Then $\odd$ decomposes into a direct sum of two simple $\even$-modules, namely $\odd = \gg_{-1} \oplus \gg_{+1}$, where concretely
\begin{equation}
\gg_{+1}\coloneqq \pp_{1} \oplus \pp_{2}, \qquad \gg_{-1}\coloneqq \qq_{1} \oplus \qq_{2}.
\end{equation}
Both are, in particular, abelian Lie subsuperalgebras of $\gg$, \emph{i.e.}, $[\gg_{\pm 1}, \gg_{\pm 1}] = 0$, and the associated $\ZZ$-grading of $\gg$, given by
\begin{equation}
\gg = \gg_{-1} \oplus \even \oplus \gg_{+1},
\end{equation}
is then compatible with the $\ZZ_{2}$-grading on $\gg$.

Fix the complex vector spaces $\gg_{\pm 1}$ and consider the $\ZZ$-gradings $S(\gg_{\pm 1}) = \bigoplus_{n=0}^{\infty}S^{n}(\gg_{\pm 1})$ and $\bigwedge(\gg_{\pm 1}) = \bigoplus_{n=0}^{\infty}\bigwedge^{n}(\gg_{\pm 1})$. Note that $I_{\wedge,S}\cap T^{1}(\gg_{\pm 1}) = \{0\}$, so that we can identify $\gg_{\pm 1}$ with $\bigwedge^{1}(\gg_{\pm 1})$ and $S^{1}(\gg_{\pm 1})$.

 We identify the dual space $\gg_{+1}^{\ast}$ with $\gg_{-1}$ using the bilinear form $2B(\cdot, \cdot)$. This identification is $\even$-invariant, since $B(\cdot, \cdot)$ is invariant. Additionally, we equate $S(\gg_{+1}^{\ast})$ with $S(\gg_{-1})$, the polynomial algebra in the variables $x_{1}, \dotsc, x_{pn}, \partial_{pn+1}, \dotsc, \partial_{mn}$. The universal enveloping algebra $\UE(\gg_{+1})$ can be identified with $\bigwedge(\gg_{+1})$.

To define the cohomology under consideration, we use the standard free resolution of
$\bigwedge(\gg_{+1})$–modules. Recall that \begin{equation}
\begin{aligned}
\mathrm{d}^{\mathfrak{p}_{1}}
 &= \sum_{k=1}^{pn} \partial_{k} \otimes x_{k},
&\qquad
\delta^{\mathfrak{p}_{1}}
 &= \sum_{k=1}^{pn} x_{k} \otimes \partial_{k}, \\[4pt]
\mathrm{d}^{\mathfrak{q}_{2}}
 &= \sum_{k=pn+1}^{mn} \partial_{k} \otimes x_{k},
&\qquad
\delta^{\mathfrak{q}_{2}}
 &= \sum_{k=pn+1}^{mn} x_{k} \otimes \partial_{k}.
\end{aligned}
\end{equation} 
Set $\delta\coloneqq -\mathrm{d}^{\mathfrak{q}_{2}} + \delta^{\mathfrak{p}_{1}},
$
and consider the complex
\begin{equation}
\dotsb
\xlongrightarrow{\delta}
S^{i}(\gg_{+1}) \otimes \bigwedge(\gg_{+1})
\xlongrightarrow{\delta}
S^{i-1}(\gg_{+1}) \otimes \bigwedge(\gg_{+1})
\xlongrightarrow{\delta}
\dotsb
\xlongrightarrow{\delta}
\CC \otimes \bigwedge(\gg_{+1})
\xlongrightarrow{\delta} 0.
\end{equation}
Here, we use that the symmetric algebra $S(\gg_{+1})$ is the polynomial algebra generated by
$
\partial_{1},\ldots,\partial_{pn}, x_{pn+1},\ldots,x_{mn},
$
so that $\delta$ indeed lowers the polynomial degree, and by Lemma~\ref{Orthogonality}, we have $\delta^{2} = 0$. Thus $\delta$ defines a differential. Moreover, it is $\kk^{\CC}$-invariant by Lemma~\ref{lemm::k-invariance}.

Fix a $(\gg,\kk^{\CC})$-supermodule $M$, and apply the contravariant functor $\Hom_{\bigwedge(\gg_{+1})}(-,M)$
to the above resolution. Then we arrive at the following complex of vector spaces
\begin{equation} \label{complex_of_k_modules}
\dotsc\xlongleftarrow{\mathrm{d}}\Hom_{\CC}(S^{i+1}(\gg_{+1}),M) \xlongleftarrow{\mathrm{d}}\Hom_{\CC}(S^{i}(\gg_{+1}),M)\xlongleftarrow{\mathrm{d}}\dotsc\xlongleftarrow{\mathrm{d}}\Hom_{\CC}(S^{0}(\gg_{+1}),M) \leftarrow 0,
\end{equation}
where $\mathrm{d}$ is the pullback operator of $\delta$, explicitly given by $\mathrm{d}\coloneqq \ddp - \delq$ if we identify $\Hom_{\CC}(S^{i}(\gg_{+1}),M)$ with $M\otimes S^{i}(\gg_{-1})$ for any $i$. 

$M$ and $S^{i}(\gg_{+1})$ are $\kk^{\CC}$-modules, where the action on $S^{i}(\gg_{+1})$ is induced by the adjoint action of $\kk^{\CC}$ on $\odd$. This identifies the spaces \begin{equation}\mathrm{C}^{i}(M)\coloneqq \Hom_{\CC}(S^{i}(\gg_{+1}), M) \cong M \otimes S^{i}(\gg_{-1})\end{equation} as $\kk^{\CC}$-modules. The $\kk^{\CC}$-invariance of $\mathrm{d}$ (\emph{cf.}~Lemma~\ref{lemm::k-invariance}) makes the complex $\mathrm{C}\coloneqq (\mathrm{C}^{i}(M), \mathrm{d})$ into a $\kk^{\CC}$-module complex. The cohomology groups of the complex $\mathrm{C}$ will be denoted by $\mathrm{H}^{i}(\gg_{+1}, M)$. These are naturally $\kk^{\CC}$-modules.

\subsubsection{Relation to Dirac cohomology}
For a unitarizable $\gg$-supermodule $\HH$, we regard $\mathrm d$ and $\delta$ as operators on $\HH\otimes M(\odd)$, so that $\Dirac=2(\mathrm d-\delta)$. The Dirac cohomology of $\HH$ is
$
\DC(\HH)=\Ker\Dirac,
$
whereas the analogue of Kostant cohomology is $\Ker\mathrm d/\Im\mathrm d$. We compare these two spaces. We begin by recording the basic properties of $\mathrm d$ and $\delta$.

\begin{lemma} \label{lemm::relations_d_delta}
 Let $\HH$ be a unitarizable simple $\gg$-supermodule. Then the following assertions hold with respect to $\langle \cdot,\cdot \rangle_{\HH \otimes M(\odd)}$:
 \begin{enumerate}
 \item[a)] $\mathrm{d}$ and $\delta$ are skew-adjoint to each other.
 \item[b)] $\Im \mathrm{d}$ is orthogonal to $\Ker \delta$ and $\Im \delta$, while $\Im \delta$ is orthogonal to $\ker \mathrm{d}$.
 \end{enumerate}
\end{lemma}
\begin{proof}
 a) The operators are defined as $\mathrm{d} = \ddp - \delq$ and $\delta = -\ddq + \delp$. By Lemma~\ref{Orthogonality}, $\ddp$ is skew-adjoint to $\delp$ and $\ddq$ is skew-adjoint to $\delq$. This proves a).
 
 b) This proof is analogous to the proofs of parts c) and d) in Lemma~\ref{Orthogonality} and will be omitted.
\end{proof}

\begin{lemma} \label{lemm::kerD,kerd}
 Let $\HH$ be a unitarizable simple $\gg$-supermodule. Then the following assertions hold with respect to $\langle \cdot,\cdot \rangle_{\HH \otimes M(\odd)}$:
 \begin{enumerate}
 \item[a)] $\HH\otimes M(\odd) = \Ker \Dirac \oplus \Im \mathrm{d} \oplus \Im \delta$.
 \item[b)] $\ker \mathrm{d} = \Ker \Dirac \oplus \Im \mathrm{d}$.
 \end{enumerate}
\end{lemma}
\begin{proof}
 a) By Lemma~\ref{lemm::Hodge_decomposition} and Corollary~\ref{iterativeKernel}, we have $\HH \otimes M(\odd) = \Ker \Dirac^{2} \oplus \Im \Dirac^{2} = \Ker \Dirac \oplus \Im \Dirac^{2}$. Moreover, $\Im \mathrm{d}$ and $\Im \delta$ are orthogonal to each other by Lemma~\ref{lemm::relations_d_delta}, $\mathrm{d}^{2} = 0$, and $\delta^{2} = 0$. We conclude:
 $$
 \HH \otimes M(\odd) = \Ker \Dirac \oplus \Im \Dirac^{2} \subset \Ker \Dirac \oplus \Im \Dirac \subset \Ker \Dirac \oplus \Im \mathrm{d} \oplus \Im \delta,
 $$
 and consequently $\HH \otimes M(\odd) = \Ker \Dirac \oplus \Im \mathrm{d} \oplus \Im \delta$.

 b) The assertion follows from a) and Lemma~\ref{lemm::relations_d_delta} once we show $\Ker \Dirac = \ker \mathrm{d} \cap \Ker \delta$. However, this assertion is clear since $\Dirac = 2(\mathrm{d}-\delta)$ and $\Im \mathrm{d}$ and $\Im \delta$ are orthogonal.
\end{proof}

Now Lemma~\ref{lemm::kerD,kerd} and Proposition~\ref{EvenRestriction} allow us to compare the $\kk^{\CC}$-module structures of $\DC(\HH)$ and $\mathrm H^\ast(\gg_{+1},\HH)$. For this, we need the following elementary observation. It follows from the fact that the $B$-dual bases induce a $\kk^{\CC}$-equivariant nondegenerate pairing between the underlying spaces, and hence between the corresponding symmetric algebras.

\begin{lemma}\label{lemm::kk_tensor}
The following assertions hold:
\begin{enumerate}
\item[a)] One has
\[
\CC[x_{1},\dotsc,x_{mn}]\cong \CC[x_{1},\dotsc,x_{pn}]\otimes \CC[x_{pn+1},\dotsc,x_{mn}]
\]
as $\kk^{\CC}$-modules, where the action is induced by the commutator action of $\kk^{\CC}$ on $\odd$.
\item[b)] The $\kk^{\CC}$-modules $\CC[x_{1},\dotsc,x_{pn}]$ and $\CC[\partial_{1},\dotsc,\partial_{pn}]$, as well as $\CC[x_{pn+1},\dotsc,x_{mn}]$ and $\CC[\partial_{pn+1},\dotsc,\partial_{mn}]$, are dual to each other. In particular, they are isomorphic.
\end{enumerate}
\end{lemma}

Combining these results, we conclude that for unitarizable simple $\gg$-supermodules, Dirac cohomology $\DC(\HH)$ and Kostant cohomology $\mathrm H^\ast(\gg_{+1},\HH)$ are isomorphic as $\kk^{\CC}$-modules, up to a twist.

\begin{theorem}\label{Injection}
For any unitarizable simple $\gg$-supermodule $\HH$ there exists a $\kk^{\CC}$-module isomorphism
$$
\DC(\HH) \cong \mathrm{H}^{\ast}(\gg_{+1},\HH) \otimes \CC_{-\rho_{\bar{1}}}.
$$
\end{theorem}
\begin{proof}
 First, by Proposition~\ref{DiracGleichKernel} and Lemma~\ref{lemm::kerD,kerd}, one has $\DC(\HH) = \Ker \Dirac \cong \ker \mathrm{d}/ \Im \mathrm{d}$.

 Second, by Proposition~\ref{EvenRestriction} and Lemma~\ref{lemm::kk_tensor}, one has the following isomorphisms of $\kk^{\CC}$-modules
\begin{equation*}
\begin{split}
M(\odd)&\cong \CC[x_{1},\dotsc,x_{mn}]\otimes \CC_{-\rho_{\bar{1}}} \\ &\cong (\CC[x_{1},\dotsc,x_{pn}] \otimes \CC[x_{pn+1},\dotsc,x_{mn}]) \otimes \CC_{-\rho_{\bar{1}}} \\ &\cong(\CC[x_{1}, \dotsc,x_{pn}] \otimes \CC[\partial_{pn+1},\dotsc,\partial_{mn}]) \otimes \CC_{-\rho_{\bar{1}}} \\ &\cong \CC[x_{1},\dotsc,x_{pn},\partial_{pn+1},\dotsc,\partial_{mn}] \otimes \CC_{-\rho_{\bar{1}}}.
\end{split}
\end{equation*}
By construction of $\mathrm{H}^{\ast}(\gg_{+1},\HH)$, the statement follows.
\end{proof}

\subsection{Dirac index} \label{subsec::Dirac_index} In this section, we introduce the Dirac index. For this, we call a $\gg$-supermodule $M$ $\Dirac^{2}$\emph{-semisimple} if $\Dirac^{2}$ acts semisimply on $M\otimes M(\odd)$. To every such supermodule we associate a virtual character, called the \emph{Dirac index}. This class includes, in particular, simple and unitarizable supermodules. This index coincides with the Euler characteristic of Dirac cohomology and satisfies several desirable properties. Moreover, for unitarizable simple $\gg$-supermodules, the Dirac index agrees with the Dirac cohomology.

For any $\gg$-supermodule $M$, the Dirac operator $\Dirac$ acts on $M \otimes M(\odd)_{\bar{0}}$ and $M \otimes M(\odd)_{\bar{1}}$, interchanging these subspaces. The \emph{Dirac index} of $M$ is defined as the virtual $\even$-supermodule
\begin{equation}
\DI(M)\coloneqq M \otimes M(\odd)_{\bar{0}} - M \otimes M(\odd)_{\bar{1}}.
\end{equation}
This is an element of the Grothendieck group of $\even$-supermodules. In contrast, the operator $\Dirac: M \otimes M(\odd)_{\bar{0},\bar{1}} \to M \otimes M(\odd)_{\bar{1},\bar{0}}$ gives rise to a decomposition of the Dirac cohomology $\DC(M)$ into even and odd parts:
\[
\DC(M) = \DC^{+}(M) \oplus \DC^{-}(M),
\]
whose \emph{Euler characteristic} is given by the virtual $\even$-supermodule
\[
\DC^{+}(M) - \DC^{-}(M).
\]
The virtual $\even$-supermodules $\DI(M)$ and the Euler characteristic are identical provided that $M$ is $\Dirac^{2}$-semisimple.

\begin{proposition}\label{prop::Euler_char}
Let $M$ be a $\Dirac^{2}$-semisimple $\gg$-supermodule. Then the Dirac index $\DI(M)$ is equal to the Euler characteristic of the Dirac cohomology $\DC(M)$, \emph{i.e.},
\[
\DI(M) = \DC^{+}(M) - \DC^{-}(M).
\]
\end{proposition}

\begin{proof}
We decompose $M \otimes M(\odd)$ into a direct sum of eigenspaces of $\Dirac^2$:
\[
M \otimes M(\odd) = \big(M \otimes M(\odd)\big)(0) \oplus \bigoplus_{c \neq 0} \big(M \otimes M(\odd)\big)(c).
\]
Since $\Dirac^2$ is even, this decomposition is compatible with the $\ZZ_2$-grading:
\[
\big(M \otimes M(\odd)\big)(c) = \big(M \otimes M(\odd)_{\bar{0}}\big)(c) \oplus \big(M \otimes M(\odd)_{\bar{1}}\big)(c).
\]
Moreover, $\Dirac$ commutes with $\Dirac^2$, so it preserves each eigenspace $\big(M \otimes M(\odd)\big)(c)$. However, $\Dirac$ switches parity, inducing isomorphisms
\[
\Dirac(c): \big(M \otimes M(\odd)_{\bar{0}, \bar{1}}\big)(c) \to \big(M \otimes M(\odd)_{\bar{1}, \bar{0}}\big)(c),
\]
with inverses given by $\frac{1}{c} \Dirac(c)$ for $c \neq 0$. Consequently, the contributions from nonzero eigenspaces cancel in the index, and we have
\[
M \otimes M(\odd)_{\bar{0}} - M \otimes M(\odd)_{\bar{1}} = \big(M \otimes M(\odd)_{\bar{0}}\big)(0) - \big(M \otimes M(\odd)_{\bar{1}}\big)(0).
\]
The Dirac operator $\Dirac$ restricts to a differential on $\Ker(\Dirac^2)$, and its cohomology is precisely the Dirac cohomology. The result then follows from the Euler–Poincaré principle.
\end{proof}

As a direct consequence, for a unitarizable simple $\gg$-supermodule, the Dirac index coincides with its Dirac cohomology (see Corollary~\ref{DiracOfSimple}).

\begin{corollary}\label{prop:Coh=Index}
Let $\HH$ be a unitarizable simple $\gg$-supermodule. Then as $\even$-supermodules
\[
\DI(\HH) \cong \DC(\HH).
\]
\end{corollary}

Furthermore, the Dirac index commutes with tensoring by finite-dimensional $\gg$-supermodules. This compatibility can be used to study the Dirac cohomology of unitarizable supermodules via translation functors.

\begin{lemma}
Let $M$ be a $\gg$-supermodule, and let $F$ be a finite-dimensional $\gg$-supermodule. Then there is a canonical isomorphism of $\even$-supermodules
\[
\DI(M \otimes F) \cong \DI(M) \otimes F.
\]
\end{lemma}

\begin{proof}
We compute:
\[
\DI(M \otimes F) = M \otimes F \otimes M(\odd)_{\bar{0}} - M \otimes F \otimes M(\odd)_{\bar{1}},
\]
while
\[
\DI(M) \otimes F = \big(M \otimes M(\odd)_{\bar{0}} - M \otimes M(\odd)_{\bar{1}}\big) \otimes F.
\]
Since $F$ is finite-dimensional, the tensor product distributes over the direct sum and difference, yielding a canonical isomorphism between the two expressions.
\end{proof}

\subsection{Formal characters} \label{subsec::formal_characters}
We now derive two expressions for the formal character of a unitarizable simple $\gg$-supermodule, one in terms of Kostant cohomology and one in terms of the Dirac index.

Let $\HH$ be a unitarizable simple $\gg$-supermodule of highest weight $\Lambda$. Since $\HH$ is a highest weight Harish-Chandra supermodule, all $\hh$-weight spaces $\HH^\mu$ are finite-dimensional. Its formal character is therefore defined by
\begin{equation}
\fch(\HH)\coloneqq \sum_{\mu\in\hh^\ast}\dim \HH^\mu\,e^\mu,
\end{equation}
where $e^\mu$ denotes the formal exponential of $\mu$. On the other hand, $\HH$ decomposes into $\kk^{\CC}$-types as
\begin{equation}
\HH=\bigoplus_{\lambda \in P^{++}_{\kk^{\CC}}}m(\lambda)F^\lambda,
\end{equation}
where $F^\lambda$ is the simple finite-dimensional $\kk^{\CC}$-module of highest weight $\lambda$ and $m(\lambda)<\infty$. Here, $P^{++}_{\kk^{\CC}}$ denotes the set of dominant integral weights in $\hh^{\ast}$ with respect to $\Delta_{c}^{+}$. Since $\hh\subseteq\kk^{\CC}$, each $F^\lambda$ has a well-defined $\hh$-character $\ch(F^\lambda)$. The $\kk^{\CC}$-type decomposition
therefore yields the expression
\begin{equation}
\fch_{\kk^{\CC}}(\HH)\coloneqq \sum_{\lambda \in P^{++}_{\kk^{\CC}}} m(\lambda)\,\fch(F^{\lambda}),
\end{equation}
which, upon restriction to $\hh\subset \kk^{\CC}$, is precisely the formal character $\fch(\HH)$. Hence the formal $\hh$-character of $\HH$ is determined by its $\kk^{\CC}$-type decomposition, and we use the notation $\ch(\cdot)$ to emphasize the $\kk^{\CC}$-type decomposition.

\subsubsection{Formal characters and Kostant's cohomology} We study the relation of $\ch(\HH)$ and $\mathrm{H}^{\ast}(\gg_{+1},\HH)$ partially using ideas and constructions presented in \cite{cheng2004analogue}.

Fix a unitarizable simple $\gg$-supermodule $\HH$. In Section~\ref{subsec::g_+1 cohomology}, we introduced the cohomology groups $\mathrm{H}^{k}(\gg_{+1},\HH)$ as the cohomology groups of the $\kk^{\CC}$-complex $\mathrm{C}\coloneqq (\mathrm{C}^{k}(\HH), \mathrm{d})$ with $\mathrm{C}^{k}(\HH)\coloneqq \Hom_{\CC}(S^{k}(\gg_{+1}),\HH)\cong \HH \otimes S^{k}(\gg_{-1})$. Recall that we treat here $\gg_{\pm 1}$ as ordinary complex vector spaces. Moreover, let $\mathrm{H}^{k}(\gg_{+1},\HH)^{\lambda}$ be the weight $\lambda$-subspace of $\mathrm{H}^{k}(\gg_{+1},\HH)$ for some weight $\lambda \in \hh^{\ast}$. As $\HH$ is a Harish-Chandra supermodule (Proposition~\ref{prop::unitarizable_are_HC_supermodules}), we have for fixed weight $\lambda$ that $\dim(\mathrm{H}^{k}(\gg_{+1},\HH)^{\lambda})\neq 0$ only for finitely many $k$. The Euler--Poincaré principle then implies that 
\begin{equation}
\sum_{k=0}^{\infty}(-1)^{k}\dim( \mathrm{H}^{k}(\gg_{+1},\HH)^{\lambda})=\sum_{k=0}^{\infty}(-1)^{k}\dim(\mathrm{C}^{k}(\HH)^{\lambda}),
\end{equation}
and considering their formal characters gives
\begin{equation}
\sum_{k=0}^{\infty}(-1)^{k} \fch\left( \mathrm{H}^{k}(\gg_{+1},\HH)\right)=\sum_{k=0}^{\infty}(-1)^{k}\fch(\mathrm{C}^{k}(\HH)) = \fch(\HH)\sum_{k=0}^{\infty}(-1)^{k}\fch(S^{k}(\gg_{+1}^{\ast})).
\end{equation}

Next, as $\kk^{\CC}$-modules, we have $S^{k}(\gg_{-1}) \cong S^{k}(\qq_{1}\oplus \pp_{2}) = S^{k}(\nn_{\bar{1}}^{-})$ for any $k$ (\emph{cf.} Lemma~\ref{lemm::kk_tensor}), and it is well-known that 
$
\sum_{k=0}^{\infty}(-1)^{k}\fch(S^{k}(\nn_{\bar{1}}^{-}))=(\prod_{\gamma\in \Delta_{\bar{1}}^{+}}(1+e^{-\gamma}))^{-1}.
$
We conclude
\begin{equation}
\sum_{k=0}^{\infty}(-1)^{k}\fch(\mathrm{C}^{k}(\HH))=\frac{1}{D_{1}}\fch(\HH), \qquad D_{1} = \prod_{\gamma\in \Delta_{\bar{1}}^{+}}(1+e^{-\gamma}),
\end{equation}
and the formal character of $\HH$ is given by
\begin{equation}
\fch(\HH)=D_{1}\sum_{k=0}^{\infty}(-1)^{k}\fch(\mathrm{H}^{k}(\gg_{+1},\HH)).
\end{equation}

We now pass to the $\kk^{\CC}$-type decomposition. Let $[\mathrm{H}^{k}(\gg_{+1},\HH):F^{\mu}]$ denote the multiplicity of the simple (unitarizable highest weight) $\kk^{\CC}$-module $F^{\mu}$ with highest weight $\mu$ in $\mathrm{H}^{k}(\gg_{+1},\HH)$. As $\kk$ is compact, $\mathrm{H}^{k}(\gg_{+1},\HH)$ is completely reducible as a $\kk^{\CC}$-module, and we can express $\ch(\mathrm{H}^{k}(\gg_{+1},\HH))$ in terms of the multiplicities: 
\begin{equation}
\ch(\mathrm{H}^{k}(\gg_{+1},\HH))=\sum_{\mu \in P^{++}_{\kk^{\CC}}}[\mathrm{H}^{k}(\gg_{+1},\HH):F^{\mu}]\ch(F^{\mu}).
\end{equation}
Altogether, we have proven the following theorem.
\begin{theorem}\label{formalCharacter}
 Let $\HH$ be a unitarizable simple $\gg$-supermodule. The formal character of $\HH$ is
 $$
\ch(\HH)=D_{1}\sum_{\mu\in P^{++}_{\kk^{\CC}}}\sum_{k=0}^{\infty}(-1)^{k}[\mathrm{H}^{k}(\gg_{+1},\HH):F^{\mu}]\ch\bigl( F^{\mu}\bigr).
 $$
\end{theorem}

\subsubsection{Formal characters and Dirac index}
We have shown that, for a $\Dirac^{2}$-semisimple supermodule $M$, the Dirac index agrees with Dirac cohomology. We now derive a $\kk^{\CC}$-character formula for $\Dirac^{2}$-semisimple Harish-Chandra supermodules. In what follows, such a supermodule is denoted by $\HH$. By definition $\HH$ admits a formal $\kk^{\CC}$-character, and we obtain a formula $\ch(\HH)$ from $\DI(\HH)$.

Recall from Section~\ref{subsec::Dirac_index} the definition:
\begin{equation}
\DI(\HH) = \HH \otimes M(\odd)_{\bar{0}} - \HH \otimes M(\odd)_{\bar{1}} = \DC^{+}(\HH) - \DC^{-}(\HH),
\end{equation}
which is the Euler characteristic of $\DC(\HH)$ by Proposition~\ref{prop::Euler_char}. In terms of characters, this reads
\begin{equation}
\ch(\HH)\bigl(\ch(M(\odd)_{\bar{0}}) - \ch(M(\odd)_{\bar{1}})\bigr) = \ch(\DC^{+}(\HH)) - \ch(\DC^{-}(\HH)).
\end{equation}
The arguments below follow the methods and ideas of \cite{branching_huang}, adapted to the present setting.

We consider the finite-dimensional vector space $\nn_{\bar{1}}^{-}$, with basis $\{x_{1},\ldots,x_{mn}\}$. We are interested in the free resolution of free $\bigwedge(\nn_{\bar{1}}^{-})$-modules:
\begin{equation}
\ldots \xrightarrow{\delta} S^{i}(\nn_{\bar{1}}^{-}) \otimes \bigwedge (\nn_{\bar{1}}^{-}) \xrightarrow{\delta} S^{i-1}(\nn_{\bar{1}}^{-}) \otimes \bigwedge (\nn_{\bar{1}}^{-})\xrightarrow{\delta} \ldots \xrightarrow{\delta} \mathbb{C} \otimes \bigwedge (\nn_{\bar{1}}^{-}) \xrightarrow{\delta} 0,
\end{equation}
where the boundary operator is 
\begin{equation}
\delta\coloneqq \ddp + \ddq =\sum_{i = 1}^{mn} \frac{\partial}{\partial x_{i}} \otimes x_{i} : S(\nn_{\bar{1}}^{-}) \otimes_{\mathbb{C}} \bigwedge (\nn_{\bar{1}}^{-}) \longrightarrow S(\nn_{\bar{1}}^{-}) \otimes_{\mathbb{C}} \bigwedge (\nn_{\bar{1}}^{-}).
\end{equation}
The boundary operator $\delta$ is invariant under the action of $\kk^{\CC}$, \emph{i.e.}, $[\kk^{\CC}, \delta] = 0$, as shown in Lemma~\ref{lemm::k-invariance}. Additionally, the proof of the following lemma follows \emph{mutatis mutandis} from \cite[Proposition 3.3.5]{huang2007dirac}.

\begin{lemma}\label{lemm::Koszul_complex}
 The following assertion holds:
$$\Ker \delta = \Im \delta \oplus \CC(1\otimes 1).
 $$
 In particular, the cohomology is generated by $\CC(1\otimes 1)$.
\end{lemma}

\begin{lemma} \label{lemm::char_minus_rho_odd}
 Let $\CC_{-\rho_{\bar{1}}}$ be the one-dimensional $\kk^{\CC}$-module with weight $-\rho_{\bar{1}}$. Then 
 \[
 \ch(\bigwedge\nolimits \nn_{\bar{1}}^{-})\bigl(\ch(M(\odd)_{\bar{0}}) - \ch(M(\odd)_{\bar{1}})\bigr) = \ch(\CC_{-\rho_{\bar{1}}}).
 \]
\end{lemma}

\begin{proof} By Proposition~\ref{EvenRestriction}, we have the following isomorphism of $\kk^{\CC}$-modules:
$$
\bigwedge\nolimits \nn_{\bar{1}}^{-} \otimes M(\odd) \cong \bigwedge\nolimits \nn_{\bar{1}}^{-} \otimes S(\nn_{\bar{1}}^{-}) \otimes \CC_{-\rho_{\bar{1}}}. 
$$ 
In particular, $\bigwedge\nolimits \nn_{\bar{1}}^{-} \otimes S(\nn_{\bar{1}}^{-})$ is the complex introduced above with $\kk^{\CC}$-invariant boundary operator $\delta$. 
By the Euler--Poincaré principle, 
$$
\ch(\bigwedge\nolimits \nn^{-}_{\bar{1}})\bigl(\ch(M(\odd)_{\bar{0}}) - \ch(M(\odd)_{\bar{1}})\bigr)
$$ 
is the Euler characteristic of this complex. However, by Lemma~\ref{lemm::Koszul_complex}, the cohomology is generated by $1\otimes 1$. The claim therefore follows from Proposition~\ref{EvenRestriction}, since $1\otimes 1$ is $\kk^{\CC}$-invariant.
\end{proof}

\begin{theorem} \label{thm::formal_character_Dirac_index}
 Let $F^{\nu}$ denote a simple $\kk^{\CC}$-module of highest weight $\nu \in \hh^{\ast}$. Define $N(\mu)\coloneqq \bigwedge\nolimits \nn^{-}_{\bar{1}} \otimes F^{\mu}$, and assume $\DC^{+}(\HH) = \sum_{\mu}F^{\mu}$ and $\DC^{-}(\HH) = \sum_{\nu}F^{\nu}$. Then 
 $$
 \ch(\HH) = \sum_{\mu}\ch(N(\mu+\rho_{\bar{1}})) - \sum_{\nu} \ch(N(\nu+\rho_{\bar{1}})).
 $$
\end{theorem}

\begin{proof}
Combining Lemma~\ref{lemm::char_minus_rho_odd}, the definition of $N(\mu)$ and the definition of the Dirac index, we have
\begin{equation*}
\begin{aligned}
 \ch(\HH) &\bigl(\ch(M(\odd)_{\bar{0}}) - \ch(M(\odd)_{\bar{1}})\bigr) \\ 
 &= \ch(\DC^{+}(\HH)) - \ch(\DC^{-}(\HH)) \\ 
 &= \sum_{\mu}\ch(F^{\mu}) - \sum_{\nu}\ch(F^{\nu}) \\ 
 &= \bigl(\sum_{\mu}\ch(N(\mu+\rho_{\bar{1}})) - \sum_{\nu}\ch(N(\nu+\rho_{\bar{1}}))\bigr)\bigl(\ch(M(\odd)_{\bar{0}}) - \ch(M(\odd)_{\bar{1}})\bigr),
 \end{aligned}
\end{equation*}
which can be rewritten as 
\[
\bigl(\ch(\HH) - \sum_{\mu}\ch(N(\mu+\rho_{\bar{1}})) + \sum_{\nu}\ch(N(\nu+\rho_{\bar{1}}))\bigr)\bigl(\ch(M(\odd)_{\bar{0}}) - \ch(M(\odd)_{\bar{1}})\bigr) = 0.
\]
We claim that the first factor must be trivial. Assume that it is non-trivial, \emph{i.e.},
\[
\ch(V) = \sum_{i=1}^{\infty} n_{i} \ch(F^{\mu_{i}}) \neq 0, \qquad V\coloneqq \HH - \sum_{\mu} N(\mu+\rho_{\bar{1}}) + \sum_{\nu}N(\nu + \rho_{\bar{1}}).
\]
Assume that $V$ contains a $\kk^{\CC}$-type $F^{\xi}$. Then, $\xi = \Lambda - \sum_{j}\beta_{j} - \sum_{i}\alpha_{i}$ for positive non-compact roots $\beta_{j}$ and positive odd roots $\alpha_{i}$. This follows by Theorem~\ref{thm::even_constituents_HW} and \cite[Proposition 3.6]{Jakobsen_Hermitian}. The $\beta_{j}$ are not present in $N(\Lambda)$ by construction. 

We consider the Weyl vector $\rho_{nc}$ associated to the set of non-compact positive roots (\emph{cf.} Section~\ref{subsubsec::Real_form}). Recall that the non-compact positive roots are $\epsilon_{k}-\epsilon_{l}$ for $1\leq k \leq p$ and $p+1\leq l \leq m$, while the odd positive roots are $\{\epsilon_{k}-\delta_{r}, -\epsilon_{l} +\delta_{s} : 1 \leq r,s \leq n, \ 1 \leq k \leq p, \ p+1\leq l \leq m\}$. Then, a direct calculation yields
\[
(\beta_{j},\rho_{nc}) > 0, \qquad (\alpha_{i},\rho_{nc}) > 0, \qquad \forall i,j.
\]
We conclude $(\xi,\rho_{nc}) \leq (\Lambda,\rho_{nc})$.

Without loss of generality, we can assume that $n_{1} \neq 0$ and $(\mu_{1}, \rho_{nc}) \geq (\mu_{i},\rho_{nc})$ for all $i$. 
Since $F^{\mu_{1}}\otimes M(\odd)$ contains $F^{\mu_{1}}\otimes 1 \cong F^{\mu_{1}-\rho_{\bar{1}}}$ with multiplicity one, the character of $F^{\mu_{1}-\rho_{\bar{1}}}$ appears in $\ch(F^{\mu_{1}})(\ch(M(\odd)_{\bar{0}})-\ch(M(\odd)_{\bar{1}}))$ with coefficient one.

By assumption, the contribution of $\ch(F^{\mu_{1}})(\ch(M(\odd)_{\bar0})-\ch(M(\odd)_{\bar1}))$ must be canceled by the remaining terms in $\ch(\nn_{\bar1}^{-})(\ch(M(\odd)_{\bar0})-\ch(M(\odd)_{\bar1}))$. Hence $F^{\mu_{1}+\rho_{\bar1}}$ must occur in $\ch(F^{\mu_i})(\ch(M(\odd)_{\bar0})-\ch(M(\odd)_{\bar1}))$ for some $i>1$, equivalently, in $F^{\mu_i}\otimes M(\odd)$ for some $i>1$.

The weights of $M(\odd)$ are of the form $-\rho_{\bar{1}} -\sum_{j}\beta_{j} - \sum_{k}\alpha_{k}$, where $\beta_{j}$ are distinct non-compact positive roots and $\alpha_{k}$ are distinct positive odd roots. We conclude that
$$
\mu_{1} = \mu_{i} - \sum_{j}\beta_{j} - \sum_{k}\alpha_{k}.
$$
This leads to a contradiction, as it would imply $(\mu_{1},\rho_{nc}) < (\mu_{i},\rho_{nc})$ since $(\beta_{j},\rho_{nc})>0$ and $(\alpha_{i},\rho_{nc})>0$ for all $i,j$. This finishes the proof.
\end{proof}

\thispagestyle{empty}


\bibliography{literatur}
\bibliographystyle{alpha}
\end{document}